\documentclass[11pt]{article}

\usepackage{setspace}
\def\spacingset#1{\renewcommand{\baselinestretch}%
{#1}\small\normalsize} \spacingset{1}
\spacingset{1.5}

\newcommand{\tightdisplayskip}{%
  \setlength{\abovedisplayskip}{6pt plus 2pt minus 2pt}%
  \setlength{\belowdisplayskip}{6pt plus 2pt minus 2pt}%
}

\AtBeginEnvironment{equation}{\tightdisplayskip}
\AtBeginEnvironment{equation*}{\tightdisplayskip}

\newlength{\originaljot}
\AtBeginEnvironment{aligned}{%
  \setlength{\originaljot}{\jot}%
  \setlength{\jot}{2pt}%
}
\AtEndEnvironment{aligned}{%
  \setlength{\jot}{\originaljot}%
}

\usepackage{indentfirst,csquotes}
\usepackage{natbib}
\usepackage{color}
\usepackage{hyperref}

\usepackage{amsthm,amsmath,amssymb}
\usepackage{xcolor,paralist,hyperref,titlesec,fancyhdr,etoolbox}

\newtheoremstyle{myremark}% <name>
  {.5\baselineskip}% Space above
  {.5\baselineskip}% Space below
  {}% Body font
  {}% Indent amount
  {\itshape}% Theorem head font
  {.}% Punctuation after theorem head
  {.5em}% Space after theorem head
  {}% Theorem head spec

\theoremstyle{myremark}
\newtheorem{theorem}{Theorem}[]
\newtheorem{definition}{Definition}
\newtheorem{example}{Example}
\newtheorem{remark}{Remark}
\newtheorem{lemma}{Lemma}[]
\newtheorem{proposition}[theorem]{Proposition}
\newtheorem{corollary}[theorem]{Corollary}
\newtheorem{assumption}{Assumption}[]

\usepackage{authblk} % 
\newcommand{\keywords}[1]{\textbf{Keywords:} #1}

\usepackage{algorithm,algorithmic} % code
\usepackage{booktabs}
\usepackage{multirow}
\hypersetup{
    colorlinks=true,  
    linkcolor=blue,   
    citecolor=blue,   
    urlcolor=magenta, 
    pdfborder={0 0 0} 
}
\usepackage[table]{xcolor}
\usepackage{array}
\usepackage{makecell} 
\usepackage{multirow}

\definecolor{lightgreen}{RGB}{204, 255, 204}
\definecolor{lightpink}{RGB}{255, 204, 204}
\definecolor{lightyellow}{RGB}{255, 255, 204}

\newcommand{\pr}{{\rm pr}}
\newcommand{\sgn}{{\rm sgn}}

\newcommand{\fdr}{{\rm FDR}}
\newcommand{\fdp}{{\rm FDP}}
\newcommand{\fcr}{{\rm FCR}}

\newcommand{\fsr}{{\rm FSR}}

\newcommand{\ca}{{\rm \scriptscriptstyle cal}}
\newcommand{\tr}{{\rm \scriptscriptstyle tr}}
\newcommand{\te}{{\rm \scriptscriptstyle te}}
\newcommand{\mix}{{\rm \scriptscriptstyle mix}}
\newcommand{\cp}{{\rm \scriptscriptstyle CP}} % conformal prediction
\newcommand{\oc}{{\rm \scriptscriptstyle OC}} % conformal prediction
\newcommand{\jomi}{{\rm \scriptscriptstyle JOMI}} % JOint Mondrian Conformal Inference
\newcommand{\isp}{{\rm \scriptscriptstyle ISP}} % Informative selective conformal prediction
\newcommand{\ora}{{\rm \scriptscriptstyle ora}} % optimal

\newcommand{\ps}{{\rm \scriptscriptstyle GC}} % Pseudo
\newcommand{\jc}{{\rm \scriptscriptstyle JC}} % Jin and Candes (2023) selection by prediction
\newcommand{\mfcr}{{\rm mFCR}} % marginal FCR

\newcommand{\cpow}{{\rm cPOW}}

\newcommand{\bh}{{\rm \scriptscriptstyle bh}}
\newcommand{\pshm}{{\rm \scriptscriptstyle GC,hm}}

\newcommand{\cphm}{{\rm \scriptscriptstyle CP,hm}}
\newcommand{\homo}{{\rm \scriptscriptstyle hm}}

\usepackage{xspace, mathtools}
\newcommand{\infosp}{\texttt{InfoSP}\xspace}
\newcommand{\infoscop}{\texttt{InfoSCOP}\xspace}
\newcommand{\infospp}{\texttt{InfoSP+}\xspace}
\newcommand{\infosppp}{\texttt{InfoSP++}\xspace}
\newcommand{\cfbh}{\texttt{cfBH}\xspace}
\newcommand{\cfbhp}{\texttt{cfBH+}\xspace}
\newcommand{\cfbhpp}{\texttt{cfBH++}\xspace}
\newcommand{\byfcr}{\texttt{BY-FCR}\xspace}
\newcommand{\bhfdr}{\texttt{BH}\xspace}
\newcommand{\scip}{\texttt{SCIP}\xspace}

\newcommand{\E}{\mathbb{E}}
\newcommand{\bigmid}{\bigm|}
\newcommand{\rmd}{{\rm d}}

\newcommand{\cC}{\mathcal{C}}
\newcommand{\cD}{\mathcal{D}}

\newcommand{\cH}{\mathcal{H}}
\newcommand{\cI}{\mathcal{I}}

\newcommand{\cL}{\mathcal{L}}

\newcommand{\cR}{\mathcal{R}}
\newcommand{\cS}{\mathcal{S}}

\newcommand{\cX}{\mathcal{X}}
\newcommand{\cY}{\mathcal{Y}}

\newcommand{\bp}{\mathbf{p}}
\newcommand{\bP}{\mathbf{P}}

\newcommand{\bbI}{\mathbb{I}}

\newcommand{\bbR}{\mathbb{R}}

\newcommand{\hk}{\hat{k}}
\newcommand{\hY}{\widehat{Y}}

\newcommand{\htau}{\hat{\tau}}

\newcommand{\hfdp}{\widehat{\fdp}}

\newcommand{\hcL}{\widehat{\cL}}
\newcommand{\hmu}{\widehat{\mu}}
\newcommand{\hpr}{\widehat{\pr}}
\newcommand{\hnu}{\hat{\nu}}

\newcommand{\tp}{\tilde{p}}
\newcommand{\tq}{\tilde{q}}

\newcommand{\tW}{\widetilde{W}}

\newcommand{\tT}{\widetilde{T}}

\newcommand{\tnu}{\tilde{\nu}}

\newcommand{\barp}{\bar{p}}
\newcommand{\bA}{\bar{A}}

\newcommand{\iid}{{i.i.d.\@ }}

\usepackage{pifont}
\newcommand{\yesgreen}{{\cellcolor{lightgreen} \ding{51}}}
\newcommand{\nonopink}{{\cellcolor{lightpink} \ding{55}}}
\newcommand{\partyelw}{{\cellcolor{lightyellow}Partially}}

\begin{document}
\title{Conformalized Large-Scale Selective Inference with Informative and Trustworthy Prediction Sets}

\author[1]{Wangcheng Li}
\author[2]{Guanlan Zhao}
\author[1]{Xu Guo}
\author[2]{Wenguang Sun
\footnote{Email for correspondence: \href{Email:wgsun@zju.edu.cn}{wgsun@zju.edu.cn}}}

\affil[1]{School of Statistics, Beijing Normal University}

\affil[2]{Center for Data Science, Zhejiang University}

\maketitle

\begin{abstract}
    In large-scale prediction problems, exhaustively following up on all test units is often impractical and inefficient, motivating a selective reporting strategy that fulfills the dual requirements of informativeness and trustworthiness. Within the InfoFCR (Informative prediction with False Coverage Rate control) framework, we propose \scip (Selective Conformal Inference for Informative Predictions), a procedure built on three key components: (i) an \emph{informative set constructor} that tailors prediction sets to individual test units according to user‑specified informativeness constraints; (ii) a \emph{trust score} that provides a principled quantification of the trustworthiness of candidate informative sets; and (iii) \emph{generalized conformal p‑values} that are used to perform FCR analysis for selecting the most promising candidates. We establish that SCIP guarantees finite‑sample FCR control and is asymptotically anti‑conservative, achieving higher statistical power than existing methods. The framework is highly versatile, accommodating a wide range of error metrics across both regression and classification tasks. Extensive numerical experiments on simulated and real data demonstrate the effectiveness of our approach.
\end{abstract}

\newpage

\section{Introduction}

In large-scale predictive inference, particularly in genomics and drug discovery, researchers routinely screen thousands of candidates, rendering exhaustive reporting and follow-up on every test unit impractical and inefficient. To extract meaningful insights from such massive datasets, an effective strategy is to selectively identify candidates that are both scientifically relevant and statistically reliable, ensuring that limited resources are efficiently devoted to actionable and high-impact candidates. This selective prediction paradigm has recently attracted considerable attention \citep{jin2023selection, bao2024selective, gazin2025selecting, jin2025confidence}. Our objective is to develop conformal inference methods \citep{vovk2005algorithmic} that jointly guarantee (i) scientific relevance, by imposing strict informativeness criteria on every reported prediction set, and (ii) statistical trustworthiness, by controlling the false coverage rate (FCR; \citealp{benjamini2005false}) at a pre‑specified level.

\subsection{The InfoFCR framework}

We first introduce the InfoFCR framework \citep{gazin2025selecting}, which formalizes FCR control under informativeness constraints.  
Consider a random pair $(X,Y) \in \mathbb{R}^d \times \mathcal{Y}$ drawn from an unknown distribution $P_{XY}$, where $X$ is the feature vector and $Y$ is the associated response. We set $\mathcal{Y} = \mathbb{R}$ for regression and $\mathcal{Y} = [K]$ with $K \ge 2$ for classification. The labeled data are split into a training set $\mathcal{D}^{\mathrm{tr}}$ and a calibration set $\{(X_i, Y_i)\}_{i=1}^n$. Our objective is to predict the unknown labels in a test set comprising $m$ feature vectors $\{X_{n+j}\}_{j=1}^{m}$.  

In large-scale studies, we employ a \emph{selective prediction procedure}, formalized as
$$
\cR = \{\cC_{n+j} : j \in \cS\},
$$
where $\cS \subseteq [m]$ denotes the set of selected indices, and $\cC_{n+j}$ is a prediction set---an interval in regression or a subset of class indices in classification.

Under the InfoFCR framework, inference is guided by two key desiderata: scientific relevance and statistical trustworthiness. To ensure scientific relevance, every reported prediction set must satisfy a user-specified \emph{informativeness constraint} as $\cC_{n+j} \in \mathcal{I}$, where $\mathcal{I}$ is a context-specific collection of admissible sets. Meanwhile, to guarantee statistical trustworthiness, we propose to control the FCR at a pre‑specified level $\alpha \in (0,1)$. The FCR is defined as the expectation of the false coverage proportion (FCP): 
\begin{equation}\label{def:FCR}
\text{FCP}(\cS) = \frac{\sum_{j \in \cS} \mathbb{I}(Y_{n+j} \notin \cC_{n+j})}{\max(1, |\cS|)}, \qquad 
\text{FCR} = \mathbb{E}[\text{FCP}(\cS)],
\end{equation}
where $\mathbb{I}(\cdot)$ is the indicator function, and $|\cS|$ denotes the cardinality of $\cS$.

By integrating scientific relevance with statistical trustworthiness, InfoFCR provides a flexible and unified framework that encompasses diverse application scenarios and naturally recovers various error metrics as special cases. We illustrate this versatility with concrete examples in Section~\ref{sec:infofcrframework}.

\subsection{An illustrative example}\label{sec:illustrativeexample}

Consider a classification task on the CIFAR-10 dataset, where the goal is to predict the labels of images belonging to three categories: Vehicle (`1'), Aircraft (`2'), and Bird (`3').
Two desiderata in our InfoFCR analysis include:
(a) \emph{informativeness}, requiring that each reported prediction set contains at most two classes (since returning a trivial prediction set with all classes provides little insight for decision-making), and (b) \emph{trustworthiness}, requiring that the FCR is controlled below $0.1$. Implementation details are presented in Appendix~\ref{assec:detailillustrativeexample}.

\begin{figure}[!t]
    \centering
    \includegraphics[width=\linewidth]{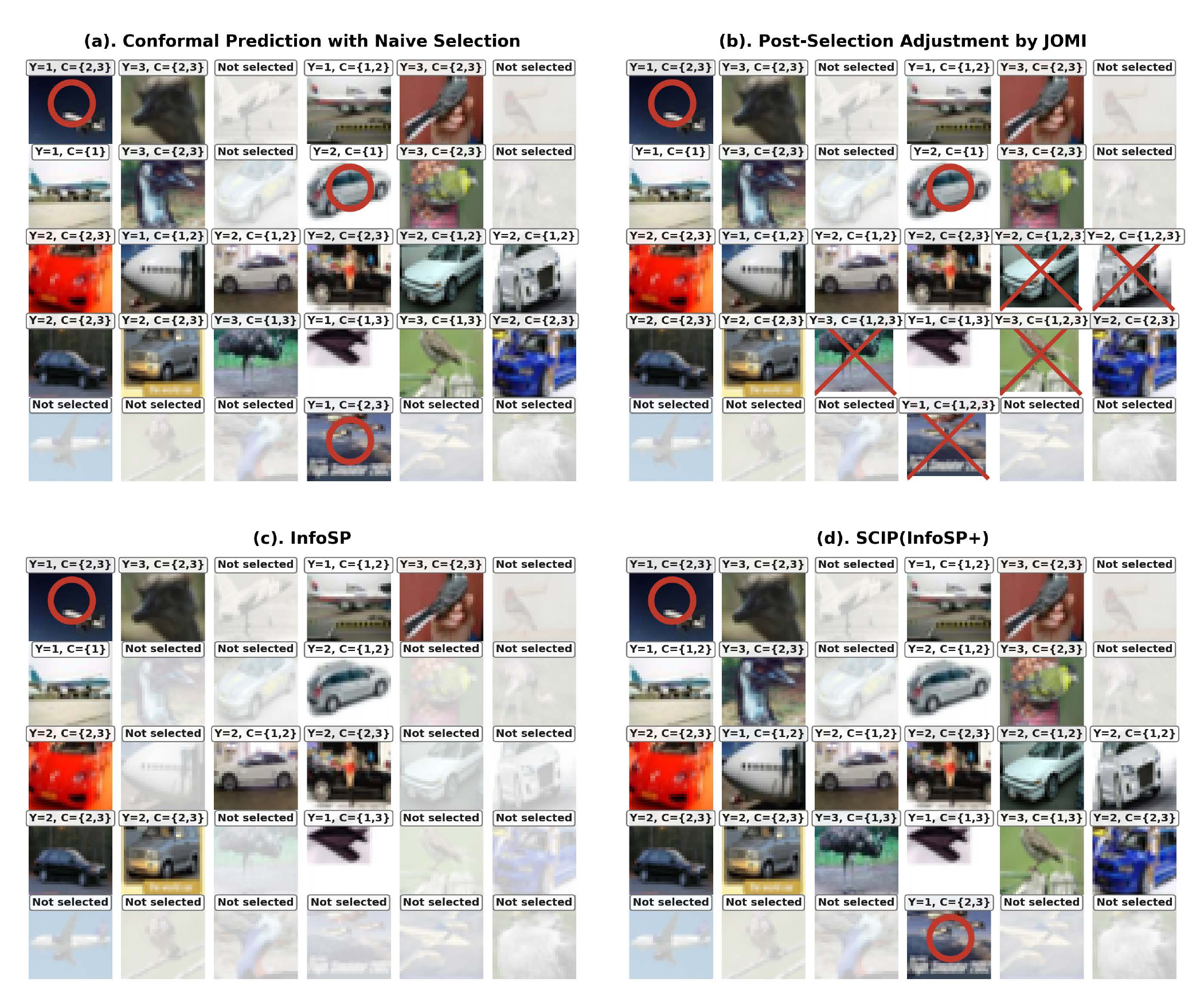}
    \caption{
        \small
        Three-class selective classification with size-two constraint. 
        (a): standard conformal prediction with uninformative sets excluded. 
        (b): \texttt{JOMI} adjustment. 
        (c): \infosp procedure.
        (d): \infospp, an instantiation of \scip.
        Red circles (`${\color{red} \bigcirc}$') indicate noncoverage; red crosses (`${\color{red} \bigtimes}$') indicate non-informative sets.
        }
    \label{fig:cifar10_onlyab}
\end{figure}

Figure~\ref{fig:cifar10_onlyab}(a) shows standard conformal prediction sets constructed to achieve $90\%$ marginal coverage. Among the 30 prediction sets, 3 fail to cover the true responses, leading to an FCP level of $0.1$%
\footnote{For visualization purposes, we use the FCP as a proxy for the FCR. The general trends and conclusions revealed here are rigorously corroborated by simulation results presented in Section~\ref{sec:numericexperiment}.}.
To ensure informativeness, 9 non-informative sets of size 3 are excluded, reducing the number of reported sets to 21. Unfortunately, this operation leads to an inflated FCP level of $0.143$ (3 errors among 21 selections).
To restore trustworthiness, we apply the \texttt{JOMI} method \citep{jin2025confidence} to correct for selection bias [Figure~\ref{fig:cifar10_onlyab}(b)].
While this correction reduces the FCP below $0.1$ (2 errors among 21 sets), 5 of the adjusted sets are no longer informative.

Panels (a) and (b) in Figure~\ref{fig:cifar10_onlyab} collectively reveal a fundamental dilemma for the InfoFCR task: unadjusted selection preserves informativeness at the expense of trustworthiness, whereas adjusted selection ensures trustworthiness at the cost of informativeness. To address this dilemma, \citet{gazin2025selecting} proposed an innovative and principled approach, termed \infosp, which simultaneously satisfies both desiderata. As shown in Figure~\ref{fig:cifar10_onlyab}(c), \infosp produces 12 informative sets with 1 noncovering, resulting in an FCP of $0.083$. Nevertheless, \infosp is inherently conservative, reporting fewer informative sets than potentially achievable.

This paper develops \scip (Selective Conformal Inference for Informative Predictions, pronounced ``skip''), a method designed to satisfy the dual requirements of informativeness and trustworthiness in selective predictions. As illustrated in Figure~\ref{fig:cifar10_onlyab}(d), \scip reports 21 informative sets with 2 noncoverings. We establish the validity and anti-conservativeness of \scip in Section~\ref{ssec:prdsproperty} and demonstrate that \infospp, as a concrete instantiation of \scip, improves \infosp; this improvement is analyzed theoretically in Section~\ref{ssec:proposedinfospp} and confirmed empirically in Section~\ref{sec:numericexperiment}.

\subsection{Preview of \scip and contributions}

The \scip method, presented in Section~\ref{sec:methodology}, comprises three main steps. First, an \emph{informative set constructor} maps each test feature $X_{n+j}$ to an informative prediction set $\cC_{n+j}$ admissible for reporting. Second, a \emph{trust score} is computed for each $\cC_{n+j}$ to quantify its reliability, with higher scores indicating a greater likelihood of coverage. Finally, a \emph{generalized conformal $p$-value} is computed for each $\cC_{n+j}$, and the Benjamini–Hochberg (\bhfdr) procedure is applied to select informative sets. The theoretical properties of \scip are investigated in Section~\ref{ssec:prdsproperty}. Section~\ref{sec:applicationofinfospp} illustrates \scip instantiations across various tasks, including conformal selection, selective informative prediction, and diversified selection. Section~\ref{sec:trust-score-cons} describes the derivation of task-specific trust scores, further enhancing the efficiency of \scip. Finally, Section~\ref{sec:numericexperiment} presents numerical results evaluating the empirical performance of \scip.

Our work makes several contributions to selective conformal inference. First, the \scip framework provides an alternative perspective on the fundamental dilemma previously identified and addressed by \infosp \citep{gazin2025selecting}. We integrate the dual requirements of informativeness and trustworthiness into selective informative prediction through the new concepts of trust scores and generalized conformal $p$-values. This reframing leads to the \scip method, which we show to guarantee finite‑sample FCR control and to be asymptotically anti‑conservative (Section~\ref{ssec:prdsproperty}).

Second, this new perspective complements and enhances existing works. In Section~\ref{ssec:proposedinfospp}, we show that \infospp, a specific instantiation of \scip, can be seen as an enhancement of the elegant \infosp that enables more efficient use of the FCR budget. This analysis clarifies why the intrinsic conservativeness of \infosp can be mitigated within the \scip framework. Furthermore, Section~\ref{ssec:diversifiedselection} demonstrates how the \scip framework strengthens the diversified selection theory \citep{wu2024optimal} by achieving finite‑sample guarantees.

Finally, the \scip framework provides a flexible yet solid foundation for future methodological innovation. Our initial exploration in Sections~\ref{sec:applicationofinfospp} and~\ref{sec:trust-score-cons} suggests that existing methods relying on fixed, pre‑trained nonconformity scores can be improved within the \scip framework by constructing data‑adaptive informative sets and more powerful trust scores, thereby further enhancing the efficiency, reliability, and interpretability of selective inference in practice.

\subsection{Related work}
Our work connects several key lines of research, including FCR analysis \citep{benjamini2005false,weinstein2013selection}, conformal selection \citep{jin2023selection,jin2023modelfree}, outlier detection \citep{bates2023testing,marandon2024adaptive}, semi-supervised multiple testing \citep{mary2022semisupervised}, selective classification \citep{rava2021burden,zhao2023controlling}, and selective conformal inference \citep{bao2024selective,jin2025confidence}. 
Due to page limitations, we briefly discuss related work in the main text and defer a detailed comparison to Appendix~\ref{asec:literaturereview}.

\section{Preliminaries: the InfoFCR Framework and Examples}
\label{sec:infofcrframework}

The InfoFCR framework \citep{gazin2025selecting} effectively integrates informativeness constraints with FCR control, providing a unified approach for large-scale selective inference and an important avenue for methodological innovation. To provide the necessary background for subsequent developments, we illustrate the versatility of this framework through several concrete examples, demonstrating how a range of classical error metrics naturally arise as special cases within this coherent formulation. 

\begin{example}[\textbf{Outlier Detection}, motivated by \citet{guan2022prediction}]
    \label{exa:outlierdetection}
    Consider a classification setting with $K$ null classes, where the objective is to simultaneously predict class labels and detect outliers. Let $(\delta_{j})_{j=1}^m \in \{0,1\}^m$ denote a decision rule, where $\delta_{j} = 0$ indicates assignment to a null class and $\delta_{j} = 1$ denotes outlier abstention (i.e., abstaining from assigning the test point to any null class). Two relevant error rates are the false discovery rate (FDR; \citealp{benjamini1995controlling}) and the false labeling rate (FLR; \citealp{guan2022prediction}): 
    \small
    \begin{equation}\label{def:FDR}
        \fdr = 
        \E \Bigg\{
            \frac{
                \sum_{j = 1}^m \bbI(Y_{n+j} \in [K],\delta_j = 1)
            }{
                1 \vee \sum_{j = 1}^m \bbI(\delta_j = 1)
            }
        \Bigg\},
        \;\;
        \text{FLR} = 
        \E \Bigg\{
            \frac{
                \sum_{j = 1}^m \bbI(Y_{n+j} \notin [K],\delta_j = 0)
            }{
                1 \vee \sum_{j = 1}^m \bbI(\delta_j = 0)
            }
        \Bigg\}.
    \end{equation}
    \normalsize 
    Define $(\cC_{n+j})_{j=1}^m \in \{\emptyset, [K]\}^m$, with $\cC_{n+j} = [K]$ ($\cC_{n+j} = \emptyset$) corresponding to assignment to null class (outlier). The InfoFCR framework accommodates both the FDR and FLR with suitable informativeness constraints. 
    \begin{itemize}
        \item[(a)] 
        Let $\cI = \{\emptyset\}$. Output $\cC_{n+j} = \emptyset$ if and only if (iff) $\delta_j = 1$, and define $\cS \coloneqq \{j \in [m]: \delta_j = 1\}$. 
        Then, the FCR defined in \eqref{def:FCR} is equivalent to the FDR in \eqref{def:FDR}.
    
        \item[(b)] Let $\cI = \{[K]\}$ and output $\cC_{n+j} = [K]$ iff $\delta_j = 0$. Define $\cS \coloneqq \{j \in [m]: \delta_j = 0\}$. Then, the FCR in \eqref{def:FCR} recovers the FLR in \eqref{def:FDR}.
    \end{itemize}
\end{example}

\begin{example}[\textbf{Conformal Selection}, motivated by \cite{jin2023selection}]
    \label{exa:conformalselection}
    Consider a regression-based conformal selection task, where a test unit $j$ is deemed ``interesting'' if its response $Y_{n+j}$ lies outside a pre-specified indifference region. Two common forms of indifference regions are considered: $(-\infty, c_0]$ for one-sided alternatives and $[c_l, c_u]$ for two-sided alternatives. 
    In practice, one-sided alternatives are typically used when interest focuses on extreme responses, such as identifying top-performing candidates in drug discovery \citep{laghuvarapu2023codrug} and candidate screening \citep{shehu2016adaptive}.
    Two-sided alternatives are appropriate when deviations in either direction are relevant, for instance, detecting anomalies in both directions in demand prediction \citep{levitskaya2023how}. 

    For one-sided alternatives, the decision rule is $(\delta_j)_{j=1}^m \in \{0,1\}^m$, where $\delta_j = 1$ indicates rejection of the null $(-\infty, c_0]$ and $\delta_j = 0$ otherwise. For two-sided alternatives, the directional rule extends to $(\delta_j^{\text{d}})_{j=1}^m \in \{-1,0,1\}^m$, where $\delta_j^{\text{d}} = -1$ ($\delta_j^{\text{d}} = 1$) corresponds to a significantly negative (positive) effect, while $\delta_j^{\text{d}} = 0$ denotes absence of evidence for deviation from the indifference region.  

 Let $(\theta_j)_{j=1}^m\in\{-1,0,1\}^m$ denote the true states, with
    $\theta_j = -1$ if $Y_{n+j} < c_l$, $\theta_j = 1$ if $Y_{n+j} > c_u$, and $\theta_j = 0$ otherwise. To control decision errors, we adopt the FDR for one-sided tests and the directional FDR (dFDR; \citealp{shaffer2002multiplicity,benjamini2005false}) for two-sided tests:
    \small 
    \begin{equation}\label{def:FDRs} 
        \fdr = \E \Bigg\{ 
            \frac{
                \sum_{j = 1}^m \bbI (Y_{n+j} \le c_0, \delta_j = 1)
            }{
                1 \vee \sum_{j = 1}^m \bbI (\delta_j = 1)
            }
        \Bigg\}, 
        \text{dFDR} = \E \Bigg\{ 
            \frac{
                \sum_{j = 1}^m \bbI (\delta_j^{\rm d} \ne \theta_j, \delta_j^{\rm d} \ne 0)
            }{
                1 \vee 
                \sum_{j = 1}^m \bbI (\delta_j^{\rm d} \ne 0)
            }
        \Bigg\}, 
    \end{equation}
    \normalsize

    Within the InfoFCR framework:
    \begin{itemize}
        \item[(a)] Let $\cI = \{(c_0, \infty)\}$, and set $\cC_{n+j} = (c_0, \infty)$ iff $\delta_j = 1$. Define $\cS=\{j\in [m]: \delta_j = 1\}$.
        Then, the resulting FCR reduces to the FDR in \eqref{def:FDRs}.
        \item[(b)] Let $\cI = \{ (-\infty, c_l), (c_u, \infty)\}$. 
        Set $\cC_{n+j}=(-\infty, c_l)$ iff $\delta_j^d=-1$ and $\cC_{n+j} = (c_u, \infty)$ iff $\delta_j^d=1$.
        Defining $\cS = \{j\in[m]: \delta_j^d\neq 0\}$, the FCR recovers the dFDR in \eqref{def:FDRs}.
    \end{itemize}
\end{example}

\begin{example}[\textbf{Selective Classification}, motivated by \cite{gazin2025selecting}]
    \label{exa:selectiveclassification}
    Consider two closely related large-scale classification problems.

    The first task involves identifying test units belonging to a pre-specified class $y_0 \in [K]$. For instance, in credit risk assessment, a bank may wish to identify applicants who are likely to default (class $y_0$) \citep{ahmed2016survey}. The decision rule is $\{\hY_{n+j} \in \{y_0, 0\} : j \in [m]\}$, where $\hY_{n+j} = 0$ indicates classification abstention. To aggregate errors, the false selection rate (FSR) is adopted \citep{rava2021burden}:
    \begin{equation}\label{def:FSR-k0}
        \fsr (\{y_0\}) = 
        \E \Bigg\{ 
            \frac{
                \sum_{j = 1}^m \bbI (Y_{n+j} \ne y_0, \hY_{n+j} = y_0)
            }{
                1 \vee \sum_{j = 1}^m \bbI (\hY_{n+j} = y_0)
            }
        \Bigg\}.
    \end{equation} 

    The second, more general task extends selection to all $K$ classes. 
    For example, researchers may wish to determine the disease status of a patient and withhold a risky decision when the evidence is insufficient \citep{olsson2022estimating}. 
    The decision rule becomes $\{\hY_{n+j} \in \{0, 1, \ldots, K\} : j \in [m]\}$, where $\hY_{n+j}=k \in [K]$ indicates that the patient is predicted to be status $k$, and $\hY_{n+j}=0$ denotes that existing evidence fails to support a decision (for any status). The corresponding FSR is
    \begin{equation}\label{def:FSR-K}
        \fsr([K]) = 
        \E \Bigg\{
            \frac{
                \sum_{j=1}^m \bbI (Y_{n+j} \ne \hY_{n+j}, \hY_{n+j} \neq 0)
            }{
                1 \vee \sum_{j = 1}^m \bbI (\hY_{n+j} \neq 0)
            } 
        \Bigg\}.
    \end{equation}
    
    We reformulate the above problems within the InfoFCR framework:
    \begin{itemize}
        \item[(a)] Let $\cI = \{\{y_0\}\}$ and set $\cC_{n+j} = \{y_0\}$ iff $\hY_{n+j} = y_0$. With $\cS = \{j \in [m] : \hY_{n+j} = y_0\}$, the FCR reduces to $\fsr (\{y_0\})$.
        \item[(b)] For selection across all classes, we set $\cI = \{C \subset [K] : |C| = 1\}$, requiring each prediction set to be a singleton. 
        Defining $\cC_{n+j} = \{k\}$ iff $\hY_{n+j} = k$ for $k \in [K]$, and $\cS = \{j \in [m] : \hY_{n+j} \ne 0\}$, the resulting FCR exactly recovers $\fsr([K])$.
    \end{itemize} 
\end{example}

Finally, beyond the error metrics above, the InfoFCR framework accommodates a rich, more complex class of informativeness constraints. In regression, such constraints include: (a) length-restricted prediction intervals; (b) intervals that exclude uninteresting null values; (c) unions of pre-specified informative intervals. In classification, constraints may include: (a) prediction sets of bounded size; (b) sets that exclude scientifically irrelevant categories. The FCR control under such constraints falls outside the scope of standard error metrics yet arises naturally in many practical applications. We will revisit these scenarios in later methodology developments and numerical experiments.
 
\section{The \scip Procedure}\label{sec:methodology}

We now introduce the \scip procedure for the InfoFCR task. Before diving into the technical details, we provide an overview of the \scip procedure, which relies on three innovative concepts:
\begin{itemize}
    \item An \emph{informative set constructor} (Section~\ref{ssec:isc}) that tailors user-specified informativeness constraints to generate adaptive, data-driven prediction sets for individual units. 

    \item A \emph{trust score} (Section~\ref{subsec:trust-score}) assigned to each informative set, quantifying its trustworthiness and serving as a building block for subsequent inference. 

    \item \emph{Generalized conformal $p$-values} (Section~\ref{subsec:gcp}), computed from the trust scores, which are used to rank and select informative prediction sets. 
\end{itemize}

\subsection{Informative set constructor}\label{ssec:isc}

The {informative set constructor} is a mapping $\cC^{\cI} : \cX \to \cI$ that assigns each feature vector to a prediction set admissible for reporting. We require this mapping to be \emph{permutation-invariant} with respect to the combined sample $\{Z_i := (X_i, Y_i)\}_{i=1}^{n+m}$, as formally defined below. 

\begin{definition}\label{def:rigorousper-inv}
    Let $G$ be a mapping indexed by $(Z_1, \ldots, Z_{n+m}, \mathbb{D})$, where $\mathbb{D}$ denotes auxiliary data (e.g., training or additional calibration data) used to construct $G$. Then $G$ is said to be \emph{permutation-invariant} with respect to $\{Z_i\}_{i=1}^{n+m}$ if: (a) for any fixed $\mathbb{D}$ and any permutation $\varsigma$ on $[n+m]$, 
    \begin{equation*}        
        G(\cdot\,; Z_1,\dots,Z_{n+m}, \mathbb{D}) = G(\cdot\,; Z_{\varsigma(1)},\dots,Z_{\varsigma(n+m)}, \mathbb{D});
    \end{equation*}
    and (b) If the random variables $\{Z_i\}_{i=1}^{n+m}$ are exchangeable, they remain exchangeable conditional on $\mathbb{D}$.
\end{definition}

The most natural approach to designing a valid $\cC^\cI$ is to pre-specify a constructor that directly reflects the inferential goal. For example, in one‑sided conformal selection with an indifference region $(-\infty, c_0]$ (Example~\ref{exa:conformalselection}-a), one may take $\cC^{\cI}(X) = (c_0, \infty)$. Similarly, in selective classification aimed at identifying units from a target class $y_0$ (Example~\ref{exa:selectiveclassification}-a), one may set $\cC^{\cI}(X) = \{y_0\}$. Such pre‑specified constructors clearly satisfy the permutation‑invariance condition.

However, pre-specifying an informative set constructor can be inefficient or infeasible in many applications. In such cases, \emph{data-driven} constructors can be designed to adapt to evidence in the test sample. Consider two-sided conformal selection (Example~\ref{exa:conformalselection}-b). Discarding directional information, the full set $(-\infty,c_l) \cup (c_u, \infty)$ does not fulfill the informativeness constraint $\cI = \{(-\infty, c_l), (c_u, \infty)\}$. One could resort to a random guess to narrow down to a single direction, but a more intuitive and effective approach is to tailor the informative set adaptively to each test unit based on its estimated signal direction. Specifically, let $\hpr$ denote the estimated probability of an event. We first form a \emph{preliminary directional decision}
\begin{equation}\label{def:delta-j}
    \delta_j = 
    \sgn \Big( 
        \hpr\big\{ 
            Y_{n+j} \in (c_u, \infty) \mid X_{n+j}
        \big\}
        - 
        \hpr\big\{ 
            Y_{n+j} \in (-\infty, c_l) \mid X_{n+j}
        \big\}
    \Big),
\end{equation}
and then let $\cC^{\cI}(X_{n+j}) = (c_u, \infty)$ if $\delta_j = 1$ and $\cC^{\cI}(X_{n+j}) = (-\infty, c_l)$ if $\delta_j = -1$. This strategy similarly applies to selective classification under a size-one constraint (Example~\ref{exa:selectiveclassification}-a), where a natural data-driven constructor is given by
$\cC^{\cI}(X) = \{\operatorname{argmax}_{k}\; \hpr(Y = k \mid X)\}$, in line with the proposal of \citet{zhao2023controlling}.
When the probability estimate is derived from an independent training dataset, both $\cC^{\cI}$ constructors satisfy the permutation-invariance condition.

The design of $\cC^\cI$ is context-specific and often depends on the complexity of $\cI$. When $\cI$ is a singleton, the pre‑specified approach suffices. For discrete $\cI$ with small cardinality, one may consider strategies based on estimated probabilities (cf.~\eqref{def:delta-j}). If $\cI$ is more complex (e.g. involving a large number of categories in classification or interval constraints in regression that cannot be easily reduced to a simple classification problem), one can adopt a general data‑driven construction based on $\cI$-adjusted $p$-values \citep{gazin2025selecting}; we elaborate on this approach and discuss a concrete instantiation of \scip in Section~\ref{ssec:proposedinfospp}.

\subsection{Trust score}\label{subsec:trust-score}

The \emph{trust score} is a mapping $T: \cX \times \cI \to \mathbb{R}^+$ that quantifies the reliability of an informative set associated with a given feature $X$. 
Higher trust scores indicate more trustworthy prediction sets (thus stronger candidates for selection).
Similar to the informative set constructor, the trust score $T(\cdot,\cC^\cI(\cdot))$ is required to be permutation‑invariant with respect to the combined sample $(Z_i)_{i=1}^{n+m}$ [c.f Definition~\ref{def:rigorousper-inv}].
We first present several illustrative examples, deferring the complex issue of optimal trust score selection to Section~\ref{sec:trust-score-cons}.

A simple way to construct a trust score (for a set $\cC$) is to tailor a monotone
non‑conformity score $V(x,y)$ (for a value $y$). For example, in regression with one-sided indifference region $(-\infty, c_0]$ (Example~\ref{exa:conformalselection}-a,  \citealp{jin2023selection}), we take $\cC^{\cI}(X) = (c_0, \infty)$. If $V(x,y)$ is monotone, the trust score can be defined as $T(x,\cC^{\cI}) = \exp\big[\inf \{-V(x,y) : y \notin (c_0, \infty)\}\big] = \exp\{-V(x,c_0)\}$.

For data-driven constructors, this approach may not be applicable as a suitable non-conformity score may not exist. Consider two-sided conformal selection (Example~\ref{exa:conformalselection}-b). Using the preliminary directional decision $\delta_j$ from~\eqref{def:delta-j}, we define
\begin{equation}\label{trust-two-sided} 
    T(X_{n+j},\cC^{\cI}) = 
    \begin{cases}
    \hpr \{Y_{n+j} \in (c_u, \infty) \mid X_{n+j}\}, & \delta_j = 1,\\
    \hpr \{Y_{n+j} \in (-\infty, c_l) \mid X_{n+j}\}, & \delta_j = -1.
    \end{cases}
\end{equation}
This trust score directly estimates the probability that $Y_{n+j}$ falls within the informative set. Similarly, for selective classification under a size-one constraint (Example~\ref{exa:selectiveclassification}-a), where a suitable non-conformity score is also unavailable, the estimated probability serves as a natural trust score: 
$T(X_{n+j},\cC^{\cI}) = \hpr(Y_{n+j} = \hY_{n+j} \mid X_{n+j}).$ For informative sets containing multiple classes, we define
$
T(X_{n+j}, \cC^\cI_{n+j}) = \sum_{k = 1}^K \bbI(k \in \cC^\cI_{n+j}) \cdot \hpr(Y_{n+j} = k \mid X_{n+j})
$.
In all the above cases, when the monotone score or probability estimate is derived from an independent training set, the trust score is permutation-invariant with respect to $(Z_i)_{i=1}^{n+m}$. 
 
\begin{remark}
The trust score differs from the non-conformity score in two fundamental respects:
(a) A non‑conformity score $V(x,y)$ measures the discrepancy between a hypothesized point $(X_{n+j},y)$ and the trend in the calibration data, effectively testing a simple null hypothesis of the form $Y_{n+j}=y$ \citep{angelopoulos2024theoretical}. In contrast, the trust score evaluates the reliability of an entire prediction set, targeting the composite null defined in~\eqref{eq:originalhypothesis}.
(b) The two scores differ in their benchmarks: the non‑conformity score compares the test point against a fixed trend (calibrated from labeled data) that is independent of the test point. By construction, however, the trust score is a function of the informative set, which may depend explicitly on the test data.
These two properties together make the trust score a more adaptive and suitable building block for informative selective prediction.  
\end{remark}

\subsection{Generalized conformal \textit{p}-value}\label{subsec:gcp}

We now turn to the problem of determining the selection set $\cS$. A foundational tool for uncertainty quantification in conformal inference is the conformal $p$-value, which arises naturally from a hypothesis-testing perspective. Motivated by this idea, we consider a multiple-testing problem induced by the informative sets:
\begin{equation}\label{eq:originalhypothesis}
    H_{0,j}: Y_{n+j} \notin \cC^\cI(X_{n+j})
    \quad \text{versus} \quad 
    H_{1,j}: Y_{n+j} \in \cC^\cI(X_{n+j}),
    \qquad j \in [m].
\end{equation}
We emphasize that \eqref{eq:originalhypothesis} is not a standard hypothesis-testing problem, in which the null hypothesis should be pre-specified before observing the data.
Instead, the null hypotheses in~\eqref{eq:originalhypothesis} are adaptive, depending on both the informativeness constraints and the test data. 
Nevertheless, this perspective provides a valuable conceptual device, inspiring a natural strategy: construct a quantity analogous to a conformal $p$-value using the trust scores 
$\{T_i \coloneqq T(X_i,\cC^{\cI}(X_i)) : i \in [n+m]\}$, just as standard conformal inference builds $p$-values from non-conformity scores. Specifically, for $j \in [m]$, we define
\begin{equation}\label{eq:pseudopvalue}
    p_j^{\ps} = 
    \frac{
        \sum_{i = 1}^n 
        \bbI \big( Y_i \notin \cC_i^\cI,\, T_i > T_{n+j} \big)
        + 
        \big\{ 
            1 + 
            \sum_{i = 1}^n 
            \bbI \bigl(Y_i \notin \cC_i^\cI,\, T_i = T_{n+j}\bigr)
        \bigr\}
        \cdot U_j
    }{1+n},
\end{equation}
where $\{U_j\}_{j=1}^m$ are independent $\mathrm{Unif}(0,1)$ variables used to break ties. 
We refer to $p_j^{\ps}$ as a \emph{generalized conformal $p$-value}, which is \emph{not} a $p$-value in the classical sense because (a) the hypotheses in \eqref{eq:originalhypothesis} are non-standard, and (b) $p_j^{\ps}$ does not generally follow a super-uniform distribution under the null. The theoretical properties of  $p_j^{\ps}$ are established in Section~\ref{ssec:prdsproperty}. 

Intuitively, if $\cC^{\cI}(X_{n+j})$ is trustworthy (i.e., associated with a high trust score), it receives a small generalized conformal $p$-value and is thus prioritized for selection. Consequently, the \scip procedure selects those sets whose generalized $p$-values fall below a threshold determined by a \emph{self-consistent} reformulation of the \bhfdr procedure \citep{blanchard2009adaptive}; see Appendix~\ref{subsec:self-cons} for details. Specifically, let 
\begin{equation}\label{eq:qbh}
    \cS = 
    \big\{ 
        j \in [m] : p_j^{\ps} \le \widehat{\alpha}
    \big\},
    \quad 
    \widehat{\alpha} = 
    \max \Bigg\{
        a \in [0,1] : 
        \frac{\alpha}{m} \sum_{j = 1}^m \bbI (p_j^{\ps} \le a) \ge a
    \Bigg\},
\end{equation}
with the convention $\max \emptyset = 0$. 
The final output of \scip is 
$
\cR^{\cI} = \{\cC_{n+j}^{\cI} : j \in \cS\}.
$
The procedure is summarized in Algorithm~\ref{alg:baseprocedure}.

\section{Theory}\label{ssec:prdsproperty}

Our theoretical analysis first establishes key properties of the generalized conformal $p$-values (Section~\ref{ssec:gcp-theory}) and then provides theoretical guarantees for \scip (Section~\ref{subsec:theory-scip}).

\subsection{Properties of generalized conformal \textit{p}-values}\label{ssec:gcp-theory}

We begin with a standard exchangeability condition in conformal inference.
\begin{assumption}\label{ass:fullexchangeability}
    The data points in $\cD^\mix = \{(X_i, Y_i): i \in [n+m]\}$ are exchangeable conditional on training data $\cD^\tr$.
\end{assumption}

Then, we establish the first key property of the generalized conformal $p$-values.

\begin{proposition}\label{the:pseudopvalueproperty}
    Under Assumption~\ref{ass:fullexchangeability}, and assuming that both the informative set constructor and the trust score function are permutation-invariant to $\{Z_i\}_{i = 1}^{n+m}$ [c.f Definition~\ref{def:rigorousper-inv}], we have, for any prescribed level $\alpha \in (0,1)$,
    \begin{equation}\label{eq:gs-unif}
        \pr\bigl(H_{0,j},\, p_j^{\ps} \le \alpha\bigr) \le \alpha.
    \end{equation}
\end{proposition}

\begin{remark}\label{rem:generalizedsuperuniformity}
    The property in \eqref{eq:gs-unif} is known as \emph{generalized super‑uniformity} \citep{jin2023selection}. It is particularly appealing because it not only provides finite-sample guarantees for FCR control, but also adapts to unknown sparsity in the test data, allowing the nominal FCR level to be asymptotically exhausted. 
\end{remark}

The next proposition concerns the positive regression dependence on a subset (PRDS) property of the generalized conformal $p$-values \citep{benjamini2001control}; see Appendix~\ref{ssec:prooftheprdspvalue} for a formal definition. Establishing this property requires the notion of an oracle conformal $p$-value (so called because it depends on the unobserved label $Y_{n+j}$):
\begin{equation*}
    p_j^{\oc} =
    \frac{
        \sum_{i = 1}^n
        \bbI \big\{ 
            \delta_i T_i > \delta_{n+j} T_{n+j} 
        \big\} 
        + 
        \Big[ 
            1 + 
            \sum_{i = 1}^n 
            \bbI \big\{ 
                \delta_i T_i = \delta_{n+j} T_{n+j} 
            \big\} 
        \Big]
        \cdot U_j
    }{
        n+1
    }
\end{equation*}
for $j \in [m]$, where $\delta_i = \bbI (Y_i \notin \cC_i^\cI)$ for $i \in [n+m]$.

\begin{proposition}\label{the:prdspvalue} 
    Let $\bp^\ps_{j \to \oc}$ be the sequence obtained by replacing the $j$-th element of $\bp^\ps = (p_j^\ps,\, j \in [m])$ with $p_j^\oc$.
    If Assumption~\ref{ass:fullexchangeability} holds and both the informative set constructor and the trust score function are permutation-invariant to $\{Z_i\}_{i = 1}^{n+m}$, then $\bp_{j \to \oc}^\ps$ is PRDS on $\{j\}$.
\end{proposition}

Given that $p_j^{\ps} = p_j^{\oc}$ under the null $H_{0,j}$, the PRDS property established in Proposition~\ref{the:prdspvalue} enables the standard leave‑one‑out technique for proving the finite‑sample validity of \scip for FCR control.

\subsection{Properties of the \scip procedure}\label{subsec:theory-scip}

We first establish a finite-sample theoretical guarantee for the \scip procedure. 

\begin{theorem}\label{the:fcrcontrolfinitesample}
    If Assumption~\ref{ass:fullexchangeability} holds and both the informative set constructor and the trust score function are permutation-invariant to $\{Z_i\}_{i = 1}^{n+m}$, then the FCR of \scip satisfies $\fcr \le \alpha$.
\end{theorem}

We next investigate the anti‑conservativeness of \scip, which, to our knowledge, has not been studied in the context of FCR analysis. Let $\cH_0 = \{j: Y_{n+j} \notin \cC^\cI(X_{n+j})\}$ and $\pi_0 = |\cH_0|/m$. 
We consider the setting where $\cC^\cI$ is fixed or depends only on the training data (which are treated as fixed in this subsection). Consider the multiple testing problem \eqref{eq:originalhypothesis}. Denote by 
$
\pi_0^* = \mathbb{E}[\pi_0] = \pr\big\{Y_{n+j} \notin \cC^\cI(X_{n+j})\big\}
$
the expected proportion of ``null hypotheses'' in the test sample, where the expectation (probability) is taken over $P_{XY}$, the joint distribution of $(X, Y)$. 

\begin{assumption}\label{ass:independentnullindicator}
 (a) The indicators $\bbI\{Y_i \notin \cC^\cI(X_i)\}$, $i \in [n+m]$, are \iid Bernoulli$(\pi_0^*)$ variables. (b) The pre‑specified FCR level $\alpha$ satisfies $\alpha < n\pi_0^*/(n+1)$. 
\end{assumption}

\begin{remark}
Condition (a) holds when the test data are i.i.d. and independent of $\cC^\cI(\cdot)$. Although the i.i.d. assumption is stronger than the exchangeability condition required for FCR control, it simplifies the theoretical analysis and is commonly adopted in the conformal literature (see, e.g., \citet{lei2018distribution} and \citet{jin2023selection}). Condition (b) is relatively mild when $n$ is large, since in practice our choice of $\alpha$ (typically less than $0.5$) should be significantly smaller than $\pi_0^*$ (typically greater than $0.5$, as null cases usually constitute the majority in multiple testing). 
\end{remark}

The next theorem establishes a lower bound for the FCR of \scip, thereby mitigating the conservativeness issue in existing conformal selection methods, including those of \citet{jin2023selection} and \citet{gazin2025selecting}.

\begin{theorem}\label{the:fcrcontrolexact}
    Suppose both the informative set constructor and the trust score function are permutation-invariant to $\{Z_i\}_{i = 1}^{n+m}$. Under Assumptions~\ref{ass:fullexchangeability}–\ref{ass:independentnullindicator}, and assuming almost surely no ties among the trust scores, then 
    $$
    \fcr \ge \alpha (1-e^{-2n (n\pi_0^*/(n+1) - \alpha)^2 + \log m}) = \alpha - O(m e^{-n}).
    $$
\end{theorem}

\section{Instantiations and Extensions of \scip}\label{sec:applicationofinfospp}

This section demonstrates the versatility of the \scip framework through several concrete examples. By appropriately specifying (a) the informative set constructor and (b) the trust score, \scip can address a diverse range of problems, including conformal selection (Section~\ref{ssec:introducecfbhp}), selective classification (Section~\ref{subsec:instan-sc}), selection with general informativeness constraints (Section~\ref{ssec:proposedinfospp}), and diversified selection (Section~\ref{ssec:diversifiedselection}).

\subsection{The \cfbhp procedure for conformal selection}\label{ssec:introducecfbhp}
We begin with the conformal selection task introduced in Example~\ref{exa:conformalselection}, where two common forms of indifference regions arise: $(-\infty, c_0]$ for one‑sided alternatives and $[c_l, c_u]$ for two‑sided alternatives. The instantiation of \scip for this problem is termed the \cfbhp procedure. The name reflects two key features: (a) it coincides with the \cfbh method of \citet{jin2023selection} in the one‑sided setting, and (b) it naturally accommodates two‑sided alternatives, which lie beyond the scope of the original \cfbh. To facilitate comparison with existing methods, throughout this subsection we employ the predictive function $\hmu(x)$—assumed independent of the mixed data—to construct the trust score, following the setup in \citet{jin2023selection}. 

For the one‑sided case with indifference region $(-\infty, c_0]$, \cfbhp sets $\cC^{\cI}(x) = (c_0, \infty)$ and $T(x, \cC^{\cI}) = \hmu(x)$.
This simple choice yields an important connection to existing work. \citet{jin2023selection} showed, through insightful heuristics and asymptotic analyses, that equipping \cfbh with the \emph{clipped score},
\begin{equation*}
    V_{\rm clp}(x,y) = \hmu(x) - c_0 - 2M \cdot \bbI(y > c_0),
    \quad \mbox{with } M > \sup_x |\hmu(x)|, 
\end{equation*}
makes \cfbh asymptotically anti‑conservative and yields strong empirical performance. As shown in Appendix~\ref{assec:anotherlookcfbh}, \cfbh with clipped scores coincides with \cfbhp described above.
This equivalence highlights the effectiveness of our framework, which yields the same methodology from a generic, principled construction rather than problem‑specific heuristics. Moreover, \scip (\cfbhp) offers two distinct advantages. First, it provides a systematic pathway to improve trust scores (Section~\ref{sec:trust-score-cons}), a flexibility not available in \citet{jin2023selection}. Second, while \citet{jin2023selection} provide heuristics for why \cfbh is anti‑conservative, our Theorem~\ref{the:fcrcontrolexact} and Corollary~\ref{cor:fcrcfbhp} below deliver an explicit lower bound that precisely quantifies the phenomenon, a theory not previously established in the literature. 

For two‑sided alternatives with indifference region $[c_l, c_u]$, \cfbhp adopts a data‑driven informative set constructor that adapts to the estimated signal direction:
\begin{equation*}
    \cC^\cI (x) = (c_u, \infty)
    \mbox{ if } 
    (c_l + c_u)/2 - \hmu(x) \le 0
    , \mbox{ and } 
    \cC^\cI(x) = (-\infty, c_l) 
    \mbox{ otherwise}.
\end{equation*}
The corresponding trust score is $T(x, \cC^\cI(x)) = \bigl| (c_l + c_u)/2 - \hmu(x) \bigr|$.
Critically, the original \cfbh method of \citet{jin2023selection} cannot handle two‑sided alternatives because it relies on a one‑sided monotone score and lacks a rule to decide the direction of the effect. In contrast, \cfbhp naturally addresses two‑sided alternatives through its flexible, data‑driven constructor.

Given that the predictive function $\hmu (x)$ is independent to the mixed sample, both the informative set constructor and the trust score are permutation-invariant.
Thus, directly applying Theorems~\ref{the:fcrcontrolfinitesample} and~\ref{the:fcrcontrolexact} yields the theoretical properties of \cfbhp. 

\begin{corollary}\label{cor:fcrcfbhp}
Suppose that $\hmu$ is learned independently of the mixed sample.    Under Assumption~\ref{ass:fullexchangeability}, \cfbhp for both one‑ and two‑sided alternatives satisfies $\fcr \le \alpha$.
    Further if the data in $\mathcal D^{\text{mix}}$ are \iid  and $\alpha < n \pi_0^* / (n+1)$, we have 
    $\fcr \ge \alpha - m \alpha e^{-2n (n \pi_0^* / (n+1) - \alpha)^2}$, where $\pi_0^* = \pr \{Y \notin \cC^\cI (X) \mid \hmu\}$.
\end{corollary}

\subsection{Application to selective classification}\label{subsec:instan-sc}

We now instantiate \scip for the selective classification tasks given in Example~\ref{exa:selectiveclassification}. 

The first task, \emph{target‑class identification}, is to identify test points belonging to a pre‑specified class $y_0 \in [K]$. Here, the informative set constructor is simply $\cC^{\cI}(x) = \{y_0\}$, and a natural trust score is the estimated probability:
\begin{equation*}
    T(x, \cC^{\cI}) = \hpr(Y = y_0 \mid X = x).
\end{equation*}
Remarkably, this instantiation of \scip recovers the \texttt{FASI} procedure of \citet{rava2021burden}, originally developed via a different technical route involving mirror processes and counting knockoffs. We provide a proof of this equivalence in Appendix~\ref{subsec:equi-sc}. This connection further highlights the broad applicability of our framework. Moreover, while \texttt{FASI} and its theory are limited to binary settings ($K=2$), \scip readily extends to selective classification with general $K \ge 2$ classes.

In many applications—such as disease diagnosis—the goal is to identify a single class for each test unit, without restricting attention to a particular pre‑specified category. This scenario, referred to as the \emph{singleton identification}, corresponds to imposing the informativeness constraint $\cI = \{\{1\}, \dots, \{K\}\}$. For this task, we adopt a data‑driven constructor with a natural trust score:
\begin{equation*}
    \cC^{\cI}(x) = \{\operatorname{argmax}_{k}\, \hpr(Y = k \mid X = x)\},
    \quad 
    T(x, \cC^{\cI}(x)) = \hpr(Y \in \cC^{\cI}(x) \mid X = x).
\end{equation*}
With these specifications, \scip coincides with the selective classification procedure of \citet{zhao2023controlling}, which has been shown to be optimal for this task under mild conditions. 
An equivalence proof is provided in Appendix~\ref{subsec:equi-sc}.

Both instantiations inherit the finite‑sample validity of \scip, provided the estimated probability function is permutation‑invariant with respect to the mixed sample. For more complex informativeness constraints (e.g., requiring prediction sets of a fixed size $k_0 > 1$ or excluding scientifically irrelevant categories), the construction of both the informative set and the trust score becomes more involved; we address such issues in the next subsection.

\subsection{Selection under general constraints: the \infospp procedure}\label{ssec:proposedinfospp}

To handle more complex informativeness constraints, we develop a general instantiation of \scip based on the notion of $\cI$-adjusted $p$-values introduced by \cite{gazin2025selecting}.
The resulting method, called the \infospp procedure, inherits the flexibility of \infosp while provably dominating it in asymptotic power. 
We first introduce the \infospp procedure and then compare it with \infosp.% in terms of statistical power.

We begin with the split conformal prediction framework \citep{lei2018distribution}.
Let $V(x,y):\cX\times\cY\to\bbR^{+}$ be a non‑conformity score and $\cD^\ca$ be a calibration set.
For any level $q\in(0,1)$, the conformal prediction set for feature $x$ is
\begin{equation}\label{eq:splitconformalpredictionset}
    \cC^{\cp,q} ( x; \cD^\ca,V )
    =\bigg\{
        y \in\cY: 
        \frac{
            1 + \sum_{i \in \cD^\ca} 
            \bbI \big\{ 
                V(X_i,Y_i) \ge V(x,y) 
            \big\}
        }{|\cD^\ca|+1} \, > \, q
    \bigg\}.
\end{equation}
The \emph{$\cI$-adjusted $p$-value} is defined as the smallest $q$ level such that the resulting conformal prediction set is informative, denoted by
$ \tq(x; \cD^\ca) = \inf \big\{
    q\in(0,1] : \cC^{\cp,q} (x;\cD^\ca) \in \cI
\big\} $ with the convention $\inf \emptyset = 1$.
The following mild requirements on the shape of informativeness constraint ensure that $\tq$ is well-defined.
\begin{assumption}[Following Assumption~2 of \cite{gazin2025selecting}]
    \label{ass:monotoneinfoconstraints}
    (i) Monotonicity: for any $\cC,\cC'\subseteq\cY$ with $\cC'\subseteq\cC$, $\cC\in\cI$ implies $\cC'\in\cI$.
    (ii) Almost surely, the function $\nu \mapsto \bbI\big\{ \{y: V(X,y) \le \nu\} \in \cI \big\}$ is right‑continuous and non‑increasing. %, and satisfies $\bbI \{ \cC^{\cp,1} (X_{n+j}; \cD^\ca) \in \cI\} = 1$.
\end{assumption}

Unlike \infosp that determines the selection set by applying \bhfdr procedure to $\cI$-adjusted $p$-values,
\infospp builds on the \emph{truncated $\cI$-adjusted $p$-values}:
\begin{equation}\label{eq:truncatediadjustedpvalue}
    \tq^+ (x, \tau; \cD^\ca) = \inf \big\{ q \in [\tau,1] : \cC^{\cp,q} (x; \cD^\ca) \in \cI \big\} 
    =
    \max \{ \tau,\tq(x; \cD^\ca) \},
\end{equation}
where $\tau \in [0,\alpha]$ is a threshold that trades off trustworthiness and informativeness in constructing the prediction sets.
Specifically, we employ a threshold $\htau_\bh^0$ that is computed using an independent calibration sample $\cD^\ca_0$:
Compute $\tq_i^0 = \tq(X_i; \cD^\ca_0)$ for all $i \in [n+m]$ and apply the $\alpha$-level \bhfdr procedure to $\{\tq_i^0\}_{i = 1}^{n+m}$ to obtain the threshold $\htau_\bh^0$.
Then, to instantiate \scip, we set $\cC_i^\cI = \cC^{\cp, \tq_i^+} (X_i; \cD^\ca_0)$%
\footnote{For the corner cases, $\tq_i^+ = 1$, we set $\cC_i^\cI = \emptyset$ and $T_i = 0$ and ignore them during the selection.}
and $T_i = 1 - \tq_i^+$ for all $i \in [n+m]$, where $\tq_i^+ = \max\{\tq_i^0, \htau_\bh^0\}$ are the truncated $\cI$-adjusted $p$-values.
Finally, \infospp employs~\eqref{eq:qbh} at level $\alpha$ to selectively report the prediction sets.
We summarize the procedure in Algorithm~\ref{alg:infospp}.

Since the threshold output by \bhfdr is permutation-invariant with respect to the input $p$-values, the truncated function $\tq^+ (x, \htau_\bh^0; \cD^\ca_0)$ is also permutation-invariant with respect to the mixed data.
The following corollary, which follows directly from Theorem~\ref{the:fcrcontrolfinitesample}, implies that \infospp provides finite-sample FCR control.

\begin{corollary}\label{cor:infosppfcr}
 Under Assumptions~\ref{ass:fullexchangeability} and \ref{ass:monotoneinfoconstraints}, all reported sets from \infospp are informative, with the FCR controlled at level $\alpha$.
\end{corollary}

We now compare the powers of \infospp and \infosp. To assess the two methods on equal footing, consider a modified \infosp that uses the threshold $\htau_\bh^0$ derived from \infospp (cf. Algorithm~\ref{alg:infospp}), denoted by $\cR^\isp_{\rm \scriptscriptstyle mod} = 
\{\cC^{\cp, \htau_{\bh}^0}_{n+j} \colon \tq_{n+j}^0 \le \htau_{\bh}^0\}$. This modified procedure has improved stability and enjoys asymptotic FCR control. Meanwhile, we simplify \infospp by replacing the unit‑specific tie‑breaking variables $U_j$ in~\eqref{eq:pseudopvalue} with a single shared variable $U$; this variant still controls the FCR (see Appendix~\ref{ssec:modifiedinfospp}).
We denote the output of modified \infospp by $\cR_{\rm \scriptscriptstyle mod}^{\rm \scriptscriptstyle ISP+}$.

Our asymptotic power analysis requires the following technical conditions.

\begin{assumption}\label{ass:conditionsforpoweranalysis}
    The data $\cD^\ca_0 \cup \cD^\ca \cup \cD^\te$ consist of \iid units and are independent of the training set $\cD^\tr$.
\end{assumption}

\begin{assumption}\label{ass:smoothnesscondition}
    Let $\tnu(x) = \min \big\{\nu \ge 0 : \{y: V(x,y) \le \nu\} \in \cI \big\}$.
    One of the following conditions holds:
    (a) There exists $\nu^* $ such that $\pr \{V(X,Y) > \nu^*\} > 0$ and $\pr \{\tnu (X) > \nu^*\} > \pr \{V(X,Y) > \nu^*\} / \alpha$.
    (b) There exists a constant $\epsilon > 0$ such that $\pr \{\tnu (X) > \nu^*\} \le \pr \{V(X,Y) > \nu^*\} / \alpha - \epsilon$ for all $\nu \in  \big\{\nu: \pr \{V(X,Y) \ge \nu\} > 0 \big\}$.
\end{assumption}

\begin{remark}
Assumption~\ref{ass:conditionsforpoweranalysis} is standard in power analyses of conformal methods \citep{marandon2024adaptive}. Assumption~\ref{ass:smoothnesscondition} imposes a mild condition on the distribution of $\tnu(X_i)$, ensuring that $\htau_\bh^0$ is either asymptotically bounded below by a positive constant or converges to zero. Similar conditions are routinely adopted in the asymptotic analysis of multiple testing procedures. Both assumptions are only used for power analysis and not needed for FCR control. 
\end{remark}

The next proposition establishes the asymptotic dominance of \infospp over \infosp, which is corroborated by the empirical results in Section~\ref{sec:numericexperiment}.

\begin{proposition}\label{prop:asymptoticdominance}
    Under Assumptions~\ref{ass:monotoneinfoconstraints} to~\ref{ass:smoothnesscondition},
    $\lim_{|\cD_0^\ca|,n,m\to\infty} 
    \pr (\cR_{\rm  mod}^\isp \subseteq \cR_{\rm mod}^{\isp+}) = 1$.
\end{proposition}

The improvement largely reflects that the generalized conformal $p$-value, built on trust scores, provides a more accurate measure of deviation from the ``null hypothesis'' of noncoverage in \eqref{eq:originalhypothesis}; see Appendix~\ref{assec:anotherlookinfosp} for further discussion.

\subsection{Diversified selection}\label{ssec:diversifiedselection}

The \scip framework extends beyond informativeness to a wider range of scenarios. We can incorporate a diversity criterion into selective prediction; the resulting \scip method prevents fixating on narrow regions, enhances representativeness, and reduces information redundancy \citep{wu2024optimal, nair2025diversifying}.

Given an informative set constructor $\cC^\cI(\cdot)$, we aim to select prediction sets with FCR control while maximizing diversity. Following \citet{wu2024optimal}, we quantify diversity by employing a \emph{similarity-aware trust score}. Let $\psi: \cX \to [0,1]$ be an estimate of $\pr\{ Y \notin \cC^\cI (X) \mid X = x\}$, and let $\Psi \in \bbR^{n+m}$ be a vector with entries $\psi_i = \psi(X_i)$. Define $\mathbb{A}_{\Psi} \in \bbR^{(n+m) \times (n+m)}$ as the similarity matrix whose $(i, j)$-th entry is given by $(1-\psi_i)(1-\psi_j) s(X_i, X_j)$, where $s: \cX \times \cX \to \bbR^+$ measures the similarity between two units.

Assume $\mathbb{A}_{\Psi}$ is positive definite.
Motivated by \citet{wu2024optimal}, define
\begin{equation}\label{eq:diversityawarescore}
    \mathbb{T} = 
    (T_1,\cdots,T_{n+m})^\top 
    = 
    \mathbb{A}_{\Psi}^{-1} 
    \big\{ 
        (\alpha u_2 - u_1) \Psi + (u_2 - \alpha u_3) \boldsymbol{1}
    \big\},
\end{equation}
where $u_1 = \Psi^\top \mathbb{A}_{\Psi}^{-1} \Psi$, $u_2 = \Psi^\top \mathbb{A}_{\Psi}^{-1} \boldsymbol{1}$, $u_3 = \boldsymbol{1}^\top \mathbb{A}_{\Psi}^{-1} \boldsymbol{1}$, $\boldsymbol{1}$ denotes the $(n+m)$-dimensional all-ones vector, and $\alpha$ is the target FCR level. 
This score captures the individual diversity while being permutation-invariant with respect to the mixed data.
The following corollary, which directly follows from Theorem~\ref{the:fcrcontrolfinitesample}, strengthens the theory in \citet{wu2024optimal} by providing finite-sample guarantees. Moreover, the flexibility of informative set constructor generalizes the scope of \citet{wu2024optimal}, accommodating a wide range of selective inference tasks (see Sections~\ref{ssec:introducecfbhp} to~\ref{ssec:proposedinfospp}).

\begin{corollary}\label{cor:diversityaware}
    Suppose that the informative set constructor is permutation-invariant.
    Under Assumption~\ref{ass:fullexchangeability}, \scip equipped with scores in~\eqref{eq:diversityawarescore} controls FCR below $\alpha$.
\end{corollary}

\section{Strategies for Constructing Powerful Trust Scores}\label{sec:trust-score-cons}

This section develops a general strategy for enhancing the efficiency of trust scores to further improve the power of \scip. To facilitate a tractable optimality analysis, we introduce a surrogate error metric, the \emph{marginal false coverage rate} (mFCR), defined as $\mathbb{E}(|\mathcal{S} \cap \mathcal{H}_0|) / \mathbb{E}(|\mathcal{S}|)$. Define the \emph{counting power} as $\cpow(\mathcal{S}) = \mathbb{E}(|\mathcal{S}|)$. When all candidate prediction sets are informative, the goal is to maximize $\cpow$ subject to $\mfcr \le \alpha$.
\begin{remark}
In many applications, different methods may construct distinct informative intervals for the same unit. Intuitively, for a given unit, a smaller prediction set is considered more powerful. Hence, the \emph{resolution-adjusted power} (rPOW), given by $\mathbb{E}\big[ \sum_{j \in \mathcal{S}} |\mathcal{C}_{n+j}|^{-1} \big]$ \citep{gazin2025selecting}, serves as an alternative criterion for comparing the efficacy of different methods. Our optimality theory is developed based on counting power. A theoretical characterization of optimality based on resolution-adjusted power remains an open problem and presents a promising direction for future research. 
\end{remark}

Motivated by the optimality theory based on the \emph{local false discovery rate} (lfdr, \citealp{efron2001empirical, sun2007oracle, heller2020optimal}), we define 
\begin{equation*}
    T^\ora (x,\cC^\cI) 
    =
    \pr \big\{ 
        Y \in \cC^\cI(X)\mid X=x
    \big\}
\end{equation*}
as the \emph{oracle} trust score for a given informative set constructor $\cC^\cI$ when the distribution of $(X,Y)$ is known.
We start with a regularity condition, which holds if the test units are \iid and the informative set constructor is learned on an independent sample.
\begin{assumption}\label{ass:independentdata}
    The random elements $\{ (X_{n+j}, Y_{n+j}, \cC^\cI(X_{n+j}) ) : j \in [m]\}$ are \iid conditional on $\cD^\tr$ and $\cD^\ca$.
\end{assumption}

The next theorem provides a guiding principle for constructing trust scores.
\begin{theorem}\label{the:maxpower}
 Consider a class of thresholding rules of the form $\cS(T,t) = \{j: T(X_{n+j})\ge t\}$.
    Let $t_{\alpha}^{*} \in [\alpha,1)$ denote the threshold such that 
    $
    \mfcr \big( \cS(T^\ora, t_{\alpha}^{*}) \big) = \alpha.
    $ 
    Under Assumption~\ref{ass:independentdata}, the counting power of 
    $\cS(T^\ora, t_{\alpha}^{*})$ is no less than that of any other selection rule 
    $\cS' = \cS(T',t')$, where $T': \cX \to \bbR$ is an arbitrary trust score and $t'$ satisfies 
    $\mfcr(\cS') \le \alpha$.
\end{theorem}

In practice, $T^{\ora}$ must be estimated from data. For classification problems, 
\begin{equation*}
    T^{\ora}(x, \cC^\cI) = \sum_{k=1}^K \mathbb{I}(k \in \cC^\cI(x)) \cdot \pr(Y = k \mid X = x)
\end{equation*}
can be approximated by replacing $\pr$ with $\hpr$, an estimated probability function that 
can be obtained directly from off-the-shelf classifiers. This construction coincides with the trust scores derived in Section~\ref{subsec:instan-sc}. For the more complex size-$k_0$ constraint, this optimality-based strategy leads to an enhanced variant of \infospp, which we term \infosppp. Numerical studies in Section~\ref{ssec:classificationsimulated} demonstrate that \infosppp improves upon both \infosp and \infospp.

For regression settings, where estimating the conditional distribution of $(Y \mid X)$ is challenging, we instead solve a binary classification problem to approximate $T^\ora$. We present a split‑conformal implementation for clarity; a more complicated method based on positive‑unlabeled learning is given in Appendix~\ref{asec:classifieronpu}.

Suppose an independent training sample $\cD^\tr$ is available. For a given informative set constructor $\cC^\cI$, define a binary label for each $i \in \cD^\tr$ as $A_i = 1$ if $Y_i \in \cC^\cI(X_i)$ and $A_i = -1$ otherwise. We train a classifier by minimizing the empirical risk
\begin{equation}\label{eq:empiricalrisk}
    \hat{\mathcal{L}} (\cD^\tr; g, l) = 
    \frac{1}{|\cD^\tr|}
    \sum_{i\in\cD^\tr}
    \Big\{
        \lambda \, \mathbb{I}(A_i=1) \, \ell(1, g(X_i))
        +
        \mathbb{I}(A_i=-1) \, \ell(-1, g(X_i))
    \Big\},
\end{equation}
where $\ell(u,v) = - \mathbb{I}(u=1)\log(v) - \mathbb{I}(u=-1)\log(1-v)$ is the cross‑entropy loss, $\lambda>0$ balances the cost of misclassifications, and $g: \cX \to \mathbb{R}$ is a measurable function whose output is used for classification.

Let $\mathcal{L}(g,l) = \mathbb{E}[\hat{\mathcal{L}}(\mathcal{D}^{\text{tr}}; g,l)]$ and let $g^{*}$ be its minimizer over all measurable functions. Following Lemma~4.3 of \citet{marandon2024adaptive}, $g^{*}$ is a strictly increasing transformation of $T^{\text{ora}}$ under mild conditions. As threshold‑based selection is invariant under any strictly increasing transformation of the score, using a consistent estimate of $g^{*}$ as the trust score achieves asymptotically optimal power. Hence, we can further improve \cfbhp as follows. For the one‑sided alternative, set $A_i = 1$ if $Y_i > c_0$ and $A_i = -1$ otherwise. For the two‑sided alternative, set $A_i = 1$ if $\mathbb{I}\{Y_i \ge c_u,\ 2\hmu(X_i) \ge c_u + c_l\} + \mathbb{I}\{Y_i \le c_l,\ 2\hmu(X_i) < c_u + c_l\} = 1$, and $A_i = -1$ otherwise. We then minimize the empirical risk in~\eqref{eq:empiricalrisk} over a specified function class to obtain $\hat{g}$, which serves as the trust score for \cfbhp. We term this procedure \cfbhpp (as it employs the approximated oracle score).

\begin{remark}
The strategy of approximating the oracle score by training a binary classifier is motivated by \citet{marandon2024adaptive}. A key distinction, however, is that our informative selective conformal prediction framework has access to both null and non-null samples, whereas the novelty detection setting considered in \citet{marandon2024adaptive} only has null samples.
\end{remark}

\section{Numerical Experiments}\label{sec:numericexperiment}

This section evaluates the proposed methods through numerical experiments. Sections~\ref{ssec:regressionsimulated} and~\ref{ssec:classificationsimulated} examine regression and classification tasks, respectively, using simulated data. The performance on real data is illustrated on two real datasets: the CIFAR-10 example discussed in Section~\ref{sec:illustrativeexample} and Appendix~\ref{assec:detailillustrativeexample}, and a drug discovery application presented in Section~\ref{ssec:drugdiscovery}.

Under the InfoFCR framework (with general constraints), the only existing methods applicable are \infosp and its variant \infoscop \citep{bao2024selective, gazin2025selecting}. The latter first screens out units for which informative prediction sets are unlikely to be constructed. In many settings, \infoscop can improve upon \infosp when the screening step effectively narrows the focus. Nevertheless, no principled screening rule exists, nor are there theoretical guarantees for power improvement. Consequently, in simulation studies we adopt a screening strategy that tends to yield higher power (Section~\ref{ssec:regressionsimulated}), while in the real‑data analysis (Section~\ref{ssec:drugdiscovery}) we examine the performance of \infoscop under varying screening parameters.

We evaluate performance using empirical FCR, counting power (cPOW), and resolution‑adjusted power (rPOW); see the beginning of Section~\ref{sec:trust-score-cons} for definitions of cPOW and rPOW. For procedures that require an additional calibration set (\infoscop and \infospp), we split the original calibration data into two equal halves. Reproducible code is available at \texttt{github.com/wangchengLi6/SCIP}.

\subsection{Simulation in regression settings} \label{ssec:regressionsimulated}

Consider a regression task where the goal is to construct prediction intervals that only include positive values, i.e., $\cI = \{[a,b] \subseteq \mathbb{R} : 0 < a \le b < \infty\}$. In addition to \infosp, \infoscop, and \infospp, we also consider a naive baseline that builds standard conformal prediction intervals and then retains only the informative ones. The simulated data obey
\begin{equation*}
    Y_i = \mu^* (X_i) + \epsilon_i, \quad
    X \sim N(0,1), \quad 
    \epsilon_i \sim 
    N \big(0, 1/4 \big), \quad 
    \mu^* (x) = (2x^2 + 1)/6.
\end{equation*}
We use the non-conformity score $V(x,y) = |y - \hmu(x)|$, where $\hmu(x)$ is a predictive function trained on an independent dataset. To assess performance under model misspecification, we consider a family of predictive functions $\hmu(x;\eta) = \{(\eta+2)x^2 + (1-\eta^3)\}/6$, with $\eta \ge 0$ controlling the estimation bias ($\eta = 0$ corresponds to the error‑free case). The target FCR level is $\alpha = 0.1$, and the sample sizes are $n = m = 1000$. For the screening step of \infoscop, we apply \cfbh at level $\alpha' = \alpha/2 = 0.05$, testing the null $Y_{n+j} \le 0$ (since an informative set never covers a unit with $Y \le 0$). The simulation results for $\eta = 0$ and $\eta = 1$ are summarized in Figure~\ref{fig:comparenewparadigm}. The data points shown in the figure are from a single representative replication, while the reported FCR, cPOW, and rPOW are averaged over 1,000 independent repetitions.

\begin{figure}[!t]
	\includegraphics[width = \linewidth]{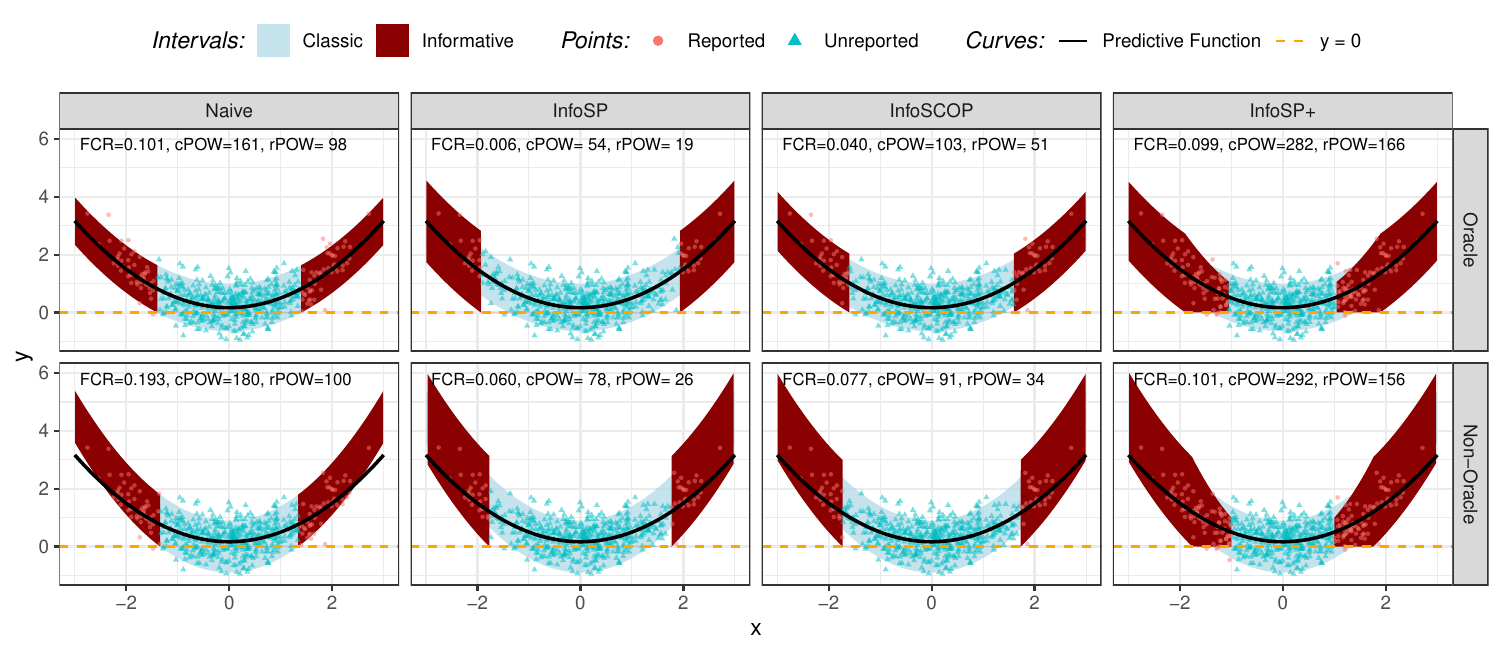}
	\centering
    \caption{
    \small
    The lightblue region represents the split conformal prediction intervals with $90\%$ coverage.
    The darkred region corresponds the selectively constructed intervals.
    The black curves correspond to the predictive functions $\hmu(x) = \mu^*(x)$ and $\hmu(x) = x^2/2$. }
    \label{fig:comparenewparadigm}
\end{figure}

We further compare the four methods under a range of $\eta$ values, which correspond to varying levels of estimation bias. The results are summarized in Figure~\ref{fig:compareparadigmacrosseta}. 

\begin{figure}[!t]
    \centering
    \includegraphics[width=\linewidth]{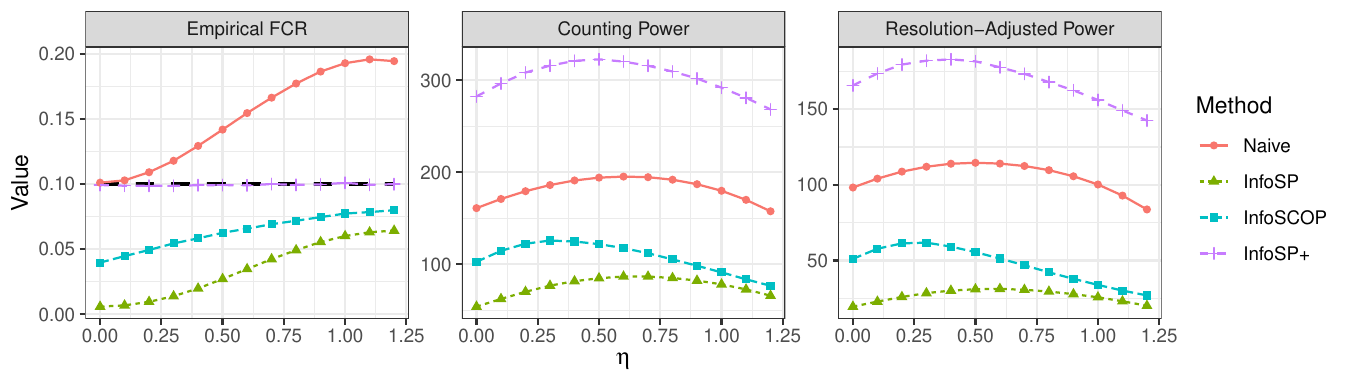}
    \caption{
    \small
    Comparison of the four methods across varies $\eta$. The reported FCR, cPOW, and rPOW are averaged over 1,000 independent repetitions. }
    \label{fig:compareparadigmacrosseta}
\end{figure}

Several important patterns can be observed. First, the naive method fails to control the FCR. In contrast, the other three methods—which inherit the model‑free advantage of conformal inference—effectively control the FCR below the nominal level regardless of the severity of model bias. Second, \infospp controls the FCR precisely at the nominal level and achieves the highest power, corroborating our theory of finite‑sample validity and anti‑conservativeness. Third, \infoscop improves upon \infosp, yielding a less conservative FCR and higher power. We emphasize, however, that this observation, based on a limited set of simulation settings, does not hold universally. The performance of \infoscop depends on the effectiveness of the screening rule, and no theoretical guarantee exists that \infoscop always dominates \infosp. This remains an interesting direction for future research.

\subsection{Simulation in classification scenario}\label{ssec:classificationsimulated}

We consider a classification problem with four categories and the informativeness constraint $\cI = \{ \cC \subseteq [4] \colon 0 < |\cC| \le 2 \}$ (see Section~\ref{subsec:instan-sc}). We compare \infospp with the naive method and \infosp. Since no effective screening criterion is available, \infoscop is excluded from the comparison. Instead, we include \infosppp, which constructs trust scores following the strategy outlined in Section~\ref{sec:trust-score-cons}. The data $\{(X_i,Y_i)\}_{i=1}^{n+m}$ are generated from the model:
\[
\pr(Y_i = k \mid X_i) = \frac{\exp\{X_i^\top \beta_k^*\}}{\sum_{k'=1}^4 \exp\{X_i^\top \beta_{k'}^*\}},
\]
where $X_i = (X_{i1}, X_{i2})^\top$ with independent standard normal components, and the coefficient vectors are
$\beta_1^* = (1,-1)^\top$, $\beta_2^* = (-1,1)^\top$, $\beta_3^* = (1,0.5)^\top$, $\beta_4^* = (0.5,1)^\top$.
We set $n = m = 1000$. The non‑conformity score is $V(x,y) = 1 - \hpr(Y = y \mid X = x)$, where the conditional probability is estimated from an independent training set of size $n$. We evaluate the naive method, \infosp, \infospp and \infosppp at varying FCR levels and report the results in Figure~\ref{fig:classificationvariesalp}.

\begin{figure}[!t]
    \centering
    \includegraphics[width=\linewidth]{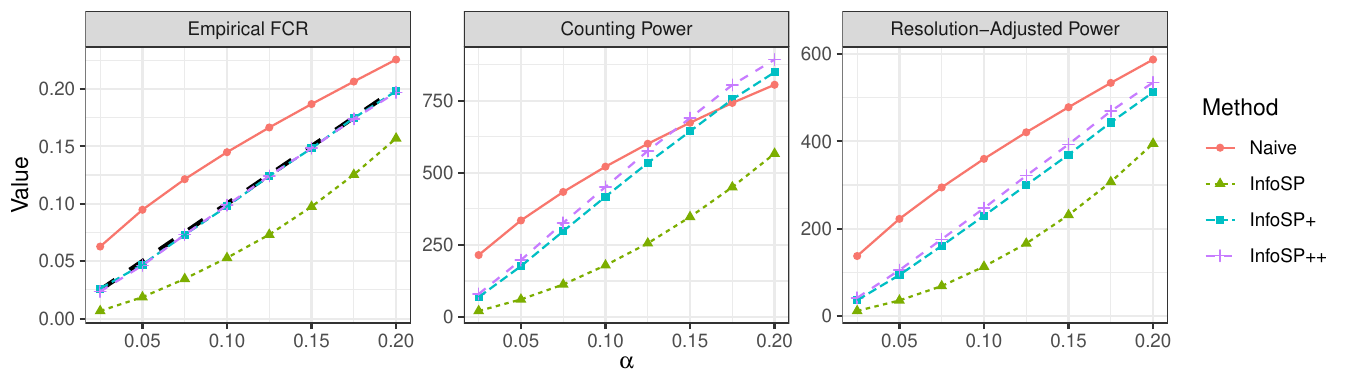}
    \caption{
    \small
    Comparison of the four methods across varies target FCR level $\alpha$. The black dashed line represents the function $y = x$. All the metrics are computed over 1,000 repetitions.}
    \label{fig:classificationvariesalp}
\end{figure}

The following patterns, similar to those observed in Section~\ref{ssec:regressionsimulated}, can be noticed. First, the naive method fails to control the FCR, whereas the other three methods remains valid, controlling the FCR below the target level. Second, \infosp is less conservative here than in the regression setting. Third, \infospp controls the FCR precisely at the nominal level and achieves higher power than \infosp. Finally, \infosppp, which uses the same informative set constructor, also controls the FCR exactly at the nominal level while further improving the power of \infospp through more effective trust scores.

\subsection{Real-world application: drug discovery}\label{ssec:drugdiscovery}

Machine learning is now widely integrated into drug discovery pipelines \citep{vamathevan2019applications}. A prominent example is virtual screening \citep{huang2007drug}, which can substantially reduce the experimental burden compared to exhaustively evaluating all candidates via high‑throughput screening (HTS). In a typical workflow, HTS is first performed on a subset of candidates to obtain labeled outcomes; a predictive model is then trained on these labeled data; finally, the model is used to screen promising candidates from the remaining unlabeled ones for follow‑up studies.

This section considers a drug–target interaction (DTI) prediction task, where the response variable $Y$ is the binding affinity of a drug–target pair and the feature object $X$ encodes the molecular structure. We use a public dataset from \citet{davis2011comprehensive}, which contains molecular structures and binding affinities for over 30,000 drug–target pairs. From this dataset, we extract a subset of $n_{\rm tot} = 10,000$ units for our analysis. Our goal is to construct prediction intervals for pairs predicted to have high binding affinity. Since downstream validation is costly, it is also essential to provide error rate guarantees.

This problem fits naturally into the InfoFCR framework. The threshold for meaningful discovery is set to the median of $Y$ in the extracted sample, $q_{0.5}=6.13$, leading to the informativeness constraint $\cI = \{[a,b] \subseteq \mathbb{R} : q_{0.5} \le a \le b < \infty\}$. We split the data into training, calibration, and test sets in a 2:2:6 ratio. The predictive model $\hat{\mu}$ is trained using the pipeline from the \texttt{DeepPurpose} repository \citep{huang2020deeppurpose}, and the non‑conformity score is $V(x,y)=|y-\hat{\mu}(x)|$. The target FCR level is $\alpha=0.2$, and all metrics are computed over 100 random data partitions.

We now describe the methods included in the comparison. It is natural to consider the naive method and \infospp (as a special case of \scip). The conformal selection method of \citet{jin2023selection} is excluded because it does not produce prediction intervals for the selected units. Furthermore, informative sets cannot be constructed for more than half of the test units; consequently, \infosp is not suitable due to its low power. Instead, we apply \infoscop by first filtering out unpromising units before constructing informative prediction intervals. A natural choice for the screening procedure is the \cfbh method, which tests the $m$ hypotheses $Y_{n+j} \le q_{0.5}$, $j\in[m]$. As no established rule exists for choosing the screening threshold, we apply \cfbh at three levels (0.2, 0.1, 0.05); the inclusion of multiple thresholds ensures a fair and comprehensive comparison. 

The results are summarized in Table~\ref{tab:drugcompareparadigm}. The naive method exhibits an inflated FCR. While all three variants of \infoscop control the FCR, their power remains low. Moreover, the performance of \infoscop is sensitive to the screening cutoff, and a principled strategy for determining this cutoff remains an open problem. In contrast, the proposed \infospp method controls the FCR at the nominal level, ensures that all reported intervals satisfy the informativeness constraint, and achieves the highest power among all methods considered.

\begin{table}[!htbp]
    % Source file path: /python/SCIP-Drug/outputs/df_drugdata_infosp_compare_paradigm_formorecases.csv
    \caption{
    Comparison of the three methods on the DTI data.}
    \centering
    \begin{tabular}{@{}lrrr@{}}
        \toprule
        Method          & FCR   & cPOW  & rPOW  \\ \midrule
        Naive           & 0.350 & 352.3 & 144.9 \\
        % \infosp   & 0 & 0   & 0   \\
        \infoscop-0.20  & 0.007 & 5.3   & 1.8   \\
        \infoscop-0.10  & 0.053 & 31.1  & 11.2  \\
        \infoscop-0.05  & 0.007 & 6.0   & 2.3   \\
        \infospp        & 0.206 & 215.0 & 69.9  \\ \bottomrule
    \end{tabular}
    \label{tab:drugcompareparadigm}
\end{table}

\singlespace
\bibliographystyle{abbrvnat}
\bibliography{reference}

@article{weinstein2017power,
  title = {A power and prediction analysis for knockoffs with lasso statistics},
  author = {Weinstein, Asaf and Barber, Rina and Candes, Emmanuel},
  journal = {arXiv preprint arXiv:1712.06465},
  year = {2017}
}

@article{angelopoulos2024theoretical,
  title = {Theoretical foundations of conformal prediction},
  author = {Angelopoulos, Anastasios N and Barber, Rina Foygel and Bates, Stephen},
  journal = {arXiv preprint arXiv:2411.11824},
  year = {2024}
}

@article{zhao2023controlling,
  title = {Controlling fsr in selective classification},
  author = {Zhao, Guanlan and Su, Zhonggen},
  journal = {arXiv preprint arXiv:2311.03811},
  year = {2023}
}

@article{wang2022elementary,
  title = {Elementary proofs of several results on false discovery rate},
  author = {Wang, Ruodu},
  journal = {arXiv preprint arXiv:2201.09350},
  year = {2022}
}

@article{gui2025acs,
  title = {ACS: An interactive framework for conformal selection},
  author = {Gui, Yu and Jin, Ying and Nair, Yash and Ren, Zhimei},
  journal = {arXiv preprint arXiv:2507.15825},
  year = {2025}
}

@article{nair2025diversifying,
  title = {Diversifying conformal selections},
  author = {Nair, Yash and Jin, Ying and Yang, James and Candes, Emmanuel},
  journal = {arXiv preprint arXiv:2506.16229},
  year = {2025}
}

@article{benjamini1995controlling,
  title = {Controlling the false discovery rate: a practical and powerful approach to multiple testing},
  author = {Benjamini, Yoav and Hochberg, Yosef},
  journal = {Journal of the Royal Statistical Society, Series B},
  year = {1995},
  volume = {57},
  number = {1},
  pages = {289--300},
  publisher = {Wiley Online Library}
}

@article{bates2023testing,
  title = {Testing for outliers with conformal p-values},
  author = {Bates, Stephen and Cand{\`e}s, Emmanuel and Lei, Lihua and Romano, Yaniv and Sesia, Matteo},
  journal = {The Annals of Statistics},
  year = {2023},
  volume = {51},
  number = {1},
  pages = {149--178},
  publisher = {Institute of Mathematical Statistics}
}

@article{marandon2024adaptive,
  title = {Adaptive novelty detection with false discovery rate guarantee},
  author = {Marandon, Ariane and Lei, Lihua and Mary, David and Roquain, Etienne},
  journal = {The Annals of Statistics},
  year = {2024},
  volume = {52},
  number = {1},
  pages = {157--183},
  publisher = {Institute of Mathematical Statistics}
}

@article{mary2022semisupervised,
  title = {Semi-supervised multiple testing},
  author = {Mary, David and Roquain, Etienne},
  journal = {Electronic Journal of Statistics},
  year = {2022},
  volume = {16},
  number = {2},
  pages = {4926--4981},
  publisher = {The Institute of Mathematical Statistics and the Bernoulli Society}
}

@article{benjamini2001control,
  title = {The Control of the False Discovery Rate in Multiple Testing under Dependency},
  author = {Benjamini, Yoav and Yekutieli, Daniel},
  journal = {The Annals of Statistics},
  year = {2001},
  volume = {29},
  number = {4},
  pages = {1165--1188}
}

@article{gazin2025selecting,
  title = {Selecting Informative Conformal Prediction Sets with False Coverage Rate Control},
  author = {Gazin, Ulysse and Heller, Ruth and Marandon, Ariane and Roquain, Etienne},
  journal = {Journal of the Royal Statistical Society, Series B},
  year = {2025},
  volume = {87},
  number = {4},
  pages = {909--929}
}

@article{jin2023modelfree,
  title = {Model-free selective inference under covariate shift via weighted conformal p-values},
  author = {Jin, Ying and Cand{\`e}s, Emmanuel J.},
  journal = {Biometrika},
  year = {2025},
  volume = {113},
  number = {1},
  pages = {asaf066},
  doi = {10.1093/biomet/asaf066}
}

@article{jin2023selection,
  title = {Selection by prediction with conformal p-values},
  author = {Jin, Ying and Cand{\`e}s, Emmanuel J},
  journal = {Journal of Machine Learning Research},
  year = {2023},
  volume = {24},
  number = {244},
  pages = {1--41}
}

@article{bao2024selective,
  title = {Selective conformal inference with false coverage-statement rate control},
  author = {Bao, Yajie and Huo, Yuyang and Ren, Haojie and Zou, Changliang},
  journal = {Biometrika},
  year = {2024},
  volume = {111},
  number = {3},
  pages = {727--742},
  publisher = {Oxford University Press}
}

@article{sun2007oracle,
  title = {Oracle and adaptive compound decision rules for false discovery rate control},
  author = {Sun, Wenguang and Cai, T Tony},
  journal = {Journal of the American Statistical Association},
  year = {2007},
  volume = {102},
  number = {479},
  pages = {901--912},
  publisher = {Taylor \& Francis}
}

@article{benjamini2005false,
  title = {False discovery rate--adjusted multiple confidence intervals for selected parameters},
  author = {Benjamini, Yoav and Yekutieli, Daniel},
  journal = {Journal of the American Statistical Association},
  year = {2005},
  volume = {100},
  number = {469},
  pages = {71--81},
  publisher = {Taylor \& Francis}
}

@article{jin2025confidence,
  title = {Confidence on the focal: Conformal prediction with selection-conditional coverage},
  author = {Jin, Ying and Ren, Zhimei},
  journal = {Journal of the Royal Statistical Society, Series B},
  year = {2025},
  volume = {87},
  number = {4},
  pages = {1239--1259},
  publisher = {Oxford University Press UK}
}

@article{lei2018distribution,
  title = {Distribution-free predictive inference for regression},
  author = {Lei, Jing and G'Sell, Max and Rinaldo, Alessandro and Tibshirani, Ryan J and Wasserman, Larry},
  journal = {Journal of the American Statistical Association},
  year = {2018},
  volume = {113},
  number = {523},
  pages = {1094--1111},
  publisher = {Taylor \& Francis}
}

@book{vovk2005algorithmic,
  title = {Algorithmic learning in a random world},
  author = {Vovk, Vladimir and Gammerman, Alexander and Shafer, Glenn},
  year = {2005},
  volume = {29},
  publisher = {Springer}
}

@article{guan2022prediction,
  title = {Prediction and outlier detection in classification problems},
  author = {Guan, Leying and Tibshirani, Robert},
  journal = {Journal of the Royal Statistical Society, Series B},
  year = {2022},
  volume = {84},
  number = {2},
  pages = {524--546},
  publisher = {Oxford University Press}
}

@book{boucheron2013concentration,
  title = {Concentration Inequalities: A Nonasymptotic Theory of Independence},
  author = {Boucheron, Stéphane and Lugosi, Gábor and Massart, Pascal},
  year = {2013},
  publisher = {Oxford University Press}
}

@article{gui2024conformal,
  title = {Conformal alignment: Knowing when to trust foundation models with guarantees},
  author = {Gui, Yu and Jin, Ying and Ren, Zhimei},
  journal = {Advances in Neural Information Processing Systems},
  year = {2024},
  volume = {37},
  pages = {73884--73919}
}

@article{wu2024optimal,
  title = {Optimal Subsampling via Predictive Inference},
  author = {Wu, Xiaoyang and Huo, Yuyang and Ren, Haojie and Zou, Changliang},
  journal = {Journal of the American Statistical Association},
  year = {2024},
  volume = {119},
  number = {548},
  pages = {2844--2856}
}

@article{shaffer2002multiplicity,
  title = {Multiplicity, directional (type III) errors, and the null hypothesis},
  author = {Shaffer, J. P.},
  journal = {Psychological Methods},
  year = {2002},
  volume = {7},
  number = {3},
  pages = {356--369}
}

@article{stephens2016false,
  title = {False discovery rates: a new deal},
  author = {Stephens, Matthew},
  journal = {Biostatistics},
  year = {2017},
  volume = {18},
  number = {2},
  pages = {275--294}
}

@article{liang2024integrative,
  title = {Integrative conformal p-values for out-of-distribution testing with labelled outliers},
  author = {Liang, Ziyi and Sesia, Matteo and Sun, Wenguang},
  journal = {Journal of the Royal Statistical Society, Series B},
  year = {2024},
  volume = {86},
  number = {3},
  pages = {671--693}
}

@article{barber2019aknockoff,
  title = {A knockoff filter for high-dimensional selective inference},
  author = {Barber, Rina Foygel and Cand{\`e}s, Emmanuel},
  journal = {The Annals of Statistics},
  year = {2019},
  volume = {47},
  number = {5},
  pages = {2504 -- 2537}
}

@article{tukey1991thephilosophy,
  title = {The Philosophy of Multiple Comparisons},
  author = {Tukey, John W.},
  journal = {Statistical Science},
  year = {1991},
  volume = {6},
  number = {1},
  pages = {100--116},
  publisher = {Institute of Mathematical Statistics}
}

@article{sun2012multiple,
  title = {Multiple Testing of Composite Null Hypotheses in Heteroscedastic Models},
  author = {Sun, Wenguang and McLain, Alexander C.},
  journal = {Journal of the American Statistical Association},
  year = {2012},
  volume = {107},
  number = {498},
  pages = {673--687}
}

@article{gang2025large,
  title = {Large-scale multiple testing of composite null hypotheses under heteroskedasticity},
  author = {Gang, Bowen and Banerjee, Trambak},
  journal = {Biometrika},
  year = {2025},
  volume = {112},
  number = {2},
  pages = {asaf007},
  publisher = {Oxford University Press}
}

@article{olsson2022estimating,
  title = {Estimating diagnostic uncertainty in artificial intelligence assisted pathology using conformal prediction},
  author = {Olsson, Henrik and Kartasalo, Kimmo and Mulliqi, Nita and Capuccini, Marco and Ruusuvuori, Pekka and Samaratunga, Hemamali and Delahunt, Brett and Lindskog, Cecilia and Janssen, Emiel AM and Blilie, Anders and others},
  journal = {Nature communications},
  year = {2022},
  volume = {13},
  number = {1},
  pages = {7761}
}

@article{shehu2016adaptive,
  title = {An adaptive personnel selection model for recruitment using domain-driven data mining},
  author = {Shehu, Muhammad Ahmad and Saeed, Faisal},
  journal = {Journal of Theoretical and Applied Information Technology},
  year = {2016},
  volume = {91},
  number = {1},
  pages = {117--130}
}

@article{laghuvarapu2023codrug,
  title = {Codrug: Conformal drug property prediction with density estimation under covariate shift},
  author = {Laghuvarapu, Siddhartha and Lin, Zhen and Sun, Jimeng},
  journal = {Advances in Neural Information Processing Systems},
  year = {2023},
  volume = {36},
  pages = {37728--37747}
}

@article{ahmed2016survey,
  title = {A survey of anomaly detection techniques in financial domain},
  author = {Ahmed, Mohiuddin and Mahmood, Abdun Naser and Islam, Md Rafiqul},
  journal = {Future Generation Computer Systems},
  year = {2016},
  volume = {55},
  pages = {278--288},
  publisher = {Elsevier}
}

@article{weinstein2013selection,
  title = {Selection adjusted confidence intervals with more power to determine the sign},
  author = {Weinstein, Asaf and Fithian, William and Benjamini, Yoav},
  journal = {Journal of the American Statistical Association},
  year = {2013},
  volume = {108},
  number = {501},
  pages = {165--176}
}

@article{efron2001empirical,
  title = {Empirical Bayes Analysis of a Microarray Experiment},
  author = {Efron, Bradley and Tibshirani, Robert and Storey, John D and Tusher, Virginia},
  journal = {Journal of the American Statistical Association},
  year = {2001},
  volume = {96},
  number = {456},
  pages = {1151--1160}
}

@article{weinstein2020selective,
  title = {Selective sign-determining multiple confidence intervals with FCR control},
  author = {Weinstein, Asaf and Yekutieli, Daniel},
  journal = {Statistica Sinica},
  year = {2020},
  volume = {30},
  number = {1},
  pages = {531--555},
  publisher = {JSTOR}
}

@inproceedings{weinstein2020online,
  title = {Online control of the false coverage rate and false sign rate},
  author = {Weinstein, Asaf and Ramdas, Aaditya},
  year = {2020},
  pages = {10193--10202},
  booktitle = {International Conference on Machine Learning},
  organization = {PMLR}
}

@article{rava2021burden,
  title = {A Burden Shared is a Burden Halved: A Fairness-Adjusted Approach to Classification},
  author = {Rava, Bradley and Sun, Wenguang and James, Gareth M. and Tong, Xin},
  journal = {Journal of the American Statistical Association},
  year = {2026},
  volume = {0},
  number = {ja},
  pages = {1--24}
}

@article{davis2011comprehensive,
  title = {Comprehensive analysis of kinase inhibitor selectivity},
  author = {Davis, Mindy I and Hunt, Jeremy P and Herrgard, Sanna and Ciceri, Pietro and Wodicka, Lisa M and Pallares, Gabriel and Hocker, Michael and Treiber, Daniel K and Zarrinkar, Patrick P},
  journal = {Nature biotechnology},
  year = {2011},
  volume = {29},
  number = {11},
  pages = {1046--1051},
  publisher = {Nature Publishing Group US New York}
}

@article{huang2020deeppurpose,
  title = {DeepPurpose: a deep learning library for drug--target interaction prediction},
  author = {Huang, Kexin and Fu, Tianfan and Glass, Lucas M and Zitnik, Marinka and Xiao, Cao and Sun, Jimeng},
  journal = {Bioinformatics},
  year = {2020},
  volume = {36},
  number = {22--23},
  pages = {5545--5547}
}

@article{heller2020optimal,
  title = {Optimal Control of False Discovery Criteria in the Two-Group Model},
  author = {Heller, Ruth and Rosset, Saharon},
  journal = {Journal of the Royal Statistical Society, Series B},
  year = {2021},
  volume = {83},
  number = {1},
  pages = {133--155}
}

@article{vamathevan2019applications,
  title = {Applications of machine learning in drug discovery and development},
  author = {Vamathevan, Jessica and Clark, Dominic and Czodrowski, Paul and Dunham, Ian and Ferran, Edgardo and Lee, George and others},
  journal = {Nature reviews Drug discovery},
  year = {2019},
  volume = {18},
  number = {6},
  pages = {463--477}
}

@book{huang2007drug,
  title = {Drug discovery research: new frontiers in the post-genomic era},
  author = {Huang, Ziwei},
  year = {2007},
  publisher = {John Wiley \& Sons}
}

@article{blanchard2009adaptive,
  title = {Adaptive False Discovery Rate Control under Independence and Dependence},
  author = {Blanchard, Gilles and Roquain, \'{E}tienne},
  journal = {Journal of Machine Learning Research},
  year = {2009},
  volume = {10},
  pages = {2837–2871}
}

@misc{levitskaya2023how,
    author       = {Viktoria Levitskaya},              
    title        = {How to boost business decisions with conformal prediction and confidence},     
    howpublished = {\url{https://redfield.ai/conformal-prediction-for-business/}}, 
    note         = {Accessed: 2023-10-01},     
    year         = {2023}                      
}

\doublespace

\newpage
\renewcommand{\thesection}{S.\arabic{section}}
\renewcommand{\theequation}{S.\arabic{equation}}
\renewcommand{\thefigure}{S.\arabic{figure}}
\renewcommand{\thetable}{S.\arabic{table}}
\renewcommand{\thealgorithm}{S.\arabic{algorithm}}
\renewcommand{\thetheorem}{S.\arabic{theorem}}
\renewcommand{\thedefinition}{S.\arabic{definition}}
\renewcommand{\theexample}{S.\arabic{example}}
\renewcommand{\thelemma}{S.\arabic{lemma}}
\renewcommand{\theassumption}{S.\arabic{assumption}}
\renewcommand{\theremark}{S.\arabic{remark}}

\setcounter{section}{0}
\setcounter{equation}{0}
\setcounter{table}{0}
\setcounter{figure}{0}

\setcounter{algorithm}{0}

\setcounter{theorem}{0}
\setcounter{definition}{0}
\setcounter{example}{0}
\setcounter{lemma}{0}
\setcounter{assumption}{0}
\setcounter{remark}{0}

\begin{center}\bf\large
	Online Supplementary Material for \\ ``Conformalized Large-Scale Selective Inference with Informative and Trustworthy Prediction Sets''
\end{center}

The supplement contains a discussion of related work (Section~\ref{asec:literaturereview}), additional methodological details (Section~\ref{asec:classifieronpu}), proofs of the equivalence of selective inference methods (Section~\ref{asec:equivproof}), proofs of the main theoretical results (Section~\ref{asec:technicalproofs}), and supplementary numerical results (Section~\ref{asec:additionalnumericalresults}).

% \begin{appendices}
\appendix

\section{Related Work}\label{asec:literaturereview}

In the context of informative selective inference, the efficacy of a method can be assessed along four key dimensions: 
(a) finite-sample validity, which guarantees that selectively reported prediction sets are statistically trustworthy; 
(b) anti-conservativeness, which ensures the FCR budget is asymptotically fully utilized; 
(c) informativeness, which guarantees that all reported prediction sets are scientifically meaningful; and 
(d) score adaptivity, which enables full exploitation of the available data. 
Table~\ref{tab:comparisonwithexistingwork} provides a comparative summary of existing methods and our proposed \scip framework along these dimensions.

\begin{table}[htb]
    \caption{
        Comparisons between \scip and various related methods based on key properties.
        \byfcr: procedure proposed by \citet{benjamini2005false};
        \texttt{Naive}: selectively reporting the informative prediction sets without adjustment;
        \texttt{SCOP}: adjusting the selected prediction sets using \texttt{SCOP} proposed by \citet{bao2024selective};
        \texttt{JOMI}: adjusting the selected prediction sets using \texttt{JOMI} proposed by \citet{jin2025confidence};
        \infosp, \infoscop: procedures for informative prediction proposed by \citet{gazin2025selecting};
        \cfbh: procedure for conformal selection proposed by \citet{jin2023selection};
        \texttt{Zhao-Su}: procedure for selective classification proposed by \citet{zhao2023controlling}.
        }
    \centering
    \small
    \renewcommand{\arraystretch}{1.2}

\begin{tabular}{|l|c|c|c|c|c|}
\hline

\multirow{2}{*}{Method} & 
\multirow{2}{*}{\makecell{\small Finite-Sample \\ Validity}} & 
\multirow{2}{*}{\makecell{\small Anti- \\ Conservativeness}} & 
\multicolumn{2}{c|}{\small Incorporating Informativeness} & 
\multirow{2}{*}{\makecell{\small Score \\ Adaptivity}} \\

\cline{4-5}

 & & & Regression & Classification & \\ 
\hline

\byfcr          & \yesgreen & \nonopink & \nonopink & \nonopink & \nonopink \\
\hline
\texttt{Naive}  & \nonopink & \nonopink & \yesgreen & \yesgreen & \nonopink \\
\hline
\texttt{SCOP}   & \yesgreen & \yesgreen & \nonopink & \nonopink & \nonopink \\
\hline
\texttt{JOMI}   & \yesgreen & \yesgreen & \nonopink & \nonopink & \nonopink \\
\hline
\infosp         & \yesgreen & \nonopink & \yesgreen & \yesgreen & \nonopink \\
\hline
\infoscop       & \yesgreen & \nonopink & \yesgreen & \yesgreen & \nonopink \\
\hline
\cfbh           & \yesgreen & \nonopink & \partyelw & \nonopink & \yesgreen \\
\hline
{\tt Zhao-Su}   & \yesgreen & \yesgreen & \nonopink & \partyelw & \yesgreen \\
\hline
{\color{red} $\bigstar$ \scip}          & \yesgreen & \yesgreen & \yesgreen & \yesgreen & \yesgreen \\
\hline
\end{tabular}
    \label{tab:comparisonwithexistingwork}
\end{table}

The remainder of this section is organized as follows. Section~\ref{assec:overview} presents a literature review of related work. Section~\ref{assec:anotherlookinfosp} provides a detailed comparison between \scip and \infosp, highlighting their differences and clarifying our contributions.

\subsection{Literature review}\label{assec:overview}

First, the seminal work of \cite{benjamini2005false} introduced the concept of FCR and highlighted a potential contradiction in the selection-adjustment paradigm: 
parameters initially deemed ``interesting'' during selection may lose their significance after post-selection adjustment, as their confidence intervals may include ``uninteresting'' values.
They addressed this issue by applying the \bhfdr procedure at selection stage and the \byfcr procedure in adjusting the selected confidence intervals.
This combination ensures both FCR control and consistency throughout the decision process.
However, \byfcr is shown to produce conservative empirical FCR levels (cf. \citealp{weinstein2013selection}) and does not explicitly account for informativeness constraints. 

Second, informativeness notion has been explored in classical selective inference. For example, multiple testing with composite nulls \citep{sun2012multiple, gang2025large} involves pre-specified indifference regions, where parameters within certain ranges are considered scientifically uninteresting, closely aligning with the concept of informativeness. Multiple testing with confident directions \citep{tukey1991thephilosophy, stephens2016false}, knockoff inference with directional FDR control \citep{barber2019aknockoff}, and sign-determining and localized selective confidence intervals \citep{weinstein2013selection, weinstein2020selective, weinstein2020online} all reflect the notion of informativeness, which explicitly requires reporting only non-null effects that are confidently positive or negative. However, these approaches are not directly applicable to modern conformal settings.

Third, our work advances the field of selective conformal prediction, which constructs prediction sets for selected test units.
The {\tt SCOP} method \citep{bao2024selective} introduced a post-selection calibration strategy to mitigate the conservativeness of \byfcr, but it limits to threshold-based selection rules and cannot accommodate general informativeness constraints.
The {\tt JOMI} method \citep{jin2025confidence} overcomes this limitation by building the post-selection calibration set via swapping technique.
Despite its elegant framework, it still faces the challenge that adjusted prediction sets may violate informativeness constraints (see the illustrative example in Section~\ref{sec:illustrativeexample}).
The \infosp method \citep{gazin2025selecting} resolves this issue within the \byfcr framework. 
Nevertheless, it suffers from the inefficiency in constructing the informative prediction sets and the conservativeness in controlling the FCR, yielding lower statistical power in practice.
In contrast, \scip offers a novel selection paradigm, directly excluding prediction sets that are either uninformative or prone to under-coverage.
It ensures both informativeness and trustworthiness while exhibiting higher power relative to \infosp.

Finally, the formulation of FCR control with informativeness constraints provides a unifying framework (cf. Section~\ref{sec:infofcrframework}) that encompasses various selective inference tasks, including outlier detection \citep{bates2023testing, marandon2024adaptive, liang2024integrative}, selective classification \citep{rava2021burden, zhao2023controlling}, and conformal selection \citep{jin2023selection, jin2023modelfree, gui2024conformal, gui2025acs}.
The \scip framework thus offers a comprehensive solution to these previously fragmented problems.
Particularly, while existing methods are often constrained by pre-specified intervals or categories,
the data-driven informative set constructor and trust-score mechanism offers flexibility to deal with general informativeness requirements and be applicable to both classification and regression tasks.
% \scip leverages data-driven informative set constructors and is applicable to both classification and regression tasks.
It recovers and frequently outperforms existing methods, as demonstrated through various examples in Section~\ref{sec:applicationofinfospp}.

\subsection{Detailed comparison with \infosp}
\label{assec:anotherlookinfosp}

We first briefly introduce the \infosp procedure and then compare it with our \scip (i.e., \infospp) procedure.

Let $\tq_{n+j} = \tq(X_{n+j}; \cD^\ca)$ denote the $\cI$-adjusted $p$-values for $j \in [m]$. 
\infosp uses these $p$-values to rank the test units for selection, with the threshold determined by the \bhfdr procedure:
\begin{equation*}%\label{eq:infosphatk1}
     \hk = \max \bigg\{
        k \in [m]:
        \sum_{j = 1}^m \bbI \big(\tq_{n+j} \le \alpha k / m \big) \ge k
    \bigg\}, 
    \quad \htau_\bh = \alpha \hk / m.
\end{equation*}
\infosp reports the set of intervals $\big\{\cC_{n+j}^{\cp, \htau_\bh}: \tq_{n+j} \le \htau_\bh\big\}$, which ensures both the informativeness constraint and FCR control simultaneously.

Both \infosp and \infospp are designed to handle general informativeness constraints, but \infospp improves upon \infosp in power. This power gain stems from two sources.

First, \scip fully utilizes the error budget (Theorem~\ref{the:fcrcontrolexact}), overcoming the conservativeness of \infosp. This advantage arises from the construction of the \emph{generalized conformal $p$-value}, which is built upon a \emph{trust score} that directly quantifies the likelihood of the non‑covering event $\{Y \notin \cC^\cI(X)\}$. By directly targeting the event that defines a false coverage, the trust score provides a more accurate and adaptive measure for error control, enabling the FCR budget to be nearly exhausted. In contrast, the $\cI$-adjusted $p$-values are derived by inverting the condition $\cC^{\mathrm{CP},q}(X) \in \cI$ and therefore lack a direct probabilistic link to the event of non‑coverage. Consequently, \infosp tends to be more conservative and leaves a substantial portion of the FCR budget unused. This conservativeness of the $\cI$-adjusted $p$-values is illustrated in Figure~\ref{fig:compareqvalpval_copy}.

\begin{figure}[!t]
    \centering
    \includegraphics[width=0.66\linewidth]{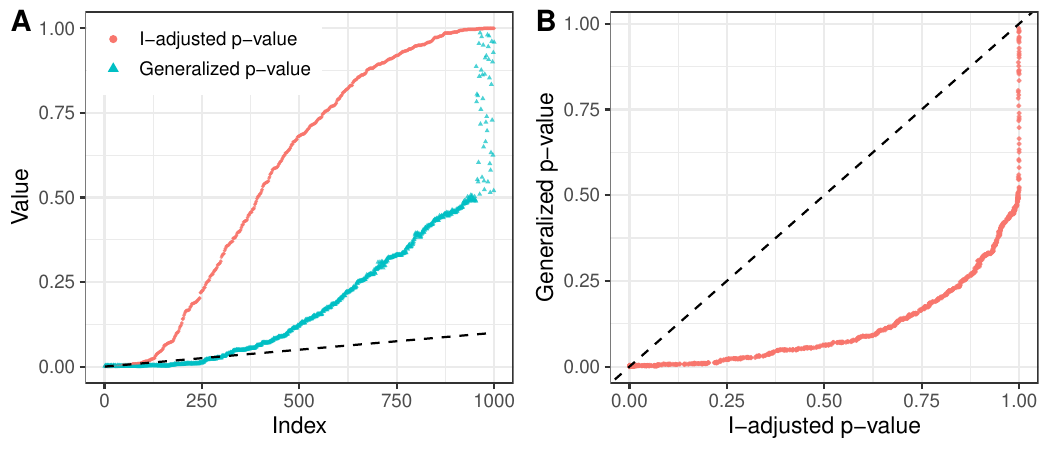}
    \caption{
    \small 
    Comparison of two $p$-values.
    The setting follows Figure~\ref{fig:comparenewparadigm} with $\eta = 1$.
    Black dashed lines in subplot-A and -B are $y = x \alpha/m$ and $y = x$, respectively. }
    \label{fig:compareqvalpval_copy}
\end{figure}

Second, \infosp assigns every reported prediction set the same marginal coverage level, $\alpha k / m$. In contrast, \infospp constructs prediction sets with test‑point‑specific marginal coverage levels by leveraging the truncated $\cI$-adjusted $p$-values. This adaptivity makes it possible to consider many candidate prediction sets that were previously excluded by \infosp. Notably, many of these newly considered candidates exhibit high reliability, allowing them to be selected without inflating the FCR. Consequently, a substantially larger number of informative prediction sets can be reported, which is another key source of the power gain.

\section{Extension: Classifier on Positive Unlabeled Data}\label{asec:classifieronpu}

To fully leverage the available data (or when an independent training set is unavailable), we discuss a method that trains the classifier on the mixed sample using a \emph{positive-unlabeled (PU) learning} approach. Specifically, units in $\{i \in \cD^\ca : Y_i \in \cC(X_i)\}$ are treated as the \emph{positive sample} and assigned the label $\bA_i = 1$, while the remaining units in the mixed sample are treated as the \emph{unlabeled sample} and assigned $\bA_i = -1$. We train the classifier via ERM with the objective function
\begin{equation*}
    \hcL_{\rm pu} (\cD^\mix; g, l)
    = 
    \frac{1}{n+m}
    \sum_{i = 1}^{n+m}
    \lambda \cdot \bbI(\bA_i = 1) \cdot \ell(1,g(X_i))
    + 
    \bbI(\bA_i = -1) \cdot \ell(-1,g(X_i)).
\end{equation*}
Let $g_{\rm pu}^*$ be the minimizer of the population objective function 
$\mathcal{L}_{\rm pu}(g,l) = \mathbb{E}[\hat{\mathcal{L}}_{\rm pu}(\mathcal{D}^{\text{mix}}; g, l)]$ over all measurable functions from $\mathcal{X}$ to $\mathbb{R}$. 
Following Lemma~4.3 of \citet{marandon2024adaptive}, $g_{\rm pu}^*$ is a strictly increasing transformation of $T^{\text{ora}}$ under mild conditions. In practice, we can minimize the empirical objective function to obtain $\hat{g}_{\rm pu}$ as the trust score. 

The permutation-invariance condition in Definition~\ref{def:rigorousper-inv} does not hold for $\hat{g}_{\rm pu}$ because the positive sample is extracted from the calibration data for training. Consequently, FCR control is not straightforward. The following proposition shows that the proposed algorithm maintains FCR control under a slightly stricter condition.

\begin{proposition}\label{prop:finitesamplefcrpulearning}
 Suppose that Assumption~\ref{ass:fullexchangeability} holds, the informative set constructor is permutation-invariant, and $\mathbb{I}\{Y_i \notin \cC(X_i)\}$, $i \in [n+m]$, are independent Bernoulli variables with probability $\pi_0^*$. Then Algorithm~\ref{alg:baseprocedure} with the PU-learned score $\hat{g}_{\rm pu}$ controls the FCR at level $\alpha$.
\end{proposition}

% \color{red}
\section{Equivalence Proofs}\label{asec:equivproof}

\subsection{The self-consistent reformulation of the \bhfdr procedure}\label{subsec:self-cons}

Let $p_1, \dots, p_m$ be a set of p-values. Consider the following two procedures. 
\begin{itemize}
    \item Procedure 1 (\bhfdr, \citealp{benjamini1995controlling}):
    \[
    \hk = \max \big\{ i : p_{(i)} \le \frac{i}{m} \alpha \big\}, 
    \quad
    \mathcal S_1 = \{ j \in [m] : p_j \le p_{(\hk)} \},
    \]
    where $p_{(1)} \le \dots \le p_{(m)}$ are the ordered p-values.

    \item Procedure 2 (self-consistent, \citealp{blanchard2009adaptive}):
    \[
    \widehat{\alpha} = \max \Big\{ a \in [0,1] : \frac{\alpha}{m} \sum_{j=1}^m \bbI(p_j \le a) \ge a \Big\},
    \quad
    \mathcal S_2 = \{ j \in [m] : p_j \le \widehat{\alpha} \}.
    \]
 ``Self-consistent'' means that the threshold and the number of rejections it produces are in agreement: $\widehat{\alpha}=(\alpha/m)r(\widehat{\alpha})$.       
\end{itemize}
We aim to show that $\mathcal S_1 = \mathcal S_2$. 
\begin{remark}
This result is well known and straightforward to establish, yet we could not locate a reference that rigorously documents the equivalence. We therefore include a proof here for completeness.
\end{remark}

\noindent\textbf{Proof of equivalence.} Let $p_{(1)} \le \dots \le p_{(m)}$ be the ordered p-values and define
$ r(a) = \sum_{j=1}^m \bbI(p_j \le a),$
the number of p-values that are less than or equal to $a$.

We first show that $ \widehat{\alpha} = \alpha\hk/m$.
Indeed, since $p_{(\hk)} \le \alpha \hk / m$,
we have $r (\alpha\hk /m ) \ge \hk,$
and therefore $
\alpha / m \times  r (\alpha \hk / m) \ge \alpha \hk / m$.
So $ a = \alpha \hk / m$ is feasible in the definition of $\widehat{\alpha}$, which implies $\widehat{\alpha} \ge \alpha \hk / m$.

Conversely, let $a \in [0,1]$ be any feasible value, and let $k=r(a)$. Then $p_{(k)} \le a$, and feasibility gives
$
a \le (\alpha/m)r(a)=(\alpha/m)k.
$
Hence $p_{(k)} \le \alpha k / m$.
so $k$ satisfies the \bhfdr condition. By maximality of $\hk$, we must have $k \le \hk$. Therefore,
$
a \le (\alpha/m)k \le \alpha \hk / m.
$
Since this holds for every feasible $a$, it follows that
$
\widehat{\alpha} \le \alpha \hk / m.
$
Combining the two inequalities yields $\widehat{\alpha} = \alpha \hk / m$.

It remains to compare the rejection sets. Since $p_{(\hk)} \le \widehat{\alpha}$,
we have $\mathcal S_1 \subseteq \mathcal S_2$. On the other hand, $|\mathcal S_2| = r(\widehat{\alpha}) \le \hk$ 
by the argument above applied to the feasible value $a=\widehat{\alpha}$. Also, $|\mathcal S_1| = \hk$.
Hence $|\mathcal S_2| \le |\mathcal S_1|$, and therefore $\mathcal S_1=\mathcal S_2$. \qed

\subsection{Equivalence between \scip and \cfbh in one-sided conformal selection} 
\label{assec:anotherlookcfbh}

Consider conformal selection with a one‑sided null hypothesis $H_{0,j}: Y_{n+j} \le 0$ \citep{jin2023selection}. 
Let $V^{\downarrow}(x,y)$ be a monotone score such that $V^{\downarrow}(x,y) \le V^{\downarrow}(x,y')$ whenever $y > y'$. The generalized conformal $p$-values are defined as
\begin{equation}\label{eq:conformalpvaluejc}
    p_j^{\text{jc}} = \frac{
        \sum_{i=1}^n \mathbb{I}\bigl(V^{\downarrow}_i > V^{\downarrow}(X_{n+j},0)\bigr) 
        + \bigl\{ 1 + \sum_{i=1}^n \mathbb{I}\bigl(V^{\downarrow}_i = V^{\downarrow}(X_{n+j},0)\bigr) \bigr\} U_j
    }{n+1},
\end{equation}
for $j \in [m]$, where $U_j \sim \text{Unif}(0,1)$ are independent tie‑breaking variables and $V^{\downarrow}_i = V^{\downarrow}(X_i, Y_i)$ for $i \in [n]$.
It can be verified that these $p$-values satisfy
\begin{equation}\label{eq:conditionalconservative}
    \begin{aligned}
        \Pr\bigl(p_j^{\text{jc}} \le \alpha,\; Y_{n+j} \le 0\bigr) 
        & \stackrel{(i)}{\le} 
        \Pr\bigl(p_j^{\text{CP}}(Y_{n+j};V^{\downarrow}) \le \alpha,\; Y_{n+j} \le 0\bigr) \\
        & \stackrel{(ii)}{\le} 
        \Pr\bigl(p_j^{\text{CP}}(Y_{n+j};V^{\downarrow}) \le \alpha\bigr) 
        \stackrel{(iii)}{=} \alpha + O(1/n),
    \end{aligned}
\end{equation}
where step (i) follows from monotonicity, and step (iii) is a standard result of split conformal prediction \citep{angelopoulos2024theoretical}. The \cfbh procedure, which applies \bhfdr to the generalized $p$-values, is summarized in Algorithm~\ref{alg:cfbh}.

\citet{jin2023selection} showed that equipping \cfbh with the clipped score, $V_{\rm clp}(x,y) = \hmu(x)- c_0 - 2 M \times \bbI(y > c_0)$ (with $M \ge \sup_x |\hmu(x)|$ sufficiently large and $\hmu$ being learned independently),
makes \cfbh asymptotically anti-conservative.

The conformal $p$-value in~\eqref{eq:conformalpvaluejc} with the clipped score becomes
\begin{equation*}
    \begin{aligned}
        p_j^{\text{jc}} 
        &= \frac{
            \sum_{i=1}^n \mathbb{I}\{V(X_i,Y_i) > V(X_{n+j},0)\} 
            + \bigl\{ 1 + \mathbb{I}\{V(X_i,Y_i) = V(X_{n+j},0)\} \bigr\} U_j
        }{n+1} \\[4pt]
        &\stackrel{(i)}{=} 
        \frac{
            \sum_{i=1}^n \mathbb{I}(Y_i \le 0) \, \mathbb{I}\{\hat{\mu}(X_i) > \hat{\mu}(X_{n+j})\} 
            + \bigl\{ 1 + \mathbb{I}(Y_i \le 0) \, \mathbb{I}\{\hat{\mu}(X_i) = \hat{\mu}(X_{n+j})\} \bigr\} U_j
        }{n+1},
    \end{aligned}
\end{equation*}
where step $(i)$ follows because $V(X_i,Y_i) < V(X_{n+j},0)$ whenever $Y_i > 0$, provided $M \ge \sup_x |\hat{\mu}(x)|$. 

Recall the generalized conformal $p$-value constructed from the trust score $T(x,\mathcal{C}^{\mathcal{I}}(x)) = \hat{\mu}(x)$, as defined in \eqref{eq:pseudopvalue} of the main text:
\[
p_j^{\ps} = 
    \frac{
        \sum_{i = 1}^n 
        \mathbb{I}\big( Y_i \notin \mathcal{C}_i^{\mathcal{I}},\; T_i > T_{n+j} \big)
        + 
        \big\{ 
            1 + 
            \sum_{i = 1}^n 
            \mathbb{I}\bigl(Y_i \notin \mathcal{C}_i^{\mathcal{I}},\; T_i = T_{n+j}\bigr)
        \bigr\}
        \cdot U_j
    }{1+n}.
\]
One can verify that $p_j^{\text{jc}}$ and $p_j^{\ps}$ coincide exactly, which establishes the equivalence. Notably, our construction follows a principled, generic strategy that uses a trust score to directly capture the non‑covering event, in contrast to the less straightforward, clever but non‑standard derivation in \citet{jin2023selection}.

\begin{algorithm}[!ht]
	%\textsl{}\setstretch{1.2}
	\renewcommand{\algorithmicrequire}{\textbf{Input:}}
	\renewcommand{\algorithmicensure}{\textbf{Output:}}
	\caption{\cfbh proposed by \citep{jin2023selection}.}
	\label{alg:cfbh}
	\begin{algorithmic}[1]
        \REQUIRE 
        Calibration data $\cD^\ca = \{(X_i,Y_i), i \in [n]\}$;
        Test data $\cD^\te = \{X_{n+j}, j \in [m]\}$;
        Monotone non-conformity score $V^{\downarrow} (x,y)$; 
        Target FDR level $\alpha$.
        \STATE
        Compute $V_i^{\downarrow} = V^{\downarrow}(X_i,Y_i)$ for $i \in [n]$ and $V_{n+j}^{\downarrow} = V^{\downarrow}(X_{n+j},0)$ for $j \in [m]$;
        \STATE 
        Calculate $p$-values $p_j^\jc$ as in \eqref{eq:conformalpvaluejc} for $j \in [m]$;
        \STATE 
        Calculate $\hat{\alpha}$ by applying $\alpha$-level \bhfdr procedure on $\{p_j^\jc, j \in [m]\}$ as in \eqref{eq:qbh};
        \ENSURE
        The selection set
        $\cS = \{j\colon  p_j^\jc \le \hat{\alpha}\}$.
	\end{algorithmic}  
\end{algorithm}

\subsection{Equivalence between \scip and procedures in selective classification}\label{subsec:equi-sc}

In this subsection, we first connect the selection set produced by the \bhfdr procedure to that produced by the counting knockoff procedure \citep{weinstein2017power}, and then present the equivalence between \scip and the two procedures in the context of selective classification \citep{rava2021burden,zhao2023controlling}.

\noindent\textbf{1. Connection between \bhfdr and the counting knockoff procedure}.

Fix an informative set constructor $\cC^\cI$ and a trust score $T$. Let $\cC_i^\cI = \cC^\cI(X_i)$ and $T_i = T(X_i, \cC^\cI_i)$ for $i \in [n+m]$. Consider the \emph{deterministic} version of~\eqref{eq:pseudopvalue}, obtained by setting $U_j \equiv 1$:
\begin{equation*}
    p_j^{\text{psdt}} = 
    \frac{
        \sum_{i=1}^n \mathbb{I}\bigl( Y_i \notin \cC_i^\cI,\; T_i \ge T_{n+j} \bigr) + 1
    }{n+1}.
\end{equation*}
Applying the \bhfdr procedure to $\{p_j^{\text{psdt}}\}_{j=1}^m$ yields
\begin{equation*}
    \widehat{\alpha}^{\text{dt}} = \max \bigg\{ a \in [0,1] : \frac{\alpha}{m} \sum_{j=1}^m \mathbb{I}(p_j^{\text{psdt}} \le a) \ge a \bigg\},
\end{equation*}
and the selection set $\cS^{\text{dt}} = \{ j \in [m] : p_j^{\text{psdt}} \le \widehat{\alpha}^{\text{dt}} \}$, with reported prediction sets $\{\cC_{n+j}^\cI : j \in \cS^{\text{dt}}\}$. Since the deterministic generalized conformal $p$-values are pointwise larger than the original randomized ones, this reporting rule still controls the FCR.

Now consider the counting knockoff procedure \citep{weinstein2017power}:
\[
\cS^{\text{ck}} = \{ j \in [m] : T_{n+j} \ge \widehat{\tau}^{\text{ck}} \}, \mbox{ where}
\]
\begin{eqnarray}\label{eq:scipck}
    \widehat{\tau}^{\text{ck}} & = & \min \bigg\{ \tau \in \{T_{n+j}\}_{j=1}^m : \widehat{\text{FDP}}(\tau) \le \alpha \bigg\}; \\
    \widehat{\text{FDP}}(\tau) & = & \frac{1 + \sum_{i=1}^n \mathbb{I}\{Y_i \notin \cC_i^\cI,\; T_i \ge \tau\}}{1 \vee \sum_{j=1}^m \mathbb{I}\{T_{n+j} \ge \tau\}} \cdot \frac{m}{n+1}.
\end{eqnarray}
The next lemma, adapted from Lemma 2.2 of \citet{mary2022semisupervised}, establishes the equivalence of the two selection sets.

\begin{lemma}
    Fix the informative set constructor $\cC^\cI$ and the trust score $T$. Then the selection sets output by the \bhfdr procedure with deterministic $p$-values and by the counting knockoff procedure coincide: $\cS^{\text{dt}} = \cS^{\text{ck}}$.
\end{lemma}

\noindent\textbf{2. Equivalence between  \scip and \texttt{FASI}.}

We first describe a simplified version of \texttt{FASI} (without sensitive attributes), a procedure for selective classification \citep{rava2021burden}. 
\texttt{FASI} uses a score function $S^{y}(x)$ that estimates the probability of feature $x$ belonging to class $y$. The goal is to select test units belonging to a specific class $y_0 \in [K]$ via the following construction:
\begin{equation}\label{eq:fasiqval}
    \widehat{Q}_j^{y_0}
    = 
    \frac{
        1 + \sum_{i = 1}^n \mathbb{I}\{Y_i \ne y_0,\; S^{y_0}(X_i) \ge S^{y_0}(X_{n+j}) \}
    }{
        1 \vee 
        \sum_{j' = 1}^m \mathbb{I}\{S^{y_0}(X_{n+j'}) \ge S^{y_0}(X_{n+j}) \}
    }
    \,
    \frac{m}{n+1}, 
    \qquad 
    j \in [m].
\end{equation}
A test unit $j$ is selected iff $S^{y_0}(X_{n+j}) \ge \widehat{\tau}^{\text{FASI}}$, where 
\begin{equation}\label{eq:fasitau}
    \widehat{\tau}^{\text{FASI}} = \min \big\{ S^{y_0}(X_{n+j}) : \widehat{Q}_j^{y_0} \le \alpha,\; j \in [m] \big\}.
\end{equation}
Denote $\cS^{\text{FASI}} = \{ j \in [m] : S^{y_0}(X_{n+j}) \ge \widehat{\tau}^{\text{FASI}} \}$.
Within the \scip framework, set $\cC^\cI = \{y_0\}$ and $T(x, \cC^\cI(x)) = S^{y_0}(x)$. Then $\widehat{Q}_j^{y_0} = \widehat{\text{FDP}}(S^{y_0}(X_{n+j}))$, and the operation defined in \eqref{eq:fasiqval}–\eqref{eq:fasitau} is equivalent to that in \eqref{eq:scipck}. Consequently, $\cS^{\text{FASI}} = \cS^{\text{dt}}$, with selected units assigned to class $y_0$.

\noindent\textbf{3. Equivalence between \scip and the method of \citet{zhao2023controlling}.}

The procedure proposed by \citet{zhao2023controlling} extends \texttt{FASI} to selection across all $K$ classes. 
Suppose an estimated probability function $\hat{\Pr}$ is available. 
The procedure first computes a preliminary decision $\hat{Y}_i = \arg\max_{y} \hat{\Pr}(Y = y \mid X = X_i)$ and a score $S_i = S(X_i) = \max_{y} \hat{\Pr}(Y = y \mid X = X_i)$ for each $i \in [n+m]$, where $0$ denotes abstention. 
Then, for $j \in [m]$, define
\begin{equation}\label{eq:zsqval}
    \widehat{Q}_j^{\text{ZS}}
    =  
    \frac{
        1 + \sum_{i = 1}^n \mathbb{I}\{Y_i \ne \hat{Y}_i,\; S(X_i) \ge S(X_{n+j}) \}
    }{
        1 \vee 
        \sum_{j' = 1}^m \mathbb{I}\{S(X_{n+j'}) \ge S(X_{n+j}) \}
    }
    \,
    \frac{m}{n+1}.
\end{equation}
The $j$-th test unit is assigned to class $\hat{Y}_{n+j}$ iff $S(X_{n+j}) \ge \widehat{\tau}^{\text{ZS}}$, where 
\begin{equation}\label{eq:zstau}
    \widehat{\tau}^{\text{ZS}} = \min \big\{ S(X_{n+j}) : \widehat{Q}_j^{\text{ZS}} \le \alpha,\; j \in [m] \big\}.
\end{equation}
The reported sets are $\{\hat{Y}_{n+j} : j \in \mathcal{S}^{\text{ZS}}\}$ with $\mathcal{S}^{\text{ZS}} = \{j \in [m] : S(X_{n+j}) \ge \widehat{\tau}^{\text{ZS}}\}$.

Within the \scip framework, set $\mathcal{C}^{\mathcal{I}}(x) = \{\arg\max_{y} \hat{\Pr}(Y = y \mid X = x)\}$ and $T(x, \mathcal{C}^{\mathcal{I}}(x)) = \max_{y} \hat{\Pr}(Y = y \mid X = x)$. Then $\widehat{Q}_j^{\text{ZS}} = \widehat{\text{FDP}}(S(X_{n+j}))$, and the operation in \eqref{eq:zsqval}–\eqref{eq:zstau} is equivalent to that in \eqref{eq:scipck}. Consequently, $\mathcal{S}^{\text{ZS}} = \mathcal{S}^{\text{dt}}$, with the $j$-th unit assigned to class $\hat{Y}_{n+j}$.

\section{Technical Proofs}\label{asec:technicalproofs}

\subsection{Preliminaries}

For a set or sequence of real values $\{a_1,\dots,a_n\}$, we denote by $\{a_1,\dots,a_n\}_{\varsigma}$ the unordered multiset of these values (allowing duplicates), disregarding the original ordering. We assume that the training set is fixed.

Let $\{V_1,\cdots,V_n,\cdots,V_{n+m}\}$ is a sequence of score values.
For $j \in [m]$, the conformal $p$-values are given by
\begin{equation}\label{eq:conformalpvaluegeneral}
    \tp_j  = 
    \frac{
        \sum_{i = 1}^n \bbI (V_i > V_{n+j})
        + 
        U_j \{ 1 + \sum_{i = 1}^n \bbI (V_i = V_{n+j}) \}
    }{
        n+1
    },
\end{equation}
where $U_j \sim \mbox{Unif }(0,1)$ are independent uniform variables to break ties. The following proposition generalizes Theorem A.1 in \cite{marandon2024adaptive} by allowing ties. 

\begin{proposition}\label{prop:keypropertyconformalpvalue}
    Fix some $j \in [m]$.
    Suppose that $\{V_1,\cdots,V_n,V_{n+j}\}$ is exchangeable conditionally on $\{V_{n+j'}: j' \in [m], j' \ne j\}$. Define an algebra
    \begin{equation*}
        W_j = \bigg\{ 
            \{V_1, \cdots, V_n, V_{n+j}\}_{\varsigma},
            \{V_{n+j'}\colon j' \in [m], j' \ne j\},
            \{U_{j'}\colon j' \in [m], j' \ne j\}
        \bigg\}.
    \end{equation*}
    We further denote $v_{(1)}^j > \cdots > v_{(L^j)}^j$ to be the values without duplicate corresponding to $\{V_1,\cdots,V_n,V_{n+j}\}_{\varsigma}$, $s_l^j = |\{i \in [n]\cup \{n+j\}: V_i = v_{(l)}^j\}|, l \in [L^j]$, 
    and $s_0^j = 0$. In scenarios without ties, we have $L^j = n+1$ and $s_l^j = 1$ for $l \in [L^j]$. 
    Define $l_j$ as the index such that $V_{n+j} = v_{(l_j)}^j$.
    Then, it follows that
    \begin{itemize}
        \item [(1)] $\tp_j$ is measurable with respect to $\{U_j, l_j , W_j\}$;
        \item [(2)] $\tp_j$ is independent of $W_j$ and $\tp_j \sim \mbox{Unif }(0,1)$;
        \item [(3)] The pair $(U_j, l_j)$ is measurable with respect to $\{\tp_j, W_j\}$;
        \item [(4)] For $j' \ne j$, $\tp_{j'}$ is measurable with respect to $\{W_j,\tp_j\}$; 
        In addition, given $W_j$, $\tp_{j'}$ is non-decreasing with respect to $\tp_j$.
        \item [(5)] For any non-decreasing set $D \subseteq [0,1]^m$, the function $\pr \{(\tp_1, \cdots,\tp_m) \in D \mid \tp_j = t \}$ is non-decreasing with respect to $t$.
    \end{itemize}
\end{proposition}

\emph{Proof of the proposition.} For the simplicity of notation, we omit the superscript $j$ on $s_l^j$, $L^j$, and $v_{(l)}^j$ in the subsequent proof.
    
\textbf{Proof of (1)}:
    The conformal $p$-value $\tp_j$ can be expressed as
    \begin{equation*}
        \tp_j = 
        \frac{
            \sum_{l = 0}^{l_j - 1} s_l 
            + 
            U_j s_{l_j}
        }{
            n+1
        }.
    \end{equation*}
    Since $\{s_l, l \in [L]\}$ is measurable with respect to $W_j$, 
    the conformal $p$-value $\tp_j$ is measurable with respect to $\{U_j, l_j , W_j\}$.

    \textbf{Proof of (2)}:
    As the elements in $\{V_1,\cdots,V_n,V_{n+j}\}$ are exchangeable conditional on $\{V_{n+j'}, j' \in [m], j' \ne j\}$ and the uniform variables $\{U_{j'},j' \in [m], j' \ne j\}$ are independent of the scores, it follows from the arguments in Proposition 2.2 of \cite{angelopoulos2024theoretical} that
    \begin{equation}\label{eq:uniformprobability}
        \pr ( l_j = l \mid W_j) = 
        \frac{
            s_l
        }{
            n+1 
        }, \quad 
        \forall \, l \in [L].
    \end{equation}
    
For $t \in (0,1) $, we define $l_t^* = \max\{l \in [L] \cup \{0\} \colon \sum_{l' = 0}^{l} s_{l'} < (n+1) \, t\}$. By the definition of $l_t^*$, it follows that $\sum_{l = 0}^{l_t^*} s_l < (n+1) \, t $ and 
    $\sum_{l = 0}^{l_t^*+1} s_l \ge (n+1) \, t $.
    Then, for any $t \in [0,1]$ we have
    \begin{equation}\label{eq:uniformempiricalpvalue}
        \begin{aligned}
            \pr \big\{ \tp_j \le t \bigmid W_j \big\}
            = &\,
            \pr \big\{ l_j \le l_t^* \bigmid W_j \big\}
            + 
            \pr \big\{ l_j = l_t^* + 1, s_{l_t^* + 1} U_j \le (n+1) \, t - \sum_{l = 0}^{l_t^*} s_l \bigmid W_j \big\}
            \\ = &\,
            \frac{
                \sum_{l = 0}^{l_t^*} s_l
            }{
                n+1
            }
            + 
            \frac{
                (n+1) \, t - \sum_{l = 0}^{l_t^*} s_l
            }{
                s_{l_t^* + 1}
            }
            \times \frac{
                s_{l_t^* + 1}
            }{
                n+1
            }
            = t,
        \end{aligned}
    \end{equation}
    which implies that $\tp_j \mid W_j \sim \mbox{Unif }(0,1)$. Hence $\tp_j$ is marginally uniform and independent of $W_j$.

    \textbf{Proof of (3)}:
    Since $U_j \in (0,1)$, we can recover $(l_j,U_j)$ from $(W_j,\tp_j)$ by 
    \begin{equation*}
        \begin{aligned}
            l_j = & \, 
            \max \{l \in [L] \colon \sum_{l' = 0}^{l-1} s_{l'} < (n+1) \tp_j \},
            \\ 
            U_j = & \, 
            \frac{
                (n+1) \tp_j - \sum_{l = 0}^{l_j-1} s_l
            }{
                s_{l_j}
            }.
        \end{aligned}
    \end{equation*}
    In addition, the first equality implies that $l_j$ is non-decreasing in $\tp_j$ given $W_j$.

    \textbf{Proof of (4)}:
    For any $j' \ne j$, the conformal $p$-value $\tp_{j'}$ can be expressed as
    \begin{equation*}
        \begin{aligned}
            (n+1) \tp_{j'} 
            = 
            \sum_{v \in \{V_1,\cdots,V_n,V_{n+j}\}_{\varsigma}}
            \Big\{ 
                \bbI (v > V_{n+j'})
                + U_{j'} \cdot \bbI (v=V_{n+j'})
            \Big\}
            + U_{j'} & 
            \\ - 
            \bbI (v_{(l_j)} > V_{n+j'})
            -
            U_{j'} \cdot \bbI (v_{(l_j)} = V_{n+j'})
            & .
        \end{aligned}
    \end{equation*}
    Thus, it is measurable with respect to $\{l_j,W_j\}$.
    Provided that $l_j$ is measurable with respect to $\{W_j , \tp_j\}$, 
    we have $\tp_{j'}$ is measurable with respect to $\{W_j , \tp_j\}$.

    Furthermore, since 
    \begin{equation*}
        \bbI (v_{(l_j)} > V_{n+j'})
        +
        U_{j'} \cdot \bbI (v_{(l_j)} = V_{n+j'})
        = 
        (1-U_{j'}) \bbI (v_{(l_j)} > V_{n+j'})
        + 
        U_{j'} \bbI (v_{(l_j)} \ge V_{n+j'})
    \end{equation*}
    is non-increasing in $l_j$, then $\tp_{j'}$ is non-decreasing in $l_j$ given $W_j$.
Since $l_j$ is non-decreasing in $\tp_j$ given $W_j$, we conclude that $\tp_{j'}$ is non-decreasing in $\tp_j$ given $W_j$.

    \textbf{Proof of (5)}:
    We have established that $\tp_j \perp W_j$ in the second property,
    which implies that
    \begin{equation*}
        \pr \{(\tp_1, \cdots,\tp_m) \in D \mid \tp_j = t \}
        = 
        \E_{W_j} \big[ 
            \pr \{(\tp_1, \cdots,\tp_m) \in D \mid \tp_j = t, W_j \}
        \big].
    \end{equation*}
    Provided that for any $j' \in [m]$, $\tp_{j'}$ is measurable with respect to $(W_j, \tp_j)$ and is non-decreasing with respect to $\tp_j$ given $W_j$,
    it follows that 
    $\pr \{(\tp_1, \cdots,\tp_m) \in D \mid \tp_j = t, W_j \}$ is non-decreasing in $t$. 
 The proof is complete by integration. \qed
 
We provide a lemma which characterizes the property of the oracle conformal $p$-values with exchangeable data and permutation-invariant score function. The first property was established in Lemma 3.2 of \cite{marandon2024adaptive}. The second property is a direct corollary of (2) in Proposition~\ref{prop:keypropertyconformalpvalue}.
\begin{lemma}\label{lem:conditionalexchangeability}
    Let $V_i = V(X_i,Y_i)$ for $i \in [n+m]$.
    If Assumption~\ref{ass:fullexchangeability} holds and the score $V$ is permutation-invariant, it follows that 
    (1) $(V_1,\cdots,V_{n+m})$ is exchangeable conditionally on $\cD^\tr$;
    (2) and the oracle conformal $p$-value $p_j^\cp (Y_{n+j}; V) \sim \mbox{Unif }(0,1)$.
\end{lemma}

\subsection{Proof of Proposition~\ref{the:pseudopvalueproperty}}

Let $\tT(x,y) \coloneq \bbI \{y \notin \cC^\cI(x)\} \cdot T(x,\cC^\cI(x))$.
    Given that the trust score is always positive, the generalized conformal $p$-value, $p_j^\ps$, equals $ p_j^\cp (Y_{n+j}; \tT) $ conditionally on $H_{0,j}$ is true.
    Thus, we have
    \begin{equation}\label{eq:pseudopvaluerelaxationa1}
        \pr (p_j^\ps \le \alpha, H_{0,j}) 
        = 
        \pr \{ p_j^\cp (Y_{n+j}; \tT) \le \alpha, H_{0,j} \}
        \le 
        \pr \{ p_j^\cp (Y_{n+j}; \tT) \le \alpha\}.
    \end{equation}
    % Recall that $\tT(x,y) = \bbI \{y \notin \cC(x)\} \times T(x,\cC(x))$.
    Provided that both $\cC^\cI$ and $T$ are permutation-invariant,
    the composite function $\tT$ is permutation-invariant as well.
    Thus, according to Lemma~\ref{lem:conditionalexchangeability},
    we have $p_j^\cp (Y_{n+j}; \tT) \sim \mbox{Unif }(0,1)$.
    Combining it with \eqref{eq:pseudopvaluerelaxationa1} yields that
    \begin{equation*}
        \pr (p_j^\ps \le \alpha, H_{0,j}) 
        \le 
        \pr \{ p_j^\cp (Y_{n+j}; \tT) \le \alpha\}
        = \alpha,
    \end{equation*}
    which completes the proof.

\subsection{Proof of Proposition~\ref{the:prdspvalue}}\label{ssec:prooftheprdspvalue}

We first introduce the PRDS property \citep{benjamini2001control}:

\begin{definition}
    A sequence of random variables $\{x_j,j \in [m]\}$ is said to be PRDS on a subset $\cH \subseteq [m]$ if,
    for any $j' \in \cH$ and non-decreasing measurable set $D \subseteq [0,1]^m$, 
    the function $u: [0,1] \to \pr \{ (x_j, j\in [m]) \in D \mid x_{j'} = u \}$ is non-decreasing.
    Here, we say a set non-decreasing if for any $x \in D$ and $x' \in [0,1]^m$, 
    $x' > x$ (in the coordinate sense) implies that $x' \in D$.
\end{definition}

Here, we consider three variants of the PRDS property for the sequence of generalized conformal p-values.
Notably, Proposition~\ref{the:prdspvalue} corresponds to statement~(b).

\begin{lemma}\label{lem:prdspvalue}
    Let $\bp^\ps = (p_j^\ps,j \in[m])$ be the sequence of generalized conformal $p$-values and $\bp^\ps_{j \to \oc}$ be the sequence obtained by replacing the $j$-th element of $\bp^\ps$ with the oracle one $p_j^\oc = p_j^\cp (Y_{n+j}; \tT)$. 
    Here, $\tT$ is the score defined in the proof of Proposition~\ref{the:pseudopvalueproperty}. 
Let $n_0$ denote the number of ``null units'' in the calibration data.   
    Under Assumptions~\ref{ass:fullexchangeability} and assuming that both the informative set constructor and the trust score function are permutation-invariant to $\{Z_i\}_{i = 1}^{n+m}$, 
    it follows that
    \begin{enumerate}
        \item [(a)]  The sequence $\{p_j^\cp (Y_{n+j}; \tT), j \in [m]\}$ is PRDS on set $\{1,\cdots, m\}$.
        \item [(b)] The sequence $\bp_{j \to \oc}^\ps$ is PRDS on $\{j\}$.
        \item [(c)] Conditional on $H_{0,j}$ is true and $n_0$,
        the sequence $\bp^\ps$ is PRDS on $\{j\}$.
    \end{enumerate}
\end{lemma}

Part (a) is an extension of Theorem~3.3 in \citet{marandon2024adaptive} that allows for ties.
Part (b) extends Lemma~5 of \citet{jin2023selection}. It characterizes the correlation structure within the sequence of generalized $p$-values under $H_{0,j}$, which facilitates the standard leave‑one‑out technique for establishing finite‑sample FCR control.
Part (c) is useful in the proof of Proposition~\ref{prop:finitesamplefcrpulearning}.

    \textbf{Proof of (a)}:
    The oracle conformal $p$-value $p_j^\cp (Y_{n+j}; \tT)$ coincides with the conformal $p$-value defined in~\eqref{eq:conformalpvaluegeneral} by setting $V_i = \tT(X_i,Y_i)$.
    As both $\cC^\cI$ and $T$ are permutation-invariant,
    the composite function $\tT(x,y) = \bbI (y \notin \cC^\cI (x)) \times T(x, \cC^\cI (x))$ is permutation-invariant as well.
    According to Lemma~\ref{lem:conditionalexchangeability}, it follows that $\{\tT(X_i,Y_i), i \in [n+m]\}$ is exchangeable.
    Thus, directly applying the fifth property of Proposition~\ref{prop:keypropertyconformalpvalue} yields the result.

    \textbf{Proof of (b)}:
    % By setting $V_i = \tT_i, i \in [n] \cup \{n+j\}$ and $V_{n+j'} = T_{n+j'}, j' \in [m]\setminus\{j\}$, 
    The $p$-value sequence $\bp_{j \to {\rm \oc}}^\ps$ equals the sequence of conformal $p$-values built on score values $\{\tT_1,\cdots, \tT_n, T_{n+1}, \cdots, T_{n+j - 1}, \tT_{n+j}, T_{n+j+1}, \cdots, T_{n+m}\}$. 
    % coincides with the conformal $p$-values in \eqref{eq:conformalpvaluegeneral}.
    As both $\tT$ and $T$ are permutation-invariant,  $\{ (T_i,\tT_i), i \in [n+m] \}$ is an exchangeable sequence.
    Consequently, we have
    \begin{equation*}
        \Big\{ \tT_i, i \in [n] \cup \{n+j\} \Big\} 
        \mbox{ is exchangeable conditional on }
        \Big\{ T_{n+j'}, j' \in [m], j' \ne j \Big\}.
    \end{equation*}
 Applying (5) of Proposition~\ref{prop:keypropertyconformalpvalue}, we conclude that  
    $$
    \pr \{ (p_1^\ps, \cdots,p_{j-1}^\ps,p_j^\cp (Y_{n+j}; \tT), p_{j+1}^\cp,\cdots,p_m^\ps) \in D \mid p_j^\cp (Y_{n+j}; \tT) = t \}
    $$
     is non-decreasing in $t$.

    \textbf{Proof of (c)}:
   Define the algebra
    \begin{equation}\label{eq:algebratildewj}
        \tW_j = 
        \Big\{ 
            \{      
                \tT_1,\cdots,\tT_n,\tT_{n+j}
            \}_{\varsigma},
            \{
                T_{n+j'}, j' \in [m], j' \ne j
            \},
            \{
                U_{n+j'}, j' \in [m], j' \ne j
            \}
        \Big\}.
    \end{equation}
Denote $v_{(1)}^j > \cdots > v_{(L^j)}^j$ to be the values without duplicate corresponding to $\{\tT_1,\cdots,\tT_n,\tT_{n+j}\}_{\varsigma}$, $s_l^j = |\{i \in [n]\cup \{n+j\}: V_i = v_{(l)}^j\}|, l \in [L^j]$, 
    and $s_0^j = 0$
    By the definition of $\tT$, $\tT_i > 0$ if $Y_i \notin \cC_i$ and $\tT_i = 0$ otherwise.
    Thus, it follows that $v_{(L)} = 0$ (when $\tT_i > 0$ for all $i\in [n] \cup \{n+j\}$, we retain $v_{(L)} = 0$ but setting $s_L = 0$).
    In addition, it follows that $\pr (l_j = l \mid \tW_j) = s_l / (n+1)$ as in \eqref{eq:uniformprobability}.

    According to the definition of $\tT$, the event $H_{0,j} = \{Y_{n+j} \notin \cC_{n+j}\}$ is equivalent to $\{\tT_{n+j} > 0\}$.
    Thus, it follows that 
    \begin{equation*}
        \begin{aligned}
            \pr \{ l_j = l \mid 
            & \,
            \tW_j, H_{0,j} \} 
            = 
            \pr \{ l_j = l \mid \tW_j, \tT_{n+j} > 0 \} 
            \\ = & \,
            \pr \{ l_j = l, \tT_{n+j} > 0 \mid \tW_j \} / 
            \pr \{ \tT_{n+j} > 0 \mid \tW_j \} 
            = 
            \begin{cases}
                0, & l = L
                \\ 
                \frac{
                    s_l
                }{
                    n- s_L + 1
                }, &  0 < l < L 
            \end{cases}.
        \end{aligned}
    \end{equation*}
    This result also implies that $p_j^\cp (Y_{n+j}; \tT) < (n-s_L + 1)/ (n+1)$ conditional on~$H_{0,j}$.
    
For $t \in (0,(n-s_L + 1)/ (n+1)) $, define $l_t^* = \max\{l \in [L-1] \cup \{0\} \colon \sum_{l' = 0}^{l} s_{l'} < (n+1) \, t\}$. Then 
\begin{equation*}%\label{eq:uniformempiricalpvalue}
        \begin{aligned}
            \pr \big\{ p_j^\ps \le t \bigmid \tW_j, H_{0,j} \big\}
            = &\,
            \pr \big\{ l_j \le l_t^* \bigmid \tW_j, H_{0,j} \big\}
            \\ &  \quad + 
            \pr \big\{ l_j = l_t^* + 1, s_{l_t^* + 1} U_j \le (n+1) \, t - \sum_{l = 0}^{l_t^*} s_l \bigmid \tW_j, H_{0,j} \big\}
            \\ = &\,
            \frac{
                \sum_{l = 0}^{l_t^*} s_l
            }{
                n-s_L+1
            }
            + 
            \frac{
                (n+1) \, t - \sum_{l = 0}^{l_t^*} s_l
            }{
                s_{l_t^* + 1}
            }
            \times \frac{
                s_{l_t^* + 1}
            }{
                n-s_L+1
            }
            = 
            \frac{n+1}{n-s_L+1} t.
        \end{aligned}
    \end{equation*}
    Note that $s_L=n_0$, it follows that
    \begin{equation}\label{eq:shrinkagesuperuniform}
        \forall \, t \in \Big(0,\frac{n- n_0 + 1}{n+1} \Big),
        \quad \quad 
        \pr (p_j^\ps \le t \mid \tW_j, H_{0,j}, n_0 \big\}
        = 
        \frac{n+1}{n-n_0+1} t,
    \end{equation}
Hence $p_j^\ps \perp \tW_j$ and $p_j^\ps \sim U (0,(n-n_0 + 1)/ (n+1))$ conditional on $H_{0,j}$ and $n_0$.

    Then, following the arguments in the proof of Proposition~\ref{prop:keypropertyconformalpvalue},
    we can see that $p_{j'}^\ps, j' \in [m]$ is measurable with respect to $\{\tW_j, p_j^\ps\}$ given $\{H_{0,j}, n_0\}$ and is non-decreasing to $p_j^\ps$.
    Thus, we have
    \begin{equation*}
        \begin{aligned}
            & \, \pr \Big\{ 
                \big( 
                    p_1^\ps,\cdots,p_m^\ps 
                \big) \in D 
                \mid p_j^\ps = t , H_{0,j}, n_0
            \Big\}
            \\ & \, = 
            \E_{\tW_j \mid (H_{0,j},n_0)}
            \Big[
                \pr \Big\{ 
                    \big( 
                        p_1^\ps,\cdots,p_m^\ps 
                    \big) \in D 
                    \mid p_j^\ps = t , H_{0,j}, n_0, \tW_j
                \Big\}
            \Big].
        \end{aligned}
    \end{equation*}
    By integration, it follows that 
    $\pr \{ 
        ( 
            p_1^\ps,\cdots,p_m^\ps 
        ) \in D 
        \mid p_j^\ps = t , H_{0,j}, n_0
    \}$ is non-decreasing to $t$,
    which completes the proof.

\subsection{Proof of Theorem~\ref{the:fcrcontrolfinitesample}}
    Let $\hk = |\cS|$ be the number of selected units.
    Then, FCR can be expressed as
    \begin{equation}\label{eq:fcrreexpress}
        \fcr = \E \bigg\{ 
            \frac{
                \sum_{j = 1}^m \bbI (p_j^\ps \le \alpha \hk/ m , H_{0,j})
            }{
                \hk
            }
        \bigg\}=
        \sum_{j = 1}^m 
        \E \bigg\{ 
            \frac{
                \bbI (p_j^\ps \le \alpha \hk/ m , H_{0,j})
            }{
                \hk
            }
        \bigg\}.
    \end{equation}

    Let $\kappa(p_1,\cdots,p_m)$ be the cardinality of selection set obtained by applying \bhfdr procedure on $p$-value sequence $(p_1,\cdots,p_m)$.
    Then, the selection number $\hk$ can be expressed as $\hk = \kappa (\bp^\ps)$ and 
    \begin{equation}\label{eq:theorem4decomp1}
        \begin{aligned}
            \E \bigg\{ 
                \frac{
                    \bbI (p_j^\ps \le \alpha \hk/ m , H_{0,j})
                }{
                    \hk
                }
            \bigg\}
            = & \,
            \sum_{k = 1}^{m}
            \frac{1}{k} \times
            \pr \big\{ \hk = k, p_j^\ps \le \alpha k / m, H_{0,j} \big\}
            \\ \overset{(i)}{=} & \,
            \sum_{k = 1}^{m}
            \frac{1}{k} \times
            \pr \big\{ \kappa (\bp^\ps_{j \to \oc}) = k, p_j^\cp (Y_{n+j}; \tT) \le \alpha k / m, H_{0,j} \big\}
            \\ \le & \,
            \sum_{k = 1}^{m}
            \frac{1}{k} \times
            \pr \big\{ \kappa (\bp^\ps_{j \to \oc}) = k, p_j^\cp (Y_{n+j}; \tT) \le \alpha k / m \big\}
            \\ \overset{(ii)}{=} & \, 
            \sum_{k = 1}^{m}
            \frac{\alpha}{m} \times
            \pr \big\{ \kappa (\bp^\ps_{j \to \oc}) = k \bigmid p_j^\cp (Y_{n+j}; \tT) \le \alpha k / m \big\},
        \end{aligned}
    \end{equation}
    where the step (i) follows that $p_j^\ps = p_j^\cp (Y_{n+j};\tT)$ conditionally on $H_{0,j}$, 
    and the step (ii) is due to the uniform property of the oracle conformal $p$-value (Lemma~\ref{lem:conditionalexchangeability}).

    Proposition~\ref{the:prdspvalue} shows that $\bp_{j \to \oc}^\ps$ is PRDS on $\{j\}$.  Using standard arguments in \cite{benjamini2001control} or \cite{wang2022elementary}, we can show that
    \begin{equation}\label{eq:prdsresult}
        \sum_{k = 1}^{m}
        \pr \big\{ \kappa (\bp^\ps_{j \to \oc}) = k \bigmid p_j^\cp (Y_{n+j}; \tT) \le \alpha k / m \big\}
        \le 1.
    \end{equation}

    Combining the results in \eqref{eq:fcrreexpress}, \eqref{eq:theorem4decomp1}, and \eqref{eq:prdsresult} yields that
    \begin{equation*}
        \fcr = 
        \sum_{i = 1}^m 
        \E \bigg\{ 
            \frac{
                \bbI (p_j^\ps \le \alpha \hk/ m , H_{0,j})
            }{
                \hk
            }
        \bigg\}
        \le \alpha,
    \end{equation*}
    which completes the proof.

\subsection{Proof of Theorem~\ref{the:fcrcontrolexact}}

We first present a useful lemma that characterizes the distribution of oracle conformal $p$-values under the alternatives. The proof of this lemma is deferred to Section~\ref{asssec:proofofanticonservativepseudopvalue}.

\begin{lemma}\label{lem:anticonservativepseudopvalue}
    Let $\bar{\pi}_0^* = \pi_0^* \times n / (n + 1)$.
    Under Assumptions~\ref{ass:fullexchangeability} to~\ref{ass:independentnullindicator},
    it follows that,
    for any $\alpha \le \bar{\pi}_0^*$,
    $\pr \{ p_j^\cp(Y_{n+j}; \tT) \le \alpha, H_{1,j}\}
    \le
    \alpha \times \exp \{  
        - 2 n (\bar{\pi}_0^* - \alpha)^2
    \}$.
\end{lemma}

   For a fixed $j \in \cH_0$, define the algebra $\tW_j$ as in \eqref{eq:algebratildewj}.
    Conditional on $H_{0,j}$, the $p$-value sequence $\bp^\ps = \bp^\ps_{j \to \oc}$ is measurable with respect to $\{\tW_j,l_j,U_j\}$.
    Thus, the number of selection $\hk = \kappa(\bp^\ps)$ is measurable with respect to $\{\tW_j,l_j,U_j\}$ , conditional on $H_{0,j}$. Let $\bp^\ps = \bp (\tW_j,l_j,U_j)$ and $\hk = \kappa (\tW_j,l_j,U_j)$.
 Denote $\bar{\bp}^\ps = \bp (\tW_j,1,0) = (\barp_1^\ps,\cdots,\barp_m^\ps)$.
    When there is no tie, it can be verified that
    \begin{equation*}
        \forall \, i \in [m], 
        \,\,
        \begin{cases}
            \barp_i^\ps \le p_i^\ps, & \mbox{if}\,\, p_i^\ps \le p_j^\ps, 
            \\ 
            \barp_i^\ps = p_i^\ps, & \mbox{if}\,\, p_i^\ps > p_j^\ps.
        \end{cases}
    \end{equation*}
    Thus, applying Lemma D.6 in \cite{marandon2024adaptive} yields that $\kappa(\tW_j,1,0) = \kappa(\tW_j,l_j,U_j) = \hk$ conditional on $\{H_{0,j}, p_j^\ps \le \alpha \hk / m\}$.

 The expectation in~\eqref{eq:fcrreexpress} can be expressed as
    \begin{equation*}
        \E \bigg\{ 
            \frac{
                \bbI (p_j^\ps \le \alpha \hk/ m , H_{0,j})
            }{
                \hk
            }
        \bigg\}
        = 
        \E_{(\tW_j,U_j)} \bigg[ 
            \E \bigg\{ 
                \frac{
                    \bbI (p_j^\ps \le \alpha \hk/ m , H_{0,j})
                }{
                    \hk
                }
                \bigmid \tW_j,U_j
            \bigg\}
        \bigg].
    \end{equation*}
 As $\hk = \kappa(\tW_j,1,0)$ conditional on $\{H_{0,j}, p_j^\ps \le \alpha \hk / m\}$, 
    it follows that
    \begin{equation*}
        \E \bigg\{ 
            \frac{
                \bbI (p_j^\ps \le \alpha \hk/ m , H_{0,j})
            }{
                \hk
            }
            \bigmid \tW_j,U_j
        \bigg\}
        = 
        \E \bigg\{ 
            \frac{
                \bbI (p_j^\ps \le \alpha \kappa(\tW_j,1,0)/ m , H_{0,j})
            }{
                \kappa(\tW_j,1,0)
            }
            \bigmid \tW_j,U_j
        \bigg\}.
    \end{equation*}
    Since $U_j$ is independent of $\tW_j$, we have
    \begin{equation}\label{eq:theorem4decomp2}
        \begin{aligned}
            & \, \E_{\tW_j,U_j} \bigg[ 
                \E \bigg\{ 
                    \frac{
                        \bbI (p_j^\ps \le \alpha \kappa(\tW_j,1,0)/ m , H_{0,j})
                    }{
                        \kappa(\tW_j,1,0)
                    }
                    \bigmid \tW_j,U_j
                \bigg\}
            \bigg]
            \\ = & \,
            \E_{\tW_j} \bigg[ 
                \E \bigg\{ 
                    \frac{
                        \bbI (p_j^\cp (Y_{n+j};\tT) \le \alpha \kappa(\tW_j,1,0)/ m , H_{0,j})
                    }{
                        \kappa(\tW_j,1,0)
                    }
                    \bigmid \tW_j
                \bigg\}
            \bigg]
            \\ = & \,
            \E_{\tW_j} \bigg[ 
                \E \bigg\{ 
                    \frac{
                        \bbI (p_j^\cp (Y_{n+j};\tT) \le \alpha \kappa(\tW_j,1,0)/ m)
                    }{
                        \kappa(\tW_j,1,0)
                    }
                    \bigmid \tW_j
                \bigg\}
            \bigg]
            \\ & \quad \quad - 
            \E_{\tW_j} \bigg[ 
                \E \bigg\{ 
                    \frac{
                        \bbI (p_j^\cp (Y_{n+j};\tT) \le \alpha \kappa(\tW_j,1,0)/ m , H_{1,j})
                    }{
                        \kappa(\tW_j,1,0)
                    }
                    \bigmid \tW_j
                \bigg\}
            \bigg].
        \end{aligned}
    \end{equation}

    Proposition~\ref{prop:keypropertyconformalpvalue} shows that $p_j^\cp(Y_{n+j};\tT)$ is independent of $\tW_j$,
    implying that
    \begin{equation}\label{eq:theorem4equiv3}
        \begin{aligned}
            & \, \E_{\tW_j} \bigg[ 
                \E \bigg\{ 
                    \frac{
                        \bbI (p_j^\cp (Y_{n+j};\tT) \le \alpha \kappa(\tW_j,1,0)/ m)
                    }{
                        \kappa(\tW_j,1,0)
                    }
                    \bigmid \tW_j
                \bigg\}
            \bigg]
            \\ = & \,
            \E_{\tW_j} \bigg[ 
                \frac{
                    \pr \big\{ 
                        p_j^\cp (Y_{n+j};\tT) \le \alpha \kappa(\tW_j,1,0)/ m
                        \bigmid 
                        \tW_j
                    \big\}
                }{\kappa(\tW_j,1,0)}
            \bigg]
            = 
            \frac{\alpha}{m}.
        \end{aligned}
    \end{equation}

    Let $\bar{\pi}_0^* = \pi_0^* \times n/(n+1)$. Consider the second term in~\eqref{eq:theorem4decomp2}. We have 
    \begin{equation}\label{eq:upperboundalterexpect}
        \begin{aligned}
            & \, \E_{\tW_j} \bigg[ 
                \E \bigg\{ 
                    \frac{
                        \bbI (p_j^\cp (Y_{n+j};\tT) \le \alpha \kappa(\tW_j,1,0)/ m , H_{1,j})
                    }{
                        \kappa(\tW_j,1,0)
                    }
                    \bigmid \tW_j
                \bigg\}
            \bigg]
            \\ \le & \,
            \E_{\tW_j} \bigg[ 
                \E \bigg\{ 
                        \bbI (p_j^\cp (Y_{n+j};\tT) \le \alpha , H_{1,j})
                    \bigmid \tW_j
                \bigg\}
            \bigg]
            \\ = & \,
            \pr (p_j^\cp (Y_{n+j};\tT) \le \alpha , H_{1,j})
            \le \alpha \times \exp\{-2n (\bar{\pi}_0^* - \alpha)^2\},
        \end{aligned}
    \end{equation}
    where the last inequality follows from Lemma~\ref{lem:anticonservativepseudopvalue}.

The proof is complete by combining the results in~\eqref{eq:theorem4decomp2}, \eqref{eq:theorem4equiv3}, and \eqref{eq:upperboundalterexpect}:
    \begin{equation*}
        \fcr \ge \alpha - m\alpha \times \exp\{-2n (\bar{\pi}_0^* - \alpha)^2\}. \qed
    \end{equation*}

\subsection{Proof of Proposition~\ref{prop:asymptoticdominance}}

\subsubsection{The validity of the modified \infospp}\label{ssec:modifiedinfospp}
We refer to the proposed algorithm with $\{U_j,j \in [m]\}$ 
substituted by shared uniform variable $U \sim \mbox{Unif }(0,1)$
as the \textit{homogeneous} algorithm.
The proposition below presents that
the homogeneous algorithm is valid for FCR control.

\begin{proposition}
     Under Assumptions~\ref{ass:fullexchangeability}, and assuming that both the informative set constructor and the trust score function are permutation-invariant to $\{Z_i\}_{i = 1}^{n+m}$,
     it follows that the homogeneous algorithm controls FCR at level $\alpha$.
\end{proposition}

For $j \in [m]$, define 
    \begin{eqnarray*}
        p_j^\pshm &= & 
        \frac{
            \sum_{i = 1}^n 
            \bbI (Y_i \notin \cC_i, T_i > T_{n+j})
            + 
            \{
                1 + 
                \sum_{i = 1}^n 
                \bbI (Y_i \notin \cC_i, T_i = T_{n+j})
            \} 
            \cdot
            U
        }{  
            1 + n 
        }; \\
        p_j^\cphm (y; \tT)
       & = & 
        \frac{
            \sum_{i = 1}^n
            \bbI \{\tT_i > \tT(X_{n+j},y)\}
            + 
            \big[
                1 + \sum_{i = 1}^n
                \bbI \{\tT_i = \tT(X_{n+j},y)\}
            \big]
            \cdot U
        }{
            1 + n
        }
 \end{eqnarray*}
    be the generalized $p$-values and conformal $p$-values in homogeneous setting, respectively.
    Let $\bp^\pshm = (p_1^\pshm, \cdots,p_m^\pshm)$.
    Denote $\bp^\pshm_{j \to \oc}$ as the homogeneous generalized $p$-value sequence 
    with the $j$-th element replaced by conformal $p$-value $p_j^\cphm (Y_{n+j}; \tT)$.
    The following lemma establishes the super-uniformity of the homogeneous conformal $p$-values 
    and the PRDS property of $\bp^\pshm_{j \to \oc}$.

    \begin{lemma}\label{lem:prdspropertyhomo}
        Under Assumptions~\ref{ass:fullexchangeability} and assuming that both the informative set constructor and the trust score function are permutation-invariant,
        it follows that
        $p_j^\cphm (Y_{n+j}; \tT) \sim \mbox{Unif }(0,1)$ and 
        the sequence $\bp_{j \to \oc}^\pshm$ is PRDS on $\{j\}$.
    \end{lemma}

    We defer its proof to Section~\ref{ssec:proofofauxiliarylemmas}.
    It follows from Lemma~\ref{lem:prdspropertyhomo} that
    \begin{equation*}
        \begin{aligned}
            \fcr = &\, 
            \sum_{j = 1}^m \sum_{k = 1}^m
            \frac{1}{k} \times \pr \big\{
                \kappa(\bp^\pshm) = k, p_j^\pshm \le \alpha k / m, H_{0,j}
            \big\}
            \\ = &\,
            \sum_{j = 1}^m \sum_{k = 1}^m
            \frac{1}{k} \times \pr \big\{
                \kappa(\bp^\pshm_{j \to \oc}) = k, p_j^\cphm (Y_{n+j}; \tT) \le \alpha k / m, H_{0,j}
            \big\}.
        \end{aligned}
    \end{equation*}
    It further follows that
    \begin{equation*}
        \begin{aligned}
            \fcr = &\, 
            \sum_{j = 1}^m \sum_{k = 1}^m
            \frac{1}{k} \times \pr \big\{
                \kappa(\bp^\pshm_{j \to \oc}) = k, p_j^\cphm (Y_{n+j}; \tT) \le \alpha k / m
            \big\} 
            \\ \le &\, 
            \frac{\alpha}{m}
            \sum_{j = 1}^m \sum_{k = 1}^m
            \pr \big\{
                \kappa(\bp^\pshm_{j \to \oc}) = k \bigmid p_j^\cphm (Y_{n+j}; \tT) \le \alpha k / m
            \big\} 
            \le \alpha.
        \end{aligned} \qed
    \end{equation*}

\subsubsection{Reformulation of the modified \infospp}\label{ssec:reformulationinfospp}

In this subsection, we present an equivalent formulation of the homogeneous version of Algorithm~\ref{alg:infospp}, 
which is useful for the analysis.
Let $p_{(1)}^\pshm \le \cdots \le p_{(m)}^\pshm$ be the ordered $p$-values 
and $\{l_j\}_{j = 1}^m$ be one of the possible permutation of $[M]$ such that $p_{l_j}^\pshm = p_{(j)}^\pshm$ for all $j \in [m]$.

Since the trust score is $T_i = 1-\tq_i^+$, it follows that $\tq_{n+l_j}^+ = \tq_{n+(j)}^+$,
where $\tq_{n+(j)}^+$ is obtained by ordering $\{\tq_{n+j}^+, j \in [m]\}$ as
$\tq_{n+(1)}^+ \le \cdots \le \tq_{n+(m)}^+$.
By the definition, \bhfdr procedure selects $\{j \colon p_j^\pshm \le p_{(\hk)}^\pshm\}$
where
$\hk = \max \{k \in [m] \colon p_{(k)}^\pshm \le \alpha k / m\}$.
Let $\htau^+ = \tq_{n+(\hk)}^+$.
Then, the prediction sets reported by modified \infospp is $\{\cC_{n+j}^{\cp, \tq_{n+j}^+}: \tq_{n+j}^+ \le \htau^+ \}$.

Now, we introduce an equivalent form to \bhfdr procedure, 
which is motivated by counting knockoff \citep{weinstein2017power}:
\begin{equation*}
    \htau_{\rm ck}^+ = \max \big\{ 
        \tau \in \{\tq_{n+j}^+,j \in [m]\}_{\varsigma}
        : 
        \hfdp (\tau) \le \alpha
    \big\},
\end{equation*}
where
\begin{equation}\label{eq:estimatefdp}
    \hfdp (\tau) = 
    \frac{
        \sum_{i = 1}^n 
        \bbI (Y_i \notin \cC_i) \cdot 
        \bbI (\tq_i^+ < \tau)
        + 
        \{
            1 + 
            \sum_{i = 1}^n 
            \bbI (Y_i \notin \cC_i) \cdot 
            \bbI (\tq_i^+ = \tau)
        \} 
        \cdot
        U
    }{
        \sum_{j = 1}^m 
        \bbI (\tq_{n+j}^+ \le \tau)
    }
    \times 
    \frac{m}{n+1}.
\end{equation}
It follows from Lemma 2.2 in \cite{mary2022semisupervised} that $\htau_{\rm ck}^+ = \htau^+$.\qed

\subsubsection{Proof of Proposition~\ref{prop:asymptoticdominance}}

We begin with a lemma that characterizes the asymptotic behavior of the threshold $\htau^0_\bh$.
Its proof is deferred to Section~\ref{asssec:proveconvergedthreshold}. 
\begin{lemma}\label{lem:convergedthreshold}
    Suppose that Assumptions~\ref{ass:monotoneinfoconstraints} and~\ref{ass:conditionsforpoweranalysis} hold. 
    Then, under the first condition of Assumption~\ref{ass:smoothnesscondition}, it follows that $\lim\inf \htau_\bh^0 > 0$ a.s.;
    Otherwise, under the second condition of Assumption~\ref{ass:smoothnesscondition}, we have $\htau_\bh^0 \to 0$ a.s..
\end{lemma}

    Recall that the output of modified \infosp and \infospp can be written as
    $\cR^\isp_{\rm \scriptscriptstyle mod} = 
    \{\cC^{\cp, \htau_{\bh}^0}_{n+j} \colon \tq_{n+j}^0 \le \htau_{\bh}^0\}$ and 
    $\cR_{\rm \scriptscriptstyle mod}^{\rm \scriptscriptstyle ISP+}
    =
    \{\cC^{\cp, \tq_{n+j}^+}_{n+j} \colon \tq_{n+j}^+ \le \htau^+\}$.
    Conditional on the event $\{\htau^+ \ge \htau_\bh^0\} $, we have $\tq_i^+ = \max (\tq_i^0, \htau_\bh^0) \ge \htau_\bh^0$, implying
    $\{j\colon \tq_{n+j}^0 \le \htau_{\bh}^0\} \subseteq \{j\colon \tq_{n+j}^+ \le \htau^+\}$.
    Moreover, for any unit reported by \infosp, $ \cC^{\cp, \htau_{\bh}^0}_{n+j} = \cC^{\cp, \tq_{n+j}^+}_{n+j}$.
    Thus, $\cR_{\rm  mod}^\isp \subseteq \cR_{\rm mod}^{\isp+}$ with probability tending to~1 as long as 
    \begin{equation*}
        \lim_{|\cD_0^\ca|,n,m\to\infty} \pr ( \htau^+ \ge \htau_\bh^0 )=1.
    \end{equation*}
    We present a proof for this result as follows. 

    According to Lemma~\ref{lem:convergedthreshold}, it follows that $\htau_\bh^0 \to 0$ a.s. under the second condition of Assumption~\ref{ass:smoothnesscondition}.
    Then, the result holds naturally.
    It remains to focus on the scenario where the first condition of Assumption~\ref{ass:smoothnesscondition} holds:

    By the definition of $\htau^+$ ($\htau_{\rm ck}^+$), 
    it suffices to establish that $\lim_{n,m,|\cD^\ca_0|\to \infty} \pr (\hfdp(\htau_\bh^0) \le \alpha) = 1$,
    where $\hfdp(\cdot)$ is defined in \eqref{eq:estimatefdp}.
    Since $\tq_i^+ = \max (\tq_i^0, \htau_\bh^0) \ge \htau_\bh^0$,
    $\hfdp(\htau_\bh^0)$ can be expressed as 
    \begin{equation*}
        \begin{aligned}
            \hfdp(\htau_\bh^0)
            = & \,
            \frac{
                (n+1)^{-1} \big[
                    1 + 
                    \sum_{i = 1}^n \bbI (\tq_i^+ = \htau_\bh^0) 
                    \bbI \{Y_i \notin \cC^{\cp, \htau_\bh^0 } (X_i; \cD^\ca_0)\} 
                \big]
            }{
                m^{-1} \sum_{j = 1}^m 
                \bbI (\tq_{n+j}^+ = \htau_\bh^0)
            }
            \times U
            \\ \le & \,
            \frac{
                (n+1)^{-1} \big[
                    1 + 
                    \sum_{i = 1}^n 
                    \bbI \{Y_i \notin \cC^{\cp, \htau_\bh^0 } (X_i; \cD^\ca_0)\} 
                \big]
            }{
                m^{-1} \sum_{j = 1}^m 
                \bbI (\tq_{n+j}^+ = \htau_\bh^0)
            }
            \times U.
        \end{aligned}
    \end{equation*}
    It further follows that 
    \begin{equation*}%\label{eq:hfdpbound2}
        \hfdp(\htau_\bh^0) 
        \le 
        \frac{
            (n+1)^{-1} \big[
                1 + 
                \sum_{i = 1}^n 
                \bbI \{Y_i \notin \cC^{\cp, \htau_\bh^0 } (X_i; \cD^\ca_0)\} 
            \big]
        }{
            m^{-1} \sum_{j = 1}^m 
            \bbI (\tq_{n+j}^0 \le \htau_\bh^0)
        }
        \times U
        := 
        A_n \times U.
    \end{equation*}

    The dominator of~$A_n$  is equal to
    \begin{equation}\label{eq:upperbounddominator}
        \begin{aligned}
            \frac{\sum_{i = 1}^{n+m} 
            \bbI (\tq_i^0 \le \htau_\bh^0)}{n+m} 
            - 
            \frac{\sum_{i = 1}^{n+m} 
            \bbI (\tq_i^0 \le \htau_\bh^0)}{n+m} 
            +
            \frac{\sum_{j = 1}^m 
            \bbI (\tq_{n+j}^0 \le \htau_\bh^0)}{m}
            % \\ \ge & \, 
            \ge 
            o_{a.s.}(1) + {\htau_\bh^0}/{\alpha},
        \end{aligned}
    \end{equation}
    where the inequality follows from the definition of $\htau_\bh^0$ and the uniform convergence property of the empirical distribution function:
    \begin{equation*}
        \sup_{t \in \bbR} 
        \bigg| 
            \frac{\sum_{j = 1}^m 
            \bbI (\tq_{n+j}^0 \le t)}{m} 
            - 
            \frac{\sum_{i = 1}^{n+m} 
            \bbI (\tq_i^0 \le t)}{n+m} 
        \bigg|
        \overset{a.s.}{\longrightarrow} 0.
    \end{equation*}

    Then, we focus on the numerator of~$A_n$. 
    First, by the definition of conformal prediction set given in~\eqref{eq:splitconformalpredictionset}, the event $\{Y_i \notin \cC^{\cp, \htau_\bh^0 } (X_i; \cD^\ca_0)\} $ is equivalent to $\{V(X_i, Y_i) > \widehat{Q} \}$, where $\widehat{Q}$ is defined as 
    \begin{equation*}
        \widehat{Q} = 
        \min 
        \bigg\{ 
            \nu \in \{V(X_i, Y_i), i \in \cD^\ca_0\}: 
            1 + 
            \sum_{i \in \cD^\ca_0}
            \bbI \{ V(X_i, Y_i) > \nu \} 
            \le 
            \htau_\bh^0 (|\cD^\ca_0| + 1)
        \bigg\}.
    \end{equation*}
    Consequently, the numerator of~$A_n$ equals 
    \begin{equation}\label{eq:decomposee2}
        \begin{aligned}
            \frac{
                \big[
                    1 + 
                    \sum_{i = 1}^n 
                    \bbI \{ V(X_i, Y_i ) > \widehat{Q}\} 
                \big]
            }{n+1}
            = & \, 
            \frac{
                \big[
                    1 + 
                    \sum_{i \in \cD^\ca_0} 
                    \bbI \{ V(X_i, Y_i ) > \widehat{Q}\} 
                \big]
            }{|\cD^\ca_0| + 1}
            +
            \\ & \hspace{-4cm}
            \frac{
                \big[
                    1 + 
                    \sum_{i = 1}^n 
                    \bbI \{ V(X_i, Y_i ) > \widehat{Q}\} 
                \big]
            }{n+1}
            - 
            \frac{
                \big[
                    1 + 
                    \sum_{i \in \cD^\ca_0} 
                    \bbI \{ V(X_i, Y_i ) > \widehat{Q}\} 
                \big]
            }{|\cD^\ca_0| + 1}.
        \end{aligned}
    \end{equation}
    
    According to the definition of $\widehat{Q}$, the first term of the right hand side in~\eqref{eq:decomposee2} is no greater than $\htau_\bh^0$.
    Meanwhile, using the uniform convergence property of the empirical distribution function, the second line of~\eqref{eq:decomposee2} is of order $o_{a.s.} (1)$.
    Thus, the numerator of~$A_n$ follows that 
    \begin{equation*}
        \frac{
            \big[
                1 + 
                \sum_{i = 1}^n 
                \bbI \{ V(X_i, Y_i ) > \widehat{Q}\} 
            \big]
        }{n+1}
        \le 
        \htau_\bh^0 + o_{a.s.}(1). 
    \end{equation*}
    Then, combining it with~\eqref{eq:upperbounddominator} yields that 
    \begin{equation}\label{eq:liminfxin}
        A_n
        \le 
        \frac{
            \htau_\bh^0 + o_{a.s.} (1)
        }{
            \htau_\bh^0 / \alpha + o_{a.s.} (1)
        }
        = 
        \alpha + o_{a.s.}(1),
    \end{equation}
    as Lemma~\ref{lem:convergedthreshold} suggests that $\lim\inf \htau_\bh^0 > 0$ under Assumption~\ref{ass:smoothnesscondition}(a).  Finally, 
    \begin{equation*}
        \begin{aligned}
            1 \ge 
            \lim_{n\to \infty} \pr \big\{ \hfdp(\htau_\bh^0) \le \alpha \big\} 
            \ge &\,  
            \lim_{n\to \infty} \pr \big( 
                U \le \alpha / A_n
            \big)
            \overset{(i)}{=} 
            \lim_{n\to \infty}
            \E \big[ (\alpha / A_n) \wedge 1 \big]
            \\ \overset{(ii)}{\ge} & \,
            \E \Big[
                \underset{n \to \infty}{\lim \inf} \, 
                \{(\alpha / A_n) \wedge 1\}
            \Big]
            = 
            \E \Big[
                1 \wedge \underset{n \to \infty}{\lim \inf} \, 
                (\alpha / A_n)
            \Big]
            \overset{(iii)}{=} 1,
        \end{aligned}
    \end{equation*}
 where step (i) follows because $U$ is an independent uniform random variable, step (ii) uses Fatou's lemma, and step (iii) follows from the bound in \eqref{eq:liminfxin}. \qed

\subsection{Proof of Theorem~\ref{the:maxpower}}

    For any trust score $T: \cX \to \bbR$, we denote $S(x) = 1-T(x)$ and 
    the selection set should be of the form $\{j: S(X_{n+j}) \le \tau\}$ for some $\tau \in \bbR$.
    Similarly, we have $S^*(x) = 1-T^\ora (x,\cC^\cI (x)) = \pr (Y \notin \cC^\cI (X) \mid X = x)$,
    $S' (x) = 1-T'(x)$, $\tau_\alpha^* = 1-t^*(\alpha)$, and $\tau' = 1-t'$.

 As the $\mfcr$ and $\cpow$  are measurable with respect to the score function $S$ and the threshold $\tau$, we have    \begin{equation*}
        \begin{aligned}
            % \label{al:mfcrdef}
            \mfcr (S,\tau) = & \,
            \frac{
                \E \big[
                    \bbI \big\{ S(X) \le \tau, Y \notin \cC^\cI (X) \big\}
                \big]
            }{
                \E \big[
                    \bbI \big\{ S(X) \le \tau\}
                \big]
            }
            = 
            \frac{
                \int \bbI \big\{ S(x) \le \tau\} S^*(x) \rmd \bP_X(x)
            }{
                \int \bbI \big\{ S(x) \le \tau\}  \rmd \bP_X(x)
            },
            \\ 
            % \label{al:cpowdef}
            \cpow (S,\tau) = & \,
            m \times 
            \E \big[
                \bbI \big\{ S(X) \le \tau, Y \in \cC^\cI (X) \big\}
            \big]
            = 
            \int \bbI \big\{ S(x) \le \tau\} \{1-S^*(x)\} \rmd \bP_X(x),
        \end{aligned}
    \end{equation*}
    where $\bP_X$ is the probability measurement for random variable $X$.
    Note that $\mfcr (S^*,\tau_\alpha^*) = \alpha$ and $\mfcr (S',\tau') \le \alpha$, 
   we have    \begin{align}
        \label{al:expressoptimalmfcr}
        &\, \int \bbI \big\{ S^*(x) \le \tau_\alpha^* \big\} 
        \big\{ S^*(x) - \alpha \big\} \rmd \bP_X(x) 
        = 0,
        \\ 
        \label{al:expressaltermfcr}
        &\, \int \bbI \big\{ S'(x) \le \tau' \big\} 
        \big\{ S^*(x) - \alpha \big\} \rmd \bP_X(x) 
        \le 0.
    \end{align}

    The first equality \eqref{al:expressoptimalmfcr} implies that $\tau_\alpha^* \ge \alpha$ almost surely under $\bP_X$.
    
    \textbf{Case I: $\tau_\alpha^* = \alpha$.}
    If $\tau_\alpha^* = \alpha$, it follows that $\bP_X(S^*(x) \ge \alpha) = 1$.
    Thus, the inequality \eqref{al:expressaltermfcr} implies that $\bP_X(S'(X) \le \tau', S^*(X) > \alpha) = 0$.
    Then, by the definition of $\cpow$, we derive that
    \begin{equation}\label{eq:morepowercase1}
        \begin{aligned}
            \cpow(S^*, \tau_\alpha^*) - \cpow(S', \tau') 
            = &\,
            \int \Big[
                \bbI \big\{ S^*(x) \le \tau_\alpha^* \big\} 
                -
                \bbI \big\{ S'(x) \le \tau' \big\} 
            \Big]
            \big\{1-S^*(x)\big\}
            \rmd \bP_X(x)
            \\ \ge &\, 
            - \int_{S'(x) \le \tau',S^*(x) > \tau_\alpha^*}
            \big\{1-S^*(x)\big\}
            \rmd \bP_X(x)
            =0.
        \end{aligned}
    \end{equation}

    \textbf{Case II: $\tau_\alpha^* > \alpha$.}
    Subtracting \eqref{al:expressaltermfcr} from \eqref{al:expressoptimalmfcr} yields that
    \begin{equation*}
        \begin{aligned}
            0 \le &\, \int \Big[
                \bbI \big\{ S^*(x) \le \tau_\alpha^* \big\} 
                -
                \bbI \big\{ S'(x) \le \tau' \big\} 
            \Big]
            \big\{S^*(x) - \alpha \big\}
            \rmd \bP_X(x)
            \\ = &\,
            \int_{S^*(x) \le \tau_\alpha^*, S'(x) > \tau'}
            \big\{S^*(x) - \alpha \big\}
            \rmd \bP_X(x)
            -
            \int_{S^*(x) > \tau_\alpha^*, S'(x) \le \tau'}
            \big\{S^*(x) - \alpha \big\}
            \rmd \bP_X(x).
        \end{aligned}
    \end{equation*}
    Provided that $(u-\alpha)/(1-u)$ is increasing with respect to $u$,
    it follows that
    \begin{equation*}
        \frac{S^*(x)-a}{1-S^*(x)}
        \begin{cases}
            \le \frac{\tau_\alpha^* - \alpha}{1 - \tau_\alpha^*}, & S^*(x) \le \tau^*_\alpha,
            \\ 
            > \frac{\tau_\alpha^* - \alpha}{1 - \tau_\alpha^*}, & S^*(x) > \tau^*_\alpha,
        \end{cases}
    \end{equation*}
    which further implies that (having accounted for those points where $S^*(x) = 1$):
    \begin{equation*}
        \begin{aligned}
            0 \le &\,
            \int_{S^*(x) \le \tau_\alpha^*, S'(x) > \tau'}
            \big\{S^*(x) - \alpha \big\}
            \rmd \bP_X(x)
            -
            \int_{S^*(x) > \tau_\alpha^*, S'(x) \le \tau'}
            \big\{S^*(x) - \alpha \big\}
            \rmd \bP_X(x) 
            \\ \le & \,
            \frac{\tau_\alpha^* - \alpha}{1 - \tau_\alpha^*}
            \bigg[ 
                \int_{S^*(x) \le \tau_\alpha^*, S'(x) > \tau'}
                \big\{1-S^*(x) \big\}
                \rmd \bP_X(x)
                -
                \int_{S^*(x) > \tau_\alpha^*, S'(x) \le \tau'}
                \big\{1-S^*(x) \big\}
                \rmd \bP_X(x) 
            \bigg]
            \\ = &\, 
            \cpow(S^*, \tau_\alpha^*) - \cpow(S', \tau').
        \end{aligned}
    \end{equation*}
   The desired result follows by combining the above with \eqref{eq:morepowercase1}. \qed

\subsection{Proof of Proposition~\ref{prop:finitesamplefcrpulearning}}

Let $\hk = |\cS|$ be the number of selected units.
    Then, FCR can be decomposed as 
    $\fcr = \sum_{j = 1}^m \E \{\bbI (p_j^\ps \le \alpha \hk/ m , H_{0,j}) / \hk\}$,
    where the expectation follows that
    \begin{equation}\label{eq:proposition6decomp1}
        \begin{aligned}
            \E \bigg\{ 
                \frac{
                    \bbI (p_j^\ps \le \alpha \hk/ m , H_{0,j})
                }{
                    \hk
                }
            \bigg\}
            = & \,
            \sum_{k = 1}^{m}
            \frac{1}{k} \times
            \pr \big\{ \hk = k, p_j^\ps \le \alpha k / m, H_{0,j} \big\}
            \\ = & \,
            \sum_{s = 0}^n \sum_{k = 1}^m
            \pr \big\{ \hk = k, p_j^\ps \le \alpha k / m, H_{0,j}, s_L = s \big\}
            \\ = & \,
            \sum_{s = 0}^n \sum_{k = 1}^m
            \pr \big\{ \hk = k \mid p_j^\ps \le \alpha k / m, H_{0,j}, s_L = s \big\} 
            \times 
            \\ & \, \quad \quad \quad \quad
            \pr \big\{ p_j^\ps \le \alpha k / m \mid  H_{0,j}, s_L = s \big\}
            \times 
            \pr \big\{ H_{0,j}, s_L = s \big\},
        \end{aligned}
    \end{equation}
    where the notation $s_L$ follows the definition in the proof of the third result of Lemma~\ref{lem:prdspvalue} (Section~\ref{ssec:prooftheprdspvalue}).
    Results in \eqref{eq:shrinkagesuperuniform} establish that 
    \begin{equation*}
        \pr \big\{ p_j^\ps \le \alpha k / m \mid  H_{0,j}, s_L = s \big\}
        \le 
        \frac{\alpha k }{m}
        \cdot 
        \frac{n+1}{n-s+1}.
    \end{equation*}
    In addition, provided that the random variables $\bbI\{Y_i \notin \cC(X_i)\}, i \in [n+m]$ are  
    independent Bernoulli distribution $B(1,\pi_0^*)$,
    it follows that 
    \begin{equation*}
        \pr \big\{ H_{0,j}, s_L = s \big\}
        = 
        \pi_0^* 
        \times 
        \binom{n}{s} \times
        (1-\pi_0^*)^s (\pi_0^*)^{n-s}
        = 
        \binom{n}{s} \times
        (1-\pi_0^*)^s (\pi_0^*)^{n-s+1}.
    \end{equation*}
    Moreover, the third PRDS result in Lemma~\ref{lem:prdspvalue} implies that
    \begin{equation*}
        \sum_{k = 1}^m 
        \pr \big\{ \hk = k \mid p_j^\ps \le \alpha k / m, H_{0,j}, s_L = s \big\} 
        \le 1.
    \end{equation*}
    Combining these three inequalities with \eqref{eq:proposition6decomp1} yields that
    \begin{equation*}
        \E \bigg\{ 
            \frac{
                \bbI (p_j^\ps \le \alpha \hk/ m , H_{0,j})
            }{
                \hk
            }
        \bigg\}
        \le 
        \frac{\alpha}{m}
        \sum_{s = 0}^n 
        \frac{n+1}{n-s+1} \times 
        \binom{n}{s} \times
        (1-\pi_0^*)^s (\pi_0^*)^{n-s+1}
        \le 
        \frac{\alpha}{m}.
    \end{equation*}
 Hence $\fcr = \sum_{j = 1}^m \E \{\bbI (p_j^\ps \le \alpha \hk/ m , H_{0,j}) / \hk\} \le \alpha$, completing the proof. \qed

\subsection{Proof of Auxiliary Lemmas}\label{ssec:proofofauxiliarylemmas}

\subsubsection{Proof of Lemma~\ref{lem:anticonservativepseudopvalue}}\label{asssec:proofofanticonservativepseudopvalue}

    Conditional on $H_{1,j} = \{ Y_{n+j} \in \cC(X_{n+j}) \}$, the oracle conformal p-value can be expressed as
    \begin{equation*}
        p_j^\cp (Y_{n+j}; \tT) \bigmid H_{1,j}
        = 
        \frac{
            \sum_{i = 1}^n \bbI (Y_i \notin \cC_i)
            + 
            U_j \{1 + \sum_{i = 1}^n \bbI (Y_i \in \cC_i) \}
        }{
            n+1
        }.
    \end{equation*}
    Consequently, 
    \begin{equation*}
        \begin{aligned}
            \pr \{H_{1,j}, p_j^\cp(Y_{n+j}; \tT) \le \alpha\}
            = 
            \pr \Big\{ 
                (1-U_j) \sum_{i = 1}^n \bbI (Y_i \notin \cC_i)
                \le 
                (\alpha - U_j) (n+1)
            \Big\}.
        \end{aligned}
    \end{equation*}
    
    As the count $\sum_{i = 1}^n \bbI (Y_i \notin \cC_i)$ is no less than $0$, 
    it suffices to analysis conditionally on $U_j \le \alpha$:
    \begin{equation*}%\label{eq:theorem2decomp1}
        \begin{aligned}
            &\, 
            \pr \{H_{1,j}, p_j^\cp(Y_{n+j}; \tT) \le \alpha\}
            \\ = &\,
            \pr (U_j \le \alpha)
            \times 
            \pr \Big\{ 
                (1-U_j) \sum_{i = 1}^n \bbI (Y_i \notin \cC_i)
                \le 
                (\alpha - U_j) (n+1)
                \bigmid 
                U_j \le \alpha
            \Big\}.
        \end{aligned}
    \end{equation*}
    Given that $U_j$ is an independent uniform variable, it follows that
    \begin{equation*}%\label{eq:theorem2decomp1}
        \begin{aligned}
            % &\, 
            \pr \{H_{1,j}, p_j^\cp(Y_{n+j}; \tT) \le \alpha\}
            = &\,
            \alpha
            \times 
            \pr \Big\{ 
                \sum_{i = 1}^n \bbI (Y_i \notin \cC_i)
                \le 
                \frac{(\alpha - U_j) (n+1)}{(1-U_j)} 
                \bigmid
                U_j \le \alpha
            \Big\}
            \\ \le &\,
            \alpha \times 
            \pr \Big\{ 
                \sum_{i = 1}^n \bbI (Y_i \notin \cC_i)
                \le 
                \alpha (n+1)
            \Big\}.
        \end{aligned}
    \end{equation*}

    Since $n \pi_0^* \ge \alpha (n+1)$,
    we have 
    \begin{equation*}
        \begin{aligned}
            \pr \Big\{ 
                \sum_{i = 1}^n \bbI (Y_i \notin \cC_i)
                \le 
                \alpha (n+1)
            \Big\}
            = &\,
            \pr \Big\{ 
                \sum_{i = 1}^n \bbI (Y_i \notin \cC_i)
                - 
                n \pi_0^*
                \le 
                \alpha (n+1)
                - 
                n \pi_0^*
            \Big\}
            \\ = &\,
            \pr \Big\{ 
                n \pi_0^*
                -
                \sum_{i = 1}^n \bbI (Y_i \notin \cC_i)
                \ge 
                n \pi_0^*
                -
                \alpha (n+1)
            \Big\}
            \\ \le & \, 
            \exp \{ -2n (\bar{\pi}_0^* - \alpha)^2 \},
        \end{aligned}
    \end{equation*}
    where the last step is due to applying the Hoeffding inequality \citep[Theorem 2.8]{boucheron2013concentration} to the \iid Bernoulli variables $\bbI\{Y_i \notin \cC(X_i)\}, i \in [n]$.

\subsubsection{Proof of Lemma~\ref{lem:prdspropertyhomo}}

    Similar to \eqref{eq:algebratildewj}, we define the algebra
    \begin{equation*}
        \tW_j^\homo = 
        \Big\{ 
            \{      
                \tT_1,\cdots,\tT_n,\tT_{n+j}
            \}_{\varsigma},
            \{
                T_{n+j'}, j' \in [m], j' \ne j
            \}
        \Big\}.
    \end{equation*}
    Due to the exchangeability, $\tT_{n+j}$ uniformly distributes over the unorded set
    $\{ \tT_1,\cdots,\tT_n,\tT_{n+j} \}_{\varsigma}$
    conditionally on $\tW_j^\homo$.
    Thus, we have $p_j^\cphm (Y_{n+j}; \tT) \sim \mbox{Unif }(0,1) $ and $p_j^\cphm (Y_{n+j}; \tT) \perp \tW_j^\homo$.

    Provided that the sequence $\bp^\pshm_{j \to \oc}$ is measurable with respect to $\{p_j^\cp(Y_{n+j};\tT), 
    W_j, U\}$ and $U$ is measurable with respect to $\{W_j,p_j^\cphm (Y_{n+j}; \tT)\}$,
    it follows that $\bp^\pshm_{j \to \oc}$ is measurable with respect to $\{p_j^\cp(Y_{n+j};\tT), 
    W_j\}$.
    Then, following the arguments in the proofs of fourth result in Proposition~\ref{prop:keypropertyconformalpvalue},
    we yields that
    $\pr \big\{ \bp^\pshm_{j \to \oc} \in D \mid p_j^\cphm (Y_{n+j}; \tT) = t, \tW_j \big\} $
    is non-decreasing with respect to $t$ for any increasing set $D$.
    Consequently, the proof of fifth result in Proposition~\ref{prop:keypropertyconformalpvalue}
    implies that
    $\pr \big\{ \bp^\pshm_{j \to \oc} \in D \mid p_j^\cphm (Y_{n+j}; \tT) = t \big\} $ is 
    non-decreasing with respect to $t$.

\subsubsection{Proof of Lemma~\ref{lem:convergedthreshold}}
\label{asssec:proveconvergedthreshold}

To begin with, we briefly introduce an alternative expression for $\htau_\bh^0$:
Given the definition of $\nu(x)$, we have $\tq(x; \cD^\ca_0) = (|\cD^\ca_0| + 1)^{-1} \big[ 1+ \sum_{i \in \cD^\ca_0} \bbI \{V(X_i, Y_i) \ge \nu(x)\} \big]$.
Then, for any $\tau \in [0,1]$, there exists $\nu_\tau$ such that $\{\tq(x; \cD^\ca_0) \le \tau \} = \{ \nu(x) > \nu_\tau \}$ and $\tau = (|\cD^\ca_0| + 1)^{-1} \big[ 1+ \sum_{i \in \cD^\ca_0} \bbI \{V(X_i, Y_i) > \nu_\tau \} \big]$.
With the above notation, we define 
\begin{equation*}
    \begin{aligned}
        \hnu_\bh^0 = 
        \min \bigg\{ &
            \nu \in \{V(X_i, Y_i), i \in \cD^\ca_0 \cup \cD^\ca \cup \cD^\te\}
            :
            \\ & \quad
            \frac{\sum_{i = 1}^{n+m} \bbI\{ \tnu(X_i) > \nu \} }{n+m}
            \ge 
            \frac{\sum_{i \in \cD^\ca_0 }  \bbI\{V(X_i, Y_i) > \nu\}  + 1}{|\cD^\ca_0| + 1} / \alpha
        \bigg\}.
    \end{aligned}
\end{equation*}
It can be verified that 
\begin{equation}\label{eq:equivalenceformtaubh}
    \htau_\bh^0 
    = 
    \frac{\sum_{i \in \cD^\ca_0 }  \bbI\{V(X_i, Y_i) > \hnu_\bh^0 \}  + 1}{|\cD^\ca_0| + 1}.
\end{equation}

Now we turn back to the proof of Lemma~\ref{lem:convergedthreshold}.
    First, we introduce the uniformly convergence property of empirical distributed function, which is given by 
    \begin{equation}\label{eq:uniformcovergence}
        \begin{aligned}
            \sup_{t \in \bbR} 
            \bigg|
                \frac{1}{n+m} \sum_{i = 1}^{n+m}
                \bbI \big\{ \nu(X_i) > t \big\}
                - 
                \pr \big\{ \nu(X_i) > t \big\}
            \bigg| 
            \overset{a.s.}{\longrightarrow}
            0,
            \\ 
            \sup_{t \in \bbR} 
            \bigg|
                \frac{1}{|\cD^\ca_0| + 1} 
                \bigg[
                    \sum_{i \in \cD^\ca_0 }
                    \bbI \big\{ V(X_i, Y_i) > t \big\}+ 1
                \bigg]
                - 
                \pr \big\{ V(X_i, Y_i) > t \big\}
            \bigg| 
            \overset{a.s.}{\longrightarrow}
            0.
        \end{aligned}
    \end{equation}
    
    \noindent (a) \textbf{When the first condition of Assumption~\ref{ass:smoothnesscondition} holds}, the uniform convergence in~\eqref{eq:uniformcovergence} implies that, there exists $N > 0$ such that $
    \forall \, n,m,|\cD^\ca_0| \ge N,$, it follows that 
    \begin{equation*} 
        \frac{
            \sum_{i = 1}^{n+m}
            \bbI \big\{ \nu(X_i) > v^* \big\}
        }{n+m} 
        \ge 
        \frac{
            \sum_{i \in \cD^\ca_0 }
            \bbI \big\{ V(X_i, Y_i) > v^* \big\}+ 1
        }{|\cD^\ca_0| + 1} 
        / \alpha
    \end{equation*}
    almost surely.
    Thus, we have $\lim\sup \hnu_\bh^0 \le \nu^*$, which implies that 
    \begin{equation*}
        \begin{aligned}
            \underset{n,m,|\cD^\ca_0| \to \infty}{\lim \inf} 
            \htau_\bh^0 
            = & \, 
            \underset{n,m,|\cD^\ca_0| \to \infty}{\lim \inf}  
            \frac{\sum_{i \in \cD^\ca_0 }  \bbI\{V(X_i, Y_i) > \hnu_\bh^0 \}  + 1}{|\cD^\ca_0| + 1}
            \\ \ge & \,
            \underset{n,m,|\cD^\ca_0| \to \infty}{\lim \inf}  
            \frac{\sum_{i \in \cD^\ca_0 }  \bbI\{V(X_i, Y_i) > \nu^* \}  + 1}{|\cD^\ca_0| + 1}
            \\ = & \, 
            \pr \{V(X_i, Y_i) > v^*\}
            > 0
        \end{aligned}
    \end{equation*}
    holds almost surely. 
    Here, the first equality is due to the equivalence shown in~\eqref{eq:equivalenceformtaubh}.

    \noindent (b) \textbf{When the second condition of Assumption~\ref{ass:smoothnesscondition} holds}, the uniform convergence in~\eqref{eq:uniformcovergence} implies that, there exists $N > 0$ such that when $\, n,m,|\cD^\ca_0| \ge N$, it follows that 
    \begin{equation}\label{eq:forconvergetozero} 
        \frac{
            \sum_{i = 1}^{n+m}
            \bbI \big\{ \nu(X_i) > v^* \big\}
        }{n+m} 
        \le 
        \frac{
            \sum_{i \in \cD^\ca_0 }
            \bbI \big\{ V(X_i, Y_i) > v^* \big\}+ 1
        }{|\cD^\ca_0| + 1} 
        / \alpha
        - 
        \frac{\epsilon}{2}
    \end{equation}
    for any $\nu \in  \big\{\nu: \pr \{V(X,Y) \ge \nu\} > 0 \big\}$ almost surely.
    
    Let $\bar{\nu} = \sup \big\{\nu: \pr \{V(X,Y) \ge \nu\} > 0 \big\}$.
    Then, inequality~\eqref{eq:forconvergetozero} implies that 
    $\hnu_\bh^0 > \bar{\nu}$ when $\, n,m,|\cD^\ca_0| \ge N$ almost surely.
    Consequently, we have 
    \begin{equation*}
        \begin{aligned}
            \underset{n,m,|\cD^\ca_0| \to \infty}{\lim \sup} 
            \htau_\bh^0 
            = & \, 
            \underset{n,m,|\cD^\ca_0| \to \infty}{\lim \sup} 
            \frac{\sum_{i \in \cD^\ca_0 }  \bbI\{V(X_i, Y_i) > \hnu_\bh^0 \}  + 1}{|\cD^\ca_0| + 1}
            \\ \le & \, 
            \underset{n,m,|\cD^\ca_0| \to \infty}{\lim \sup} 
            \frac{\sum_{i \in \cD^\ca_0 }  \bbI\{V(X_i, Y_i) > \bar{\nu} \}  + 1}{|\cD^\ca_0| + 1} = 0
        \end{aligned}
    \end{equation*}
    almost surely, completing the proof. \qed

\section{Additional Details about Numerical Results}\label{asec:additionalnumericalresults}

\subsection{Implementation Details of the Illustrative Example}
\label{assec:detailillustrativeexample}

In this section, we provide implementation details and additional results for the illustrative example in Section~\ref{sec:illustrativeexample}.

Specifically, we consider an image classification task on the CIFAR-10 dataset with the goal of distinguishing among three classes: \emph{Vehicle} (`1'), \emph{Aircraft} (`2'), and \emph{Bird} (`3').
Restricting to these three classes yields $n_{\text{tot}}=15{,}000$ images, of which we randomly select $20\%$ for training a lightweight convolutional neural network. 
The softmax scores are denoted $\hmu_k(x) = \hpr(Y = k \mid X = x)$ for $k \in [K]$.
We randomly sample $n = 6000$ images for calibration and $m = 30$ images for testing.
Throughout this example, the target error rate is $\alpha = 0.1$.

For implementing the standard split conformal prediction method (cf. \citealp{angelopoulos2024theoretical}), let $V(x,y) = 1 - \hmu_y (x)$ be the non-conformity score and $V_i = V(X_i, Y_i)$ for $ i \in [n]$.
Then, the conformal prediction set for the $j$-th test unit is given by 
\begin{equation*}
    \cC_{n+j}^{\cp}
    := 
    \cC^{\cp, \alpha} (X_{n+j}; \{(X_i, Y_i)\}_{i = 1}^{n}, V)
    =\bigg\{
        y \in [K]: 
        \frac{
            1 + \sum_{i = 1}^n 
            \bbI \big\{ 
                V_i \ge V(X_{n+j},y) 
            \big\}
        }{n+1} \ge \alpha
    \bigg\}.
\end{equation*}

The \texttt{JOMI} procedure \citep{jin2023selection} can be used to correct selection bias using the \emph{reference set}:
Let $\cS = \{j: |\cC_{n+j}^\cp| \le 2\}$ be index set of units selected by naive strategy.
For any selected unit $j \in \cS$ and hypothesized value $y \in [3]$, the reference set collects the calibration units $i \in [n]$ that are selected after swapping with the hypothesized test point $(X_{n+j},y)$. Formally, we define
\begin{equation*}
    \cD^{\rm \scriptscriptstyle Ref} (j ,y)=
    \Big\{i \in n:
        \big| \cC^{\cp, \alpha}(X_i; \{(X_{i'},Y_{i'})\}_{i' \ne i} \cup \{(X_{n+j}, y)\}, V) \big| \le 2 
    \Big\}.
\end{equation*}
Then, the adjusted prediction set for $j$-th test unit is given by
\begin{equation*}
    \cC^{\jomi}_{n+j}
    = 
    \bigg\{ 
        y \in [K]:
        \frac{
            \sum_{i \in \cD^{\rm \scriptscriptstyle Ref}(j , y)}
            \bbI \big\{ V_i \ge V(X_{n+j},y )\big\}
        }{
            |\cD^{\rm \scriptscriptstyle Ref}(j , y)| + 1
        }
        \ge \alpha
    \bigg\}.
\end{equation*}

The \infosp procedure \citep{gazin2025selecting} first computes $\cI$-adjusted $p$-value for each test unit, denoted by 
\begin{equation*}
    \begin{aligned}
        \tq_{n+j} 
        = & \, 
        \tq (X_{n+j}; \{(X_i,Y_i) \}_{i = 1}^{n}) 
        \\ = & \, 
        \inf \big\{
            q\in(0,1] : |\cC^{\cp, q} (X_{n+j}; \{(X_i, Y_i)\}_{i = 1}^{n}, V)| \le 2
        \big\},
    \end{aligned}
    \qquad 
    j \in [m].
\end{equation*}
We then apply $\alpha$-level \bhfdr procedure to the sequence of $\cI$-adjusted $p$-values to derive a threshold $\htau_{\bh}$.
The prediction sets reported by \infosp are given by $\big\{ \cC^{\cp, \htau_\bh} (X_{n+j}; \{(X_i, Y_i)\}_{i = 1}^{n}, V) : j \in [m], \tq_{n+j} \le \htau_{\bh}\big\}$.

Finally, we repeatedly resample the calibration and test sets and report the averages over 100 repetitions in Table~\ref{tab:cifar10compareparadigm}.
The results are consistent with those shown in the introduction: the naive selection meets the informativeness constraint but yields an inflated error rate; \infosp attains both informativeness and trustworthiness, yet exhibits lower statistical power than \infospp.

\begin{table}[!htbp]
    \centering
    \caption{
        Comparison of naive selection, \infosp, and \infospp.
        Average of two metrics over 100 iterations are reported.
        \textbf{cPOW} (counting power) is the number of reported prediction sets.
    }
    \setlength{\tabcolsep}{7pt}
    \renewcommand{\arraystretch}{1.15}
    \small
    
    \begin{tabular}{lcc}
        \toprule
        Method & FCR (\%) & cPOW  \\
        \midrule
    
        \texttt{Naive Selection}
        & 11.77 & 23.16 \\
    
        \infosp
        & 9.23 & 16.42 \\
    
        \infospp
        & 9.54 & 22.86 \\
    
        \bottomrule
    \end{tabular}

    \label{tab:cifar10compareparadigm}
\end{table}

\subsection{Algorithm descriptions}

We summarize \scip and \infospp in Algorithms~\ref{alg:baseprocedure} and~\ref{alg:infospp}, respectively.

\begin{algorithm}[!ht]
	%\textsl{}\setstretch{1.2}
	\renewcommand{\algorithmicrequire}{\textbf{Input:}}
	\renewcommand{\algorithmicensure}{\textbf{Output:}}
	\caption{\scip: selective conformal inference with informative predictions}
	\label{alg:baseprocedure}
	\begin{algorithmic}[1]
        \REQUIRE 
        Calibration data $\cD^\ca = \{(X_i,Y_i), i \in [n]\}$;
        Test data $\cD^\te = \{X_{n+j}, j \in [m]\}$;
        Informative set constructor $\cC^\cI (\cdot)$;
        Trust score $T(\cdot,\cC^\cI (\cdot))$;
        Target FCR level~$\alpha$.
        \STATE
        Construct informative sets $\cC_i^\cI = \cC^\cI (X_i)$ for $i \in [n+m]$;
        \STATE 
        Calculate $p$-values $p_j^\ps$ as in \eqref{eq:pseudopvalue} for $j \in [m]$;
        \STATE 
        Calculate threshold $\hat{\alpha}$ by applying \bhfdr procedure on $\{p_j^\ps, j \in [m]\}$ as in \eqref{eq:qbh};
        \ENSURE
        The reported sets
        $\cR = \{\cC^\cI_{n+j}\colon j \in [m], p_j^\ps \le \hat{\alpha}\}$.
	\end{algorithmic}  
\end{algorithm}

\begin{algorithm}[!ht]
	%\textsl{}\setstretch{1.2}
	\renewcommand{\algorithmicrequire}{\textbf{Input:}}
	\renewcommand{\algorithmicensure}{\textbf{Output:}}
	\caption{\infospp: An instance of \scip based on \infosp.}
	\label{alg:infospp}
	\begin{algorithmic}[1]
        \REQUIRE 
        Non-conformity score $V(\cdot, \cdot)$
        % Conformal prediction set function $\cC^{\cp,q}(\cdot;\cdot)$;
        Calibration data $\cD^\ca$;
        Test data $\cD^\te$;
        Additional calibration data $\cD^\ca_0$;
        Target FCR level $\alpha$.
        \STATE 
        Calculate the $\cI$-adjusted p-values $\tq_i^0 = \tq(X_i; \cD^\ca_0)$ for $i \in [n+m]$;
        \STATE 
        Calculate the threshold $\htau_\bh^0$ by applying $\alpha$-level \bhfdr procedure to $\{\tq_i^0, i \in [n+m]\}$;
        \STATE 
        Calculate the truncated $\cI$-adjusted p-values $\tq_i^+ = \max \{\htau_\bh^0, \tq_i^0\}$ for $i \in [n+m]$;
        \STATE 
        Construct the informative sets $\cC_i^\cI = \cC^{\cp,\tq_i^+} (X_i; \cD^\ca_0, V)$ for $i \in [n+m]$;
        \STATE 
        Calculate the trust scores $T_i = 1-\tq_i^+$ for $i \in [n+m]$;
        \STATE 
        Calculate conformal p-values $p_j^\ps$ as in \eqref{eq:pseudopvalue}, for $j \in [m]$;
        \STATE 
        Calculate $\hat{\alpha}$ by applying $\alpha$-level \bhfdr procedure to $\{p_j^\ps, j \in [m]\}$ as in \eqref{eq:qbh};
        \ENSURE
        The reported informative conformal prediction sets
        $\{\cC_{n+j}^\cI \colon p_j^\ps \le \hat{\alpha}\}$.
	\end{algorithmic}  
\end{algorithm}

% \end{appendices}

\end{document}